\documentclass[12pt, reqno]{amsart}%
\usepackage[english]{babel}
\usepackage{amsthm}
\usepackage{amsmath}
\usepackage{amssymb}
\usepackage{dsfont}
\usepackage{graphicx}
\usepackage{esint}
\usepackage[shortlabels]{enumitem}
\usepackage{amscd,mdwlist,relsize}
\usepackage[margin=1.2in]{geometry}
\usepackage{color}
\usepackage{cite}
\usepackage{bbm, mathrsfs,pgf,tikz}
\usepackage{tcolorbox}
\usepackage[backgroundcolor=white, bordercolor=blue,
linecolor=blue]{todonotes}
\usepackage{url}
\usepackage{dsfont}
\usepackage{cases}
\usetikzlibrary{arrows}

\usepackage[pagebackref=false, pdfpagelabels, plainpages=false]{hyperref} 
\hypersetup{
colorlinks=true,
linkcolor=red,
filecolor=magenta,
urlcolor=cyan,
citecolor=blue,
}

\theoremstyle{plain}  
\newtheorem{theorem}{Theorem}[section]
\newtheorem{lemma}[theorem]{Lemma}

\newtheorem{corollary}[theorem]{Corollary}

\theoremstyle{definition}  
\newtheorem{definition}[theorem]{Definition}
\newtheorem{example}[theorem]{Example}

\newtheorem{remark}[theorem]{Remark}   
\theoremstyle{remark}
\newtheorem*{remark*}{Remark}

\numberwithin{equation}{section}

\DeclareMathOperator{\dist}{dist}
\DeclareMathOperator{\supp}{supp}

\DeclareMathOperator{\loc}{loc}

\DeclareMathOperator{\pv}{\operatorname{p.\!v.}}

\DeclareMathOperator{\cE}{\mathcal{E}}

\DeclareMathOperator{\cJ}{\mathcal{J}}

\DeclareMathOperator{\R}{\mathbb{R}}

\renewcommand{\d}{\,\mathrm{d}}

\newcommand{\eps}{\varepsilon}

\newcommand{\vertiii}[1]{{\left\vert\kern-0.25ex\left\vert\kern-0.25ex\left\vert #1 \right\vert\kern-0.25ex\right\vert\kern-0.25ex\right\vert}}
\addto\extrasenglish{%

}
\makeatletter
\@namedef{subjclassname@2020}{\textup{2020} Mathematics Subject Classification}
\makeatother

\begin{document}

\title{Robust interpolation inequalities via Chebyshev-type integral inequalities}

\author{Guy Foghem}
\address{{\tiny Brandenburgische Technische Universit\"at Cottbus--Senftenberg, Fakult\"{a}t 1: MINT Fachgebiet Mathematik, Platz der Deutschen Einheit 1, 03046 Cottbus, Germany.} \href{https://orcid.org/0000-0002-8917-7309}{ORCID}}
\email{guy.foghem[at]b-tu.de}
\thanks{Financial support from the Deutsche Forschungsgemeinschaft (DFG) through the Walter Benjamin Programme (project FO~1699/1-1) is gratefully acknowledged.}

\begin{abstract}
We establish  robust log-convex interpolation inequalities within the scale of Gagliardo seminorms. We achieve this by  deriving some Chebyshev-type integral inequalities for general non-synchronous functions. Our primary motivation for establishing these robust interpolation inequalities stems from the study of the asymptotic nonlocal-to-local stability of weak solutions to the boundary  Dirichlet problem associated with the regional fractional $p$-Laplacian. More precisely, if $u_s \in W^{s,p}(\Omega)$ weakly satisfies $(-\Delta)_{p, \Omega}^s u_s = f_s $ in $\Omega$ and $ \gamma^s_0(u_s) = g_s$ on  $\partial\Omega,$  with $\frac{1}{p} < s \leq 1$  and  $\Omega \subset \R^d$ is bounded Lipschitz, then, under appropriate convergence of the data $f_s$ and $g_s$ as $s \to 1^-$, we establish that  $\| u_s - u_1 \|_{W^{\eta,p}(\Omega)} \xrightarrow{s \to 1^-} 0 $ for all $0 \leq \eta < 1$.
\end{abstract}


\maketitle


\tableofcontents

\section{Introduction}\label{sec:introduction}
Given an open set $\Omega\subset\R^d$, the Gagliardo seminorm of order $s\in [0,1]$ is defined
by
\begin{align*}
|u|^p_{W^{s,p}(\Omega)} :=
\begin{cases}
\|u\|_{L^p(\Omega)}^p, & \text{if } s = 0, \\
\displaystyle  \iint_{\Omega\times\Omega}
\frac{|u(x) - u(y)|^p}{|x - y|^{d + sp}}\d y\d x, & \text{if } 0 < s < 1, \\
\|\nabla u\|_{L^p(\Omega)}^p, & \text{if } s = 1.
\end{cases}
\end{align*}
Throughout, we adopt the convention that $|u|_{W^{s,p}(\Omega)}=\infty$ when $u\in L^p(\Omega)~\setminus~W^{s,p}(\Omega)$. For $0<\sigma<\eta<s< 1$ with $\eta=\theta \sigma+ (1-\theta)s$ with $\theta\in (0,1)$,
\begin{align*}
\iint_{\R^d \times \R^d}
\frac{|u(x) - u(y)|^p}{|x - y|^{d + \eta p}}\d y\d x=  \iint_{\R^d \times \R^d}
\frac{|u(x) - u(y)|^p}{|x - y|^{d + \theta \sigma p}|x - y|^{d + (1-\theta)sp}}\d y\d x.
\end{align*}
Applying Hölder's inequality yields the following well-known interpolation inequality, also recognized as a log-convex inequality
\begin{align*}
|u|_{W^{\eta, p}(\R^d)}\leq |u|^{\theta}_{W^{\sigma, p}(\R^d)} |u|^{1-\theta}_{W^{s, p}(\R^d)}.
\end{align*}
This estimate also remains true up a constant, for $s=1$ or $\sigma=0$. Roughly speaking, this estimate can be genuinely interpreted as a quantitative expression reflecting  the interpolation identity $W^{\eta, p}(\R^d) = \big[ W^{\sigma, p}(\R^d),\, W^{s, p}(\R^d) \big]_{\theta, p},$
where the right-hand side denotes the real interpolation space via Lions-Peetre $ K$-method; see for instance \cite{BeLo76,Trie95-interpolation,Lun18}. However, the above inequality is not robust, in the sense that it is not consistent as $\eta \to 1^-$ or $\eta \to 0^+$. Indeed, it is a well-known fact that for any sufficiently smooth, non-constant function $u \in L^p(\R^d)$, we have
\begin{align*}
|u|^p_{W^{\eta, p}(\R^d)} \xrightarrow{\eta\to 1^-}\infty\quad \text{ and }\quad |u|^p_{W^{\eta, p}(\R^d)} \xrightarrow{\eta\to 0^+}\infty.
\end{align*}
The overarching goal of this paper is to establish a robust interpolation inequality, ensuring that the blow-up near $\eta=0$ and $\eta=1$ is avoided.  To remedy these two anomalies, Brezis--Bourgain--Mironescu \cite{BBM01} and Maz'ya--Shaposhnikova \cite{MS02} showed, respectively, that the factor $\eta(1-\eta)$ annihilates  the singularities at $\eta=1$ and $\eta=0$.  More precisely, one has
\begin{align} \label{eq:frac-BBM}
& \eta(1-\eta)|u|^p_{W^{\eta, p}(\R^d)}
\xrightarrow{\eta\to 1^-} K_{d,p}  \frac{|\mathbb{S}^{d-1}|}{p}\|\nabla u\|^p_{L^p(\R^d)}, \qquad u\in W^{1,p}(\R^d),\\[2ex]
\label{eq:frac-mazya}
&\eta(1-\eta)|u|^p_{W^{\eta, p}(\R^d)} \xrightarrow{\eta\to 0^+}\frac{2|\mathbb{S}^{d-1}|}{p}\| u\|^p_{L^p(\R^d)}, \qquad u\in \bigcup_{0 < s < 1} W^{s,p}(\R^d),
\end{align}
where the constant $K_{d,p}$,  see \cite{Fog23}, is given by
\begin{align*}
K_{d,p}
&= \frac{1}{|\mathbb{S}^{d-1}|}\int_{\mathbb{S}^{d-1}} |w_d|^p \, d\sigma_{d-1}(w)
= \frac{\Gamma\big(\frac{d}{2}\big) \Gamma\big(\frac{p+1}{2}\big)}{\Gamma\big(\frac{d+p}{2}\big) \Gamma\big(\frac{1}{2}\big)}.
\end{align*}
The original proof of \eqref{eq:frac-mazya}, i.e., the  limit of the fractional seminorm as $\eta\to0^+$  due to Maz'ya--Shaposhnikova~\cite{MS02} is quite technical and uses advance tools like robust hardy inequalities. We  present  in Theorem \ref{thm:mazy'a-bbm-limit} a relatively simple approach asymptotic of the limit as $\eta\to 0^+$
as well as the proof of \eqref{eq:frac-BBM}, i.e., the   asymptotic as $ \eta \to 1^- $, due to Brezis--Bourgain--Mironescu~\cite{BBM01}.
Note in passing that, the formula of Maz'ya--Shaposhnikova was recently extended to a more general setting in \cite{DDGA25}.
In view of the aforementioned asymptotics, it is natural to opt for the following renormalized Sobolev seminorm
$[u]_{W^{s,p}(\Omega)}$ of order $s \in [0,1]$,
\begin{align*}
[u]^p_{W^{s,p}(\Omega)} :=
\begin{cases}
\displaystyle \frac{2|\mathbb{S}^{d-1}|}{p} \|u\|^p_{L^p(\Omega)}, & \text{if } s = 0, \\
\displaystyle s(1 - s) \iint_{\Omega\times\Omega}
\frac{|u(x) - u(y)|^p}{|x - y|^{d + sp}}\d y\d x, & \text{if } 0 < s < 1, \\
\displaystyle K_{d,p} \frac{|\mathbb{S}^{d-1}|}{p} \|\nabla u\|^p_{L^p(\Omega)}, & \text{if } s = 1.
\end{cases}
\end{align*}
The endpoint $p = \infty$ may be viewed heuristically as the limit of $ [u]_{W^{s,p}(\R^d)}$ or $|u|_{W^{s,p}(\R^d)}$ as $p \to \infty$, thereby motivating
\begin{align*}
|u|_{W^{s,\infty}(\Omega)}= [u]_{W^{s,\infty}(\Omega)} :=
\begin{cases}
\displaystyle \|u\|_{L^\infty(\Omega)}, & \text{if } s = 0, \\
\displaystyle \sup_{x,y\in \Omega,\,x \neq y} \frac{|u(x) - u(y)|}{|x-y|^s}, & \text{if } 0 < s < 1, \\
\displaystyle \|\nabla u\|_{L^\infty(\Omega)}, & \text{if }s = 1.
\end{cases}
\end{align*}
We must add the caveat that our convention for $[u]_{W^{0,p}(\Omega)}$, for a general $\Omega \subset \R^d$, does not necessarily reflect the limiting behavior of  $[u]_{W^{s,p}(\Omega)}$ as $s \to 0^+$.
In fact, while the asymptotic behavior of $[u]_{W^{s,p}(\Omega)}$ as $s \to 1^-$ is governed by the geometric regularity of $\Omega$, we shall in see Section \ref{sec:assymp-s-near-1-0} that the asymptotic behavior of $[u]_{W^{s,p}(\Omega)}$ as $s \to 1^-$  depends on the geometric shape of $\Omega$.
\noindent Our primary motivation for considering the robust interpolation inequality stems from the study of the asymptotic behavior of solutions to the Dirichlet problem associated with the regional fractional $p$-Laplacian $(-\Delta)^s_{p,\Omega}$ as $s\to 1^-$. Therefore, another natural choice for renormalizing the Gagliardo seminorm is to consider the normalization constant associated with the fractional $p$-Laplacian $(-\Delta)^s_p \equiv (-\Delta)^s_{p,\mathbb{R}^d}$, where the regional operator $(-\Delta)^s_{p,\Omega} $ is defined by
\begin{align*}
\begin{split}
(-\Delta)^s_{p,\Omega} u(x)
&:= \widetilde{C}_{d,p,s}\pv \int_{\Omega}\frac{|u(x)-u(y)|^{p-2} (u(x)-u(y))}{|x-y|^{d+sp}}\d y,\qquad s\in (0,1)\\
(-\Delta)^1_{p,\Omega} u(x)
&:= -\Delta_p u(x)	=-\mathrm{div}(|\nabla u(x)|^{p-2}\nabla u(x)), \qquad \qquad s=1.
\end{split}
\end{align*}
Here following \cite{Fog25}, our chosen normalizing constant is
\begin{align}\label{eq:normalized-frac-cons}
\widetilde{C}_{d,p,s} =
\begin{cases}
C_{d,p,s}, & \text{if } sp > 1, \\
C_{d,2,\frac{sp}{2}}, & \text{if } sp \leq 1
\end{cases}
\quad \text{with}\quad
C_{d,p,s}= \frac{s(1-2s)\pi^{\frac{1-d}{2}}\Gamma(\frac{d+sp}{2})}{\Gamma(\frac{1+s p}{2})\Gamma(p(1-s)) \cos(s\pi)}.
\end{align}
 Namely, we consider the renormalized seminorm
\begin{align*}
[u]^{*p}_{W^{s,p}(\Omega)}
& :=
\begin{cases}
\|u\|_{L^p(\Omega)}^p, & \text{if } s = 0, \\
\displaystyle \frac{\widetilde{C}_{d,p,s}}{2} \iint_{\Omega\times\Omega}
\frac{|u(x) - u(y)|^p}{|x - y|^{d + sp}}\d y\d x, & \text{if } 0 < s < 1, \\
\|\nabla u\|_{L^p(\Omega)}^p, & \text{if } s = 1,
\end{cases}\\
[u]^{*}_{W^{s,\infty}(\Omega)} &:=[u]_{W^{s,\infty}(\Omega)}=|u|_{W^{s,\infty}(\Omega)}
\qquad \text{ if } p=\infty.
\end{align*}
In short, we can write $[u]^{*}_{W^{s,p}(\Omega)} =\big(B_{d,p,s}\big)^{\frac{1}{p}}[u]_{W^{s,p}(\Omega)},$ $s\in [0,1]$
where

\begin{align*}
B_{d,p,s}=  \frac{\widetilde{C}_{d,p,s}}{2s(1-s)}, \qquad B_{d,p,0}=\frac{p}{2|\mathbb{S}^{d-1}|}\qquad\text{and}\qquad B_{d,p,1}=\frac{p}{K_{d,p}|\mathbb{S}^{d-1}|}.
\end{align*}
As we  show in Section \ref{sec:assymp-s-near-1-0}, the normalization constant $\widetilde{C}_{d,p,s}$, is chosen so as to ensure that the following properties hold true for $u\in C_c^\infty(\R^d)$.
\begin{itemize}
\item $\widehat{(-\Delta)^s u}(\xi)= |\xi|^{2s}\widehat{u}(\xi)$  for all $u\in C_c^\infty(\R^d)$, where $ \widehat{u}(\xi) =\int_{\R^d} e^{-ix\cdot \xi} u(x)\d x$.
\item  $(-\Delta)^s_p u(x)\to -\Delta_p u(x)\quad$  and  $\quad[u]^*_{W^{s,p}(\R^d)}\to  \|\nabla u\|_{L^{p}(\R^d)}$ as  $s\to1^-$.
\item $(-\Delta)^s_p u(x)\to |u(x)|^{p-2} u(x)\quad$  and  $\quad[u]^*_{W^{s,p}(\R^d)}\to \|u\|_{L^{p}(\R^d)}$ as  $s\to0^+$.
\item The following asymptotic behaviors hold
\begin{align*}
\lim_{s\to 1^-} \frac{K_{d,p}C_{d,p,s}}{s(1-s)}=\lim_{s\to 0^+}\frac{\Gamma(p+1)C_{d,p,s}}{s(1-s)}=\lim_{s\to 0^+} \frac{2C_{d,2, \frac{sp}{2}}}{s(1-s)}=\frac{2p}{|\mathbb{S}^{d-1}|}.
\end{align*}
\end{itemize}
In particular we have $\widetilde{C}_{d,p,s}\asymp s(1-s)$. Others normalizing constants for the fractional $p$-Laplacian are proposed in \cite{dTGCV21} and also in \cite{DJS25} where the need for asymptotic near $s \to 0^+$ analysis, in connection with the logarithmic $p$-Laplacian, is considered.
\noindent  Our first main result consists of the following robust interpolation inequality.
\begin{theorem}
\label{thm:robust-interpo-esti}
Let $u\in L^p(\mathbb{R}^d)$ with $1\leq p\leq \infty$ and let $0\leq \sigma<\eta<s\leq 1$ say $\eta=\theta s+(1-\theta)\sigma$ with $
\theta=\frac{\eta-\sigma}{s-\sigma}\in(0,1).$
Then the following variants of the log-convex interpolation inequalities hold.
\medskip

\noindent
$\mathbf{(I)}$. (Dimension free constant).  For the renormalized seminorm $[\cdot]_{W^{s,p}(\R^d)}$ we have
\begin{align*}
[u]_{W^{\eta,p}(\R^d)}
&\leq
\varkappa N_p(\sigma,\eta,s)
\,[u]_{W^{s,p}(\R^d)}^{\theta}
\cdot [u]_{W^{\sigma,p}(\R^d)}^{1-\theta},\\
[u]_{W^{\eta,p}(\R^d_+)}
&\leq
4^{1/p}\varkappa N_p(\sigma,\eta,s)
\,[u]_{W^{s,p}(\R^d_+)}^{\theta}
\cdot [u]_{W^{\sigma,p}(\R^d_+)}^{1-\theta},
\end{align*}
where $\tiny{\varkappa = \begin{cases} 2^{\frac{1}{p}} & \text{if } d = 1, \\[-2pt]
 1 & \text{if } d \geq 2. \end{cases}}$ and the dimension-free constant $N_p(\sigma,\eta,s)$ is defined by
\begin{align*}
N_p(\sigma,\eta,s)
= 2^{\frac{(s-\eta)}{s-\sigma}(\eta+1-\sigma)}
\left(\frac{s-\sigma}{\eta-\sigma}\right)^{\frac{\eta-\sigma}{(s-\sigma)p}}
\left(\frac{s-\sigma}{s-\eta}\right)^{\frac{s-\eta}{(s-\sigma)p}},
\end{align*}
and satisfies $1\leq N_p(\sigma,\eta,s)\leq  2^{1+\frac{1}{p}}$.
\medskip

$\mathbf{(II)}$. (Gagliardo seminorm).
For the Gagliardo  seminorm $|\cdot|_{W^{s,p}(\R^d)}$ we have
\begin{align*}
|u|_{W^{\eta,p}(\R^d)}
&\leq
\Upsilon_{d,p}(\sigma,\eta,s)
|u|_{W^{s,p}(\R^d)}^\theta\cdot
|u|_{W^{\sigma,p}(\mathbb{R}^d)}^{1-\theta},\\
|u|_{W^{\eta,p}(\R^d_+)}
&\leq
4^{1/p}\Upsilon_{d,p}(\sigma,\eta,s)
|u|_{W^{s,p}(\R^d_+)}^\theta\cdot
|u|_{W^{\sigma,p}(\R^d_+)}^{1-\theta}.
\end{align*}
where, we have   $\Upsilon_{d,\infty}(\sigma,\eta,s)=N_\infty(\sigma,\eta,s),$ while
\begin{align*}
\Upsilon_{d,p}(\sigma,\eta,s)
&=
\frac{\varkappa N_p(\sigma,\eta,s)}{[\eta(1-\eta)]^{\frac{1}{p}}}
\times
\begin{cases}
\big[(s(1-s))^\theta(\sigma(1-\sigma))^{1-\theta}\big]^{\frac{1}{p}},
& \sigma\neq0,\ s\neq1,
\\
\big[(K_{d,p}\tfrac{|\mathbb{S}^{d-1}|}{p})^\theta
(\sigma(1-\sigma))^{1-\theta}\big]^{\frac{1}{p}},
& \sigma\neq0,\, s=1,
\\
\big[(s(1-s))^\theta(\tfrac{ 2|\mathbb{S}^{d-1}|}{p})^{1-\theta}\big]^{\frac{1}{p}},
& \sigma=0, \, s\neq1,
\\
\big[\tfrac{|\mathbb{S}^{d-1}|}{p}K_{d,p}^\theta2^{1-\theta}\big]^{\frac{1}{p}},
& \sigma=0,\, s=1.
\end{cases}
\end{align*}

$\mathbf{(III)}$. (Renormalized seminorm).  For the renormalized seminorm $[\cdot]^*_{W^{s,p}(\R^d)}$ we have
\begin{align*}
[u]^*_{W^{\eta,p}(\mathbb{R}^d)}
&\leq
\Upsilon^*_{d,p}(\sigma,\eta,s)
[u]_{W^{s,p}(\mathbb{R}^d)}^{*\theta}
[u]_{W^{\sigma,p}(\mathbb{R}^d)}^{*\,(1-\theta)},
\\
[u]^*_{W^{\eta,p}(\mathbb{R}^d_+)}
&\leq
4^{1/p}\Upsilon^*_{d,p}(\sigma,\eta,s)
[u]_{W^{s,p}(\mathbb{R}^d_+)}^{*\theta}
[u]_{W^{\sigma,p}(\mathbb{R}^d_+)}^{*\,(1-\theta)},
\end{align*}
where, we have  $\Upsilon^*_{d,\infty}(\sigma,\eta,s)=N_\infty(\sigma,\eta,s),$ while
\begin{align*}
\Upsilon^*_{d,p}(\sigma,\eta,s)
=\varkappa N_p(\sigma,\eta,s) \Big(\frac{B_{d,p,\eta}}
{B_{d,p,\sigma}^{\theta}B_{d,p,s}^{1-\theta}}
\Big)^{\frac{1}{p}} \leq 2^{1+\frac{2}{p}}\frac{B_{\max}}{B_{\min}},
\end{align*}
with  $B_{\max}= \max_{s\in [0,1]} B_{d,p,s}$ and $B_{\min}= \min_{s\in [0,1]} B_{d,p,s}$.
\end{theorem}
\smallskip

\noindent It worth mentioning that, the cactus case $p = 2$ is particularly amenable to analysis using the Fourier transform. As obtained in  Theorem \ref{thm:robust-interpo-L2}, we get  $\Upsilon^*_{d,2}(\sigma,\eta, s)=1$, namely we have the following sharper estimate
\begin{align*}
[u]^*_{H^\eta (\R^d)}\leq[u]^{*\,\theta}_{H^s (\R^d)} \cdot  [u]^{*\,(1-\theta)}_{H^\sigma(\R^d)}.
\end{align*}
However, the resulting constant $N_2(\sigma, \eta, s)$ obtained with this method is less sharp when the dimension is large. Indeed  for  the seminorms $[\cdot ]_{H^t (\R^d)}$, we deduce
\begin{align*}
[u]_{H^\eta (\R^d)}
&\leq N_2(\sigma, \eta, s)[u]_{H^s (\R^d)}^{\theta}  \cdot [u]_{H^\sigma(\R^d)}^{1-\theta}
\end{align*}
 where $N_2(\sigma, \eta, s) =\big(\tfrac{B_{d,2,\sigma}^{\theta}B_{d,2,s}^{1-\theta}}{B_{d,2,\eta}} \big)^{1/2}\leq \sqrt{2d}.$
It is worth emphasizing that, rather than interpolating norms, our interpolation inequality applies only to seminorms. More generally, inequalities of this type are known as Gagliardo–Nirenberg inequalities. We refer the interested reader to \cite{BM18} for a thorough and comprehensive treatment of these inequalities. The robust log-convex inequality in the case $\sigma = 0$ and $s = 1$ was carried out in a less precise form in \cite[Theorem 3.54]{guy-thesis} and the non-robust ones can be found in \cite[Chapter 7]{Leo23}.
Although we did not set out  to address the open problem posed in \cite[Question 1.10]{BSY22}, the robust interpolation inequality established here yields, as a byproduct, a positive answer to that question in a considerably more general setting. It is important to highlight that a different answer to the latter open question is given in \cite[Theorem 3.9]{DoMi23}.
The key tool for establishing the robust interpolation inequality in Theorem \ref{thm:robust-interpo-esti} is the decomposition of the Gagliardo seminorm into short-range and long-range interactions. To this end, let us observe that passing through polar coordinates yields
\begin{align*}
\iint_{\R^d \times \R^d} \frac{ |u(x + h) - u(x)|^p}{|h|^{d+s p}}\d h\d x
&= \int_0^\infty   \frac{U(t) }{t^{1+s p}}\d t,
\end{align*}
where $U:[0,\infty)\to [0,\infty)$ is the averaged $L^p$-modulus of continuity of $u$,
\begin{align*}
U(t) &= \int_{\mathbb{S}^{d-1}} \int_{\R^d} |u(x + t\omega) - u(x)|^p \d x \d\sigma_{d-1}(\omega).
\end{align*}
\begin{definition}[Short and long range interaction]
\label{def:long-short-range}
\noindent We introduce the  short-range (local-range) interaction  $L_p(\cdot,u,r):[0,1]\to [0,\infty]$  for $0\leq s< 1$ by
\begin{align*}
 L_p(s, u,r)&:= \frac{(1-s)^{\frac{1}{p}}}{(2r)^{(1-s)} } \Big(\int_0^r \frac{U(t)}{t^{1+s p}} \d t\Big)^{\frac{1}{p}}=\frac{(1-s)^{\frac{1}{p}}}{(2r)^{(1-s)} } \Big(\iint_{\R^d \times B(x,r)} \hspace{-1ex} \frac{ |u(x) - u(y)|^p}{|x-y|^{d+s p}}\d y \d x\Big)^{\frac{1}{p}}.
\end{align*}
 The long-range (or nonlocal-range) interaction $T_p(\cdot,u,r):[0,1]\to [0,\infty)$ also called the fractional tail is defined  for $0<s\leq 1$ by
\begin{align*}
T_p(s, u,r)&:= s^{\frac{1}{p}}r^{s}\Big(\int_r^\infty \frac{U(t)}{t^{1+s p}} \d t\Big)^{\frac{1}{p}}=s^{\frac{1}{p}}r^{s} \Big(\iint_{\R^d \times B^c(x,r)} \hspace{-1ex} \frac{ |u(x) - u(y)|^p}{|x-y|^{d+s p}}\d y \d x\Big)^{\frac{1}{p}}.
\end{align*}
The definition of the endpoint cases $L_p(1, u, r)$ and $T_p(0, u, r)$ are fully justified by Theorem~\ref{thm:asymp-larg-short}, which establishes the following asymptotic limits
\begin{align*}
L_p(1, u,r)= \lim_{s\to1^-} L_p(s, u,r)
&= \Big(\frac{|\mathbb{S}^{d-1}|}{p} K_{d,p}\Big)^{\frac{1}{p}}\|\nabla u\|_{L^p(\R^d)}= [u]_{W^{1,p}(\R^d)},\\
T_p(0, u,r)= \lim_{s\to0^+} T_p(s, u,r)
&=\Big(\frac{2|\mathbb{S}^{d-1}|}{p}\Big)^{\frac{1}{p}}\|u\|_{L^p(\R^d)}=[u]_{W^{0,p}(\R^d)}.
\end{align*}
By heuristically taking the limit as $p \to \infty$, we analogously define the quantities
\begin{align*}
 L_\infty(s, u,r)
&:=(2r)^{-(1-s)}\hspace{-2ex}\sup_{0<|x-y|\leq r} \frac{|u(x)-u(y)|}{|x-y|^s}\quad \qquad (0\leq s< 1),
\\
 T_\infty(s, u,r) & := r^{s}\sup_{|x-y|\geq r} \frac{|u(x)-u(y)|}{|x-y|^s} \quad \qquad\quad  (0<s\leq 1),\\
L_\infty(1,u,r) &:= [u]_{W^{1,\infty}(\R^d)},
\qquad\text{and}\qquad  T_\infty(0,u,r) :=[u]_{W^{0,\infty}(\R^d)}.
\end{align*}
\end{definition}
The short-range interaction  $L_p(s,u,r)$ encodes  the local (and hence the regularity of $u$) contribution to the fractional seminorm $|u|_{W^{s,p}(\R^d)}$  by  capturing  the behavior of $u$ at short scales inside the ball $ B(0,r)$.
The long-range interaction  $T_p(s,u,r)$ encodes the non-local contribution to the fractional seminorm by capturing the behavior of $u$ at long scales outside the ball $ B(0,r)$. In other words $T_p(s,u,r)$ captures  the nonlocal effect of $|u|_{W^{s,p}(\R^d)}$.
\noindent The cornerstone of our approach to Theorem~\ref{thm:robust-interpo-esti} is the following monotonicity result, which stands as our second main result.
\begin{theorem}[Monotonicity of  short and long range interaction]
\label{thm:r-split-monot-semin}
Let  $u\in L^p(\R^d)$, $1\leq p\leq \infty$ and $r>0$. The following assertions are true.
\begin{enumerate}[$(A)$]
\item
The map $s\mapsto  L_p(s,u,r)$ is continuous and  increasing. Namely, we have
\begin{align*}
L_p(\eta, u,r)\leq  L_p(s, u,r)\qquad \text{for \,\, $0\leq  \eta<s\leq 1$}.
\end{align*}
In fact for $p=\infty$ we have the better estimate $
2^{-\eta}L_\infty(\eta, u,r)\leq  2^{-s}L_\infty(s, u,r).$

\item The map $s\mapsto T_p(s,u,r)$ is continuous and for $0\leq \sigma< \eta \leq 1$ we have
\begin{align*}
T_p(\eta, u,r) &\leq 2^{1+\frac{1}{p}}\, \sigma^{\frac{1}{p}} r^{ \sigma}\Big(\int_0^\infty \frac{U(t)}{t^{ 1+\sigma p}} \d t\Big)^{\frac{1}{p}}=2^{1+\frac{1}{p}}  r^\sigma(1-\sigma)^{-\frac{1}{p}}[u]_{W^{\sigma,p}(\R^d)}.
\end{align*}
\item If $d \geq 2$ then the map $s \mapsto T_p(s, u, r)$ is $2$-almost decreasing, that is,
\begin{align*}
T_p(\eta, u, r) \leq2 T_p(\sigma, u, r) \qquad \text{ for }\quad  0 \leq \sigma <\eta \leq 1.
\end{align*}
In fact, this estimate remains true provided that the following condition holds:
\begin{align}\label{cond:interp-admissible}
\tag{G}
d\geq 2\quad \text{or} \quad \sigma = 0 \quad \text{or} \quad p = 2 \quad \text{or} \quad p = \infty.
\end{align}
\end{enumerate}
\end{theorem}
Let us make three important comments regarding Theorem \ref{thm:r-split-monot-semin}. (a) Broadly speaking, Theorem \ref{thm:r-split-monot-semin} implies that the mapping $s \mapsto L_p(s,u,r)$ increases until it reaches the local part of $u$, given by $L_p(1,u,r) = [u]_{W^{1,p}(\R^d)}$, whereas $s \mapsto T_p(s,u,r)$ essentially decreases until it reaches the nonlocal part of $u$, given by $T_p(0,u,r) = [u]_{W^{0,p}(\R^d)}$. This corroborates the idea hinted above that $L_p(s,u,r)$ truly encodes the locality of $u$, while $T_p(s,u,r)$ encodes the nonlocality of $u$. (b) The gap for $d=1$ is due to the fact that for $d \geq 2$  that our strategy relies heavily  on the  slicing method of the unit sphere $\mathbb{S}^{d-1}$ by great spheres, which brings into play the so-called Funk–Radon transform (see Theorem \ref{thm:slicing-integ-on-sphere}). However, although we were unable to prove the inequality $T_p(\eta, u, 1) \le 2T_p(\sigma, u, 1)$ for $0 \le \sigma < \eta < 1$ in general when $d=1$, we have shown that it holds in the special cases $\sigma=0$, $p=2$, and $p=\infty$. It is therefore natural to formulate the following conjecture.

\medskip
\textbf{Conjecture (open problem).}
There exists $C_p > 0$ such that for all $u \in L^p(\mathbb{R})$
\begin{align*}
\eta\int_1^\infty\frac{U(t)}{t^{1+\eta p}}\d t\leq C_p  	\sigma \int_1^\infty\frac{U(t)}{t^{1+\sigma p}}\d t \qquad 0<\sigma<\eta<1 \quad \& \quad d=1.
\end{align*}
This estimate is true for $\sigma=0$ with $C_p=2^{p}$. In Section \ref{sec:conjecture}, we present some related observations  and explain how a sharper constant $C_p$ can help improve the monotonicity properties of the map $s \mapsto T_p(s,u,r)$ for $d \geq 1$.

(c) The proof of Theorem \ref{thm:r-split-monot-semin} relies on the application of Chebyshev-type inequalities.  Loosely speaking, we say that a class $\mathcal{C}$ of pairs of integrable functions $(f,g)$ over a measure space $(X, \mu)$ with $0<\mu(X)<\infty$ satisfies a Chebyshev-type integral inequality if there exists a constant $c > 0$, independent of $f$ and $g$, such that
\begin{align*}
\int_X f(x)g(x)\d \mu (x) \geq \,  (\leq ) \, c \Big(\int_X f(x)\d \mu (x) \Big)\Big(\int_X g(x)\d \mu (x) \Big) \quad\text{for all $(f,g) \in \mathcal{C}$}.
\end{align*}
For $c = 1/\mu(X)$, one recovers the classical Chebyshev integral inequality, which holds when $f$ and $g$ are synchronous (resp. asynchronous), i.e., they satisfy
\begin{align*}
(f(x)-f(y))(g(x)-g(y))\geq \, (\leq )\, 0 \qquad  \text{for all $x,y\in X$}.
\end{align*}
Most the existing literature only addresses the Chebyshev integral inequality with the specific constant $c = 1/\mu(X)$. This is likely due to the fact that such inequalities are traditionally applied in probability theory rather than in mathematical analysis. Given the lack of references treating the inequality with a general constant $c$, we provide a detailed discussion in Section \ref{sec:chebyshev-ineq} on several classes of functions satisfying these general inequalities; serving as a key ingredient to prove Theorem \ref{thm:r-split-monot-semin}.  In particular, given $r>0$ and a suitable measurable function $g: [0,\infty)\to [0,\infty]$ we investigate  in Theorem \ref{thm:cadic-monotonicity} the monotonicity of the following singular integrals
$T,L: [0,\infty) \to[0,\infty]$ with
\begin{align*}
z\mapsto L(z)&= zr^{-z}\int_0^r  g(t)t^{z-1} \d t,
\qquad \text{with} \qquad  L(0)= g(0), \\
z\mapsto T(z)&= zr^{z} \int_r^\infty g(t) t^{-z-1} \d t,
\qquad \text{with} \qquad T(0)=g(\infty).
\end{align*}
As a primary application of the robust log-convex inequalities from Theorem \ref{thm:robust-interpo-esti}, we are able obtain the strong convergence of weak solutions to the Dirichlet boundary value problem associated with the fractional $p$-Laplacian $(-\Delta)^s_{p,\Omega}$.

\begin{theorem}
\label{thm:converg-dirichlet-frac-regional}
Let $\Omega \subset \R^d$ be a bounded Lipschitz domain and let $1<p<\infty$. Assume that $(f_s)_s \subset L^{p'}(\Omega)$, $p'=p/(p-1)$, converges weakly to $f_1$ in $L^{p'}(\Omega)$ as $s\to 1^-$, and that $(g_s)_s \subset W^{1-\frac{1}{p},p}(\partial\Omega)$ converges strongly to $g_1$ in $W^{1-\frac{1}{p},p}(\partial\Omega)$, namely, $\|g_s-g_1\|_{W^{1-\frac{1}{p},p}(\partial\Omega)}
\xrightarrow{s\to1^-}0.$  For each $\frac{1}{p} < s \le 1$, let  the function
$u_s \in W^{s,p}(\Omega)$ satisfies in the weak sense the Dirichlet problem
\begin{align*}
(-\Delta)^s_{p,\Omega} u_s= f_s \quad \text{in }\,\,  \Omega, \qquad \text{  and }\qquad \gamma^s_0 (u_s )= g_s \quad \text{on } \,\, \partial \Omega,
\end{align*}
\noindent with $(-\Delta)^1_{p,\Omega}=-\Delta_{p} $. Then $(u_s)_s$ strongly converges to $u_1$  in the following sense
\begin{align*}
\lim_{s\to1}\|u_s -u_1\|_{W^{\eta,p}(\Omega)} =0 \qquad\text{for all }\quad 0\leq \eta<1.
\end{align*}
\end{theorem}
The asymptotic behavior and convergence properties for the nonlocal Dirichlet problem governed by the full fractional $p$-Laplacian—and more generally in the context of pseudo-differential $p$-L\'evy operators—have also been investigated in \cite{Fog25, Fog26} see also \cite{FoKa24,BudT26,guy-thesis,Voi17} for the case $p=2$.
\noindent We summarize below the other contributions of this paper, all of which are direct consequences and applications of our robust interpolation inequality. The structure of the article and its remaining contributions are organized as follows:
\begin{itemize}
    \item \textit{Section~\ref{sec:chebyshev-ineq} (Chebyshev inequalities):} We establish several Chebyshev-type integral inequalities for general non-synchronous functions, which serve as the technical groundwork for our main results.

    \item \textit{Section~\ref{sec:Lp-moduos-of-cont} ($L^p$-Modulus of Continuity):} We analyze the analytical properties and asymptotic profiles of the averaged $L^p$-modulus of continuity function.

    \item \textit{Section~\ref{sec:assymp-s-near-1-0} (Seminorm Asymptotics):} We investigate the asymptotic behavior of fractional Sobolev seminorms and the normalized regional $p$-Laplacian as $s \to 0^+$ and $s \to 1^-$.  In particular, Theorem~\ref{thm:mazya-on-dom} generalizes the Maz'ya--Shaposhnikova formula to arbitrary open sets.
 Namely, for any open set $\Omega \subset \R^d$ and $u \in \bigcup_{0 < s < 1} W^{s,p}(\Omega) \cap L^p(\Omega, \delta_x^{-sp})$ with $1 \le p < \infty$, we have
\begin{align*}
\lim_{s \to 0^+} s(1-s)|u|^p_{W^{s,p}(\Omega)}=
\frac{2}{p} \int_\Omega |u(x)|^p
\int_{ \mathbb{S}^{d-1}}\hspace{-2ex} \mathds{1}_{\mathbb{S}^{d-1}\setminus S_\Omega(x)}(w)\d \sigma_{d-1}(w)  \d x,
\end{align*}
where $S_\Omega(x) =
 \{ w \in \mathbb{S}^{d-1} :  \lim\limits_{r\to \infty} \mathds{1}_{\Omega^c}(x+rw)=1\}$ is the escape sector at $x$.

    \item \textit{Section~\ref{sec:robust-interpo} (Robust Interpolation): } We prove  Theorem~\ref{thm:r-split-monot-semin} and our core robust log-convex interpolation estimates Theorems~\ref{thm:robust-interpo-esti}. By the mean of the robust extension operator developed in Appendix~\ref{sec:robust-exten},  Theorem~\ref{thm:robust-interpo-dom} extends the robust log-convex inequalities to domains with compact Lipschitz boundaries.
    \item \textit{Section~\ref{sec:application-interpola} (Applications $\&$ Convergence):} We explore the consequences of the robust interpolation inequalities. The main results include:
    \begin{itemize}
        \item \emph{Asymptotic compactness:}  In the spirit of the results from \cite{Pon04} and \cite{BBM01}, we obtain in Theorem~\ref{thm:asymp-compact-frac},  a Lions--Peetre type compactness result by  establishing the strong simultaneous convergence sequence of functions and their traces.
        \item \emph{Robust regional Poincar\'e--Friedrichs inequality:} In Theorem~\ref{thm:robust-poinca-fried-regio}, we obtain a robust regional Poincar\'e--Friedrichs inequality,  $\|u\|_{L^p(\Omega)} \leq C [u]_{W^{s,p}(\Omega)}$  for all $u \in W_0^{s,p}(\Omega)$ and $s_0 < s < 1$
for constants $C = C(d,p,\Omega) > 0$ and $s_0 = s_0(d,p,\Omega) \in (\frac{1}{p}, 1)$ that are independent of $s$.
\item \emph{Convergence of Dirichlet problem:} We show in  Theorem~\ref{thm:converg-dirichlet-frac-regional}, the strong nonlocal-to-local convergence in $W^{\eta,p}(\Omega)$, $0\leq \eta<1$ as $s\to 1^-$ of unique weak solutions for the regional fractional $p$-Laplacian $(-\Delta)^s_{p,\Omega}$.
 \end{itemize}
\item \textit{Appendix~\ref{sec:robust-exten} (Robust extension):} We show in Theorem \ref{thm:robust-extension} that if the boundary $\partial \Omega$ Lipschitz and  compact, there exists a robust Sobolev extension operator $E: W^{s,p}(\Omega) \to W^{s,p}(\R^d)$, namely, the operator norm of $E$ is controlled by a constant that is entirely independent of $s$.
\end{itemize}

\medskip

\emph{Acknowledgment:}
The author expresses their sincere gratitude to Sangchul Lee, Rui Chen, Tobias Weth, and Sven Jarohs for their inspiring and insightful discussions regarding the one-dimensional gap conjecture.

\section{Chebyshev-type integral inequalities}
\label{sec:chebyshev-ineq}

The classical Chebyshev inequality \cite{Cheb82}, along with its historical development and key contributions, is thoroughly discussed in \cite[Chapter~X]{MPF93} and \cite{MiVa75,MiPe90}.   Necessary and sufficient conditions for the validity of Chebyshev's inequalities are established in \cite{Shi11}. In probability theory, Chebyshev's integral inequality is frequently referred to as the correlation inequality (see for instance Corollary \ref{cor:cheby}), reflecting the fact that comonotone functions of a random variable are always positively correlated. Beyond this, the inequality has broad applications across analysis, statistics, and information theory, where it underpins deviation bounds, the weak law of large numbers, and moment estimates. 

\subsection{General Chebyshev-type integral inequalities}
In this section, $f,g \in L^1(X)$ are assumed to be  two integrable functions on measure space $(X,\mu)$ with a  finite positive measure, i.e., $0 < \mu(X) < \infty$.
\begin{theorem}[Chebyshev's inequality]\label{thm:cheby}
Let $f,g \in L^1(X,\mu)$ with $0 < \mu(X) < \infty$.

If $f$ and $g$ are \emph{synchronous} (also called \emph{comonotone}), that is,
\begin{align*}
(f(x) - f(y))(g(x) - g(y)) \geq 0, \quad \text{for all } x,y \in X,
\end{align*}
then the following inequality holds:
\begin{align*}
\int_X f(x) g(x) \d\mu(x) \geq\frac{1}{\mu(X)}
\int_X f(x)\d\mu(x) \int_X g(x)\d\mu(x).
\end{align*}
Equivalently, if $f$ and $g$ are \emph{asynchronous}  (also called \emph{countermonotone}), that is,
\begin{align*}
(f(x) - f(y))(g(x) - g(y)) \leq 0, \quad \text{for all } x,y \in X,
\end{align*}
then the reverse inequality holds:
\begin{align*}
\int_X f(x) g(x) \d\mu(x) \leq\frac{1}{\mu(X)}
\int_X f(x)\d\mu(x) \int_X g(x)\d\mu(x).
\end{align*}
\end{theorem}

\begin{proof}
We only establish the first claim. Integrating
\begin{align*}
0\leq (f(x) - f(y))(g(x) - g(y))= f(x)g(x)+ f(y)g(y) -f(x)g(y)- f(y)g(x)
\end{align*}
over $X \times X$ yields desired estimate since
\begin{align*}
0&\leq 2\mu(X) \int_X   f(x)g(x) \d \mu(x)- 2\int_Xg(y) \d \mu(y) \int_X  f(x)\ d \mu(x).
\end{align*}
\end{proof}
\begin{remark}
It is somewhat counterintuitive that comonotonicity for functions on $X \subset \R$ does not necessarily imply that $f$ and $g$ are individually monotone. Rather, it means that $f$ and $g$ vary in the same direction--either both increasing or both decreasing--over the same subsets of $X$. In other words, the two functions move in sync. Consequently, the term synchronous functions may be more intuitive in many contexts, as it directly emphasizes tandem-movement rather than individual monotonicity. For example, for any function $k:\R \to \R$, the functions $f(x)= k^2(x)$ and $g(x)= k^{4}(x)$ share the same monotonicity on the same sets. 
\end{remark}
  In probability space becomes Theorem~\ref{thm:cheby} the following.
\begin{corollary}\label{cor:cheby}
Let  $\mathbf{Z}: \Omega \to \R$ be a random variable on a  probability space $(\Omega, \mathcal{F}, \mathbb{P})$ and $f, g : \R \to \R$ be measurable functions such that $f(\mathbf{Z})$ and $g(\mathbf{Z})$ are integrable.

If $f$ and $g$ are synchronous, then the following inequality holds
\begin{align*}
\mathbb{E}[f(\mathbf{Z})g(\mathbf{Z})] \geq \mathbb{E}[f(\mathbf{Z})] \cdot \mathbb{E}[g(\mathbf{Z})].
\end{align*}
If $f$ and $g$ are asynchronous, then the reverse inequality holds
\begin{align*}
\mathbb{E}[f(\mathbf{Z})g(\mathbf{Z})] \leq \mathbb{E}[f(\mathbf{Z})] \cdot \mathbb{E}[g(\mathbf{Z})].
\end{align*}
\end{corollary}
Next, we  establish several Chebyshev-type integral inequalities with explicit constants, under relaxed variants of comonotonicity.
\begin{definition}[(b,c)-Synchronicity]
Consider two functions $f, g: X \to \R$,  $b > 0$ and $c > 0$. We say that  the pair $(f, g)$ is $(b,c)$-\emph{synchronous} if
\begin{align*}
(f(x) - b f(y)) (g(x) - c g(y)) \geq 0 \quad \text{for all } x, y \in X.
\end{align*}
We say that $(f, g)$ is $(b,c)$-\emph{asynchronous} if
\begin{align*}
(f(x) - b f(y)) (g(x) - c g(y)) \leq 0 \quad \text{for all } x, y \in X.
\end{align*}
\end{definition}

The result below generalizes the classical Chebyshev inequality, Theorem~\ref{thm:cheby}; the proof is analogous and omitted for brevity.
\begin{theorem}[Chebyshev's inequality I]\label{thm:cheby-scaled}
Let $f,g \in L^1( X, \mu)$ and  $b > 0$ and $c > 0$.

If $f$ and $g$ are $(b,c)$-\emph{synchronous} then
\begin{align*}
\int_X f(x) g(x) \d\mu(x) \geq \frac{b + c}{1 + bc} \cdot \frac{1}{\mu(X)}
\Big(\int_X f(x)\d\mu(x) \Big)\Big(\int_X g(x)\d\mu(x)\Big).
\end{align*}
If $f$ and $g$ are $(b,c)$-\emph{asynchronous}  then
\begin{align*}
\int_X f(x) g(x) \d\mu(x) \leq\frac{b + c}{1 + bc} \cdot \frac{1}{\mu(X)}
\Big(\int_X f(x)\d\mu(x) \Big)\Big(\int_X g(x)\d\mu(x)\Big).
\end{align*}
In particular, if $b=1$ or $c=1$ one gets the Chebyshev inequality from Theorem~\ref{thm:cheby}.
\end{theorem}

\begin{definition}[Weak (b,c)-Synchronicity]
Consider two functions $f, g: X \to \R$,  $b > 0$ and $c > 0$. We say that the pair $(f, g)$ is \emph{weakly $(b,c)$-synchronous} (or \emph{weakly $(b,c)$-comonotone}) if,  for all $x, y \in X$, we have
\begin{align*}
\big( f(x) \leq b f(y) \quad \text{and} \quad g(x) \leq c g(y) \big)
\quad \text{or} \quad
\big( f(y) \leq b f(x) \quad \text{and} \quad g(y) \leq c g(x) \big).
\end{align*}
Equivalently, there holds if for all $x, y \in X$,
\begin{align*}
(f(x) - b f(y))(g(x) - c g(y)) \geq 0\quad \text{or} \quad
(f(x) - b^{-1} f(y))(g(x) - c^{-1}g(y))\geq 0.
\end{align*}
Similarly, we say that $(f, g)$ is \emph{weakly $(b,c)$-asynchronous} (or \emph{weakly $(b,c)$-countermonotone}) if,  for all $x, y \in X$, we have
\begin{align*}
\big(f(x) \leq b f(y) \quad \text{and} \quad cg(x) \geq  g(y) \big)
\quad \text{or} \quad
\big( f(y) \leq b f(x) \quad \text{and} \quad cg(y) \geq g(x) \big).
\end{align*}
Equivalently, there holds if for all $x, y \in X$,
\begin{align*}
(f(x) - b f(y))(g(x) - c^{-1} g(y)) \leq 0\quad \text{or} \quad
(f(x) - b^{-1} f(y))(g(x) - cg(y))\leq 0.
\end{align*}

\end{definition}

\begin{remark}
It is easy to verify that if $(f, g)$ is $(b, c)$-synchronous (resp. $(b, c)$-asynchronous), then it is also weakly $(b, c)$-synchronous (resp. weakly $(b, c)$-asynchronous). The converse does not generally hold.
However, in the special case when $b = c = 1$, all notions of synchrony and asynchrony coincide. That is, $f$ and $g$ are synchronous (resp. asynchronous) if and only if $(f, g)$ is $(1,1)$-synchronous (resp. $(1,1)$-asynchronous), if and only if $(f, g)$ is weakly $(1,1)$-synchronous (resp. weakly $(1,1)$-asynchronous).
\end{remark}

\begin{theorem}[Chebyshev's inequality II]
\label{thm:cheby-weak-sync}
Let $f,g \in L^1(X,\mu)$ with $f\geq 0$ and $g \geq 0$. Let $b>0$ and $c>0$. Define $b_{*}=\max(b,\tfrac{1}{b})$ and $c_{*}=\max(c,\tfrac{1}{c})$.

If $(f,g)$ is  weakly $(b,c)$-synchronous then we have
\begin{align*}
\int_X f(x) g(x)\,\d\mu(x) \geq \frac{b^{-1}_{*}+c^{-1}_{*}}{1+b_{*}c_{*}}\frac{1}{\mu(X)} \Big(\int_X g(x)\,\d\mu(x)\Big)\Big(\int_X f(x)\,\d\mu(x)\Big).
\end{align*}
If $(f,g)$ is weakly $(b,c)$-asynchronous then we have
\begin{align*}
\int_X f(x) g(x)\,\d\mu(x) \leq \frac{b_{*}+c_{*}}{1+b^{-1}_{*}c^{-1}_{*}}\frac{1}{\mu(X)}\Big(\int_X g(x)\,\d\mu(x)\Big)\Big(\int_X f(x)\,\d\mu(x)\Big).
\end{align*}
The special case $b = c = 1$, implies the classical Chebyshev inequality.
\end{theorem}

\begin{proof}
Assume $(f,g)$ is weakly $(b,c)$-asynchronous. For fixed $x, y \in X$, there are two possible cases. If $f(x) \leq b f(y)$  and $cg(x) \geq g(y)$ that is $(g(x) - c^{-1} g(y))(f(x) - b f(y)) \leq 0,$
then we have
\begin{align*}
g(x)f(x) + b c^{-1} g(y)f(y) \leq bg(x)f(y) +  c^{-1} g(y)f(x).
\end{align*}
If $f(y) \leq b f(x)$ and $cg(y) \geq g(x)$ then  similarly we derive
\begin{align*}
g(x)f(x) + b^{-1} c g(y)f(y) \leq b^{-1} g(x)f(y) +  c g(y)f(x).
\end{align*}
Since $f\geq 0$ and $g\geq 0$, each  case implies the  inequality
\begin{align*}
g(x)f(x)+(b_{*}\,c_{*})^{-1}g(y)f(y)\leq b_{*} g(x)f(y)+c_{*}g(y)f(x).
\end{align*}
Integrating both sides with respect to $x$ and $y$ over $X \times X$ implies
\begin{align*}
\int_X f(x) g(x)\,\d\mu(x) \leq \frac{b_{*}+c_{*}}{1+b^{-1}_{*}c^{-1}_{*}}\frac{1}{\mu(X)} \Big(\int_X g(x)\,\d\mu(x)\Big)\Big(\int_X f(x)\,\d\mu(x)\Big).
\end{align*}
Analogously, if $(f,g)$ is weakly $(b,c)$-synchronous then we have
\begin{align*}
\int_X f(x) g(x)\,\d\mu(x) \geq \frac{b^{-1}_{*}+c^{-1}_{*}}{1+b_{*}c_{*}}\frac{1}{\mu(X)}\Big(\int_X g(x)\,\d\mu(x)\Big)\Big(\int_X f(x)\,\d\mu(x)\Big).
\end{align*}
\end{proof}

\subsection{Special Chebyshev-type inequalities} Our purpose here is to analyze the situation where $X$ is a subset of $\R$. Namely we have in mind the situation $X=\mathbb{N}$ or $X=\mathbb{Z}$ for the discrete setting or simply an interval for the continuous setting. To this end we introduce the notion of almost monotonicity.
\begin{definition}
Let  $f,g: X \to [0,\infty)$ be two functions defined on a nonempty set $X \subset \mathbb{R}$ and let $b>0$, $c>0$.

\noindent We say that $f$ is \emph{$b$-almost increasing} if
\begin{align*}
x \leq y \implies f(x) \leq bf(y) \quad \text{for all } x, y \in X.
\end{align*}
\noindent We say that $f$ is \emph{$b$-almost decreasing} if
\begin{align*}
x \leq y \implies  f(y) \leq bf(x) \quad \text{for all } x, y \in X.
\end{align*}
\noindent We say that $(f,g)$ is \emph{$(b,c)$-positively monotone} if $f$ is almost $b$-increasing and $g$ is almost $c$-increasing, or if $f$ is almost $b$-decreasing and $g$ is almost $c$-decreasing.

\noindent We say that $(f,g)$ is \emph{$(b,c)$-negatively monotone} if $f$ is almost $b$-increasing and $g$ is almost $c$-decreasing, or vice versa.
\end{definition}

\begin{example}
In general for a bounded  function $ \phi:X\to [m, M]$ with $0<m<M$ and any monotone function, $f:X\to [0, \infty)$ the function $g(t)= \phi(t) f(t)$ is $c$-almost monotone with $c=Mm^{-1}$. For instance $g(t) = t^\tau(2+ \sin(\log (1+|t|)))$ or  $g(t) = t^\tau(2 + \sin(t))$, for $\tau \in \R$
is $3$-almost monotone, but not monotone.
\end{example}

\begin{remark}
\textbf{(I)} The definition of $b$-almost monotonicity automatically enforces a restriction on the parameter $b$ and the range of $f$, namely that $f$ must not change sign unless $b=1$. In fact, either $f$ does not change sign on $X$ \textup{(}with $b \geq 1$ if $f \geq 0$, and $b \leq 1$ if $f \leq 0$\textup{)}, or $b = 1$ and $f$ is ordinarily monotone. Indeed,  assume $f(x)<0<f(y)$ for some $x,y$ then by definition we would have
$f(x) \leq b f(x)$ and $f(y) \leq b f(y)$ forcing  $b=1$ and hence $f$ is monotone.
\smallskip 

\textbf{(II)} If $(f,g)$ is $(b,c)$-synchronous \textup{(}resp. $(b,c)$-asynchronous\textup{)}, then
\begin{align*}
f \text{ is $b$-almost increasing} &\iff g \text{ is $c$-almost increasing \textup{(}resp. $c$-almost decreasing\textup{)},} \\
f \text{ is $b$-almost decreasing} &\iff g \text{ is $c$-almost decreasing \textup{(}resp. $c$-almost increasing\textup{)}.}
\end{align*}
In particular, if $f$ and $g$ are both $b$- and $c$-almost monotone, respectively, then the pair $(f,g)$ is $(b,c)$-positively monotone \textup{(}resp. $(b,c)$-negatively monotone\textup{)}.
\medskip 

\textbf{(III)} If  $f$ be $b$-almost monotone and $g$ be $c$-almost monotone, then 
\begin{align*}
(f,g) \text{ is $(b,c)$-synchronous} 
&\implies (f,g) \text{ is positively  $(b,c)$-monotone}\\ 
&\implies (f,g) \text{ is weakly $(b,c)$-synchronous}, \\ 
(f,g) \text{ is $(b,c)$-asynchronous}  
&\implies (f,g) \text{ is negatively $(b,c)$-monotone} 
\\
&\implies (f,g) \text{ is weakly $(b,c)$-asynchronous}. 
\end{align*}
\end{remark}

\begin{theorem}[Chebyshev's inequality III]\label{thm:almost-cheby}
Let  $b, c\in [1,\infty)$ and $f, g\in L^1( X) $ with $X \subset \R$ such that such that $f$ is $b$-almost monotone and $g$ is $c$-almost monotone. If $(f,g)$ is positively $(b,c)$-monotone, then
\begin{align*}
\int_X f(x) g(x)\, \d\mu(x)
\geq
\frac{b^{-1}+c^{-1}}{1+bc}\frac{1}{\mu(X)}
\Big( \int_X f(x)\, \d\mu(x) \Big)
\Big( \int_X g(x)\, \d \mu(x)\Big),
\end{align*}
If $(f,g)$ is negatively $(b,c)$-monotone, then
\begin{align*}
\int_X f(x) g(x)\, \d\mu(x)
\leq
\frac{b+c}{1+b^{-1}c^{-1} }\frac{1}{\mu(X)}
\Big( \int_X f(x)\, \d\mu(x) \Big)
\Big( \int_X g(x)\, \d \mu(x)\Big).
\end{align*}
\end{theorem}

\begin{proof}
This  follows from  Theorem \ref{thm:cheby-weak-sync}. Since if $f$ and $g$ have the same monotonicity then $(f, g)$ is $(b, c)$-weakly synchronous. Otherwise, $(f, g)$ is $(b, c)$-weakly asynchronous.
\end{proof}

\begin{theorem}[Chebyshev's inequality IV]\label{thm:almost-cheby-sing}Let  $a,r\in (0,\infty)$ and $b, c\in [1,\infty)$ be fixed. Let $f, g : (0,\infty) \to [0,\infty)$ be two measurable functions such  that $f$ is $b$-almost monotone and $g$ is $c$-almost monotone. If $(f,g)$ is positively $(b,c)$-monotone, then
\begin{align*}
&\int_0^r f(t) g(t) t^{a-1} \d t
\geq
ar^{-a}\,\frac{b^{-1}+c^{-1}}{1+bc}
\Big( \int_0^r f(t) t^{a-1} \d t \Big)
\Big( \int_0^r g(t) t^{a-1} \d t \Big),\\
&\int_r^\infty f(t) g(t) t^{-a-1} \d t
\geq
ar^{a}\,\frac{b^{-1}+c^{-1}}{1+bc}
\Big( \int_r^\infty f(t) t^{-a-1} \d t \Big)
\Big( \int_r^\infty g(t) t^{-a-1} \d t \Big).
\end{align*}
If $(f,g)$ is negatively $(b,c)$-monotone, then
\begin{align*}
&\int_0^r f(t) g(t) t^{a-1} \d t
\leq
ar^{-a}\,\frac{b+c}{1+b^{-1}c^{-1}}
\Big( \int_0^r f(t) t^{a-1} \d t \Big)
\Big( \int_0^r g(t) t^{a-1} \d t \Big),\\
&\int_r^\infty f(t) g(t) t^{-a-1} \d t
\leq
ar^{a}\,\frac{b+c}{1+b^{-1}c^{-1}}
\Big( \int_r^\infty f(t) t^{-a-1} \d t \Big)
\Big( \int_r^\infty g(t) t^{-a-1} \d t \Big).
\end{align*}
\end{theorem}

\begin{proof}
It is sufficient to apply Theorem~\ref{thm:almost-cheby} on measure spaces
$(X=(0,r),\,\d\mu(t)= t^{a-1}\d t)$ and
$(X=(r,\infty),\,\d\mu(t)= t^{-a-1}\d t)$.
\end{proof}

\begin{definition}
Let $g : X \to (0,\infty)$ be a function, $X \subset \R$ and $c>1$. We say that $g$ is called
\emph{$c$-adically decreasing} if
\begin{align*}
g(ct) \leq g(t)
\quad \text{for all } t \in X \text{ such that } ct \in X .
\end{align*}
Likewise, $g$ is called \emph{$c$-adically increasing} if
\begin{align*}
g(ct) \geq g(t)
\quad \text{for all } t \in X \text{ such that } ct \in X .
\end{align*}
\end{definition}

\begin{example}
Here are some non-trivial examples of $c$-adic monotone functions.
\begin{enumerate}[(a)]
\item For $\tau>0$, function $g(t)=\frac{|\sin (t)|^\tau}{t^\tau}$, $t\in \R$ dydically increasing but not monotone.
\item The function $g(t) = x^\tau\sin^2(\log_c(t)\pi)$ with $\tau\in \R$ and $c>1$ is $c-$adically monotone but not monotone.
\item In general, let $f:X\to \R$ be monotone. If $\phi:X \to [0, \infty)$ is $c$-adically monotone  with the same monotonicity, so is the function $g(t)= \phi(t) f(t)$.
\end{enumerate}
\end{example}

\begin{theorem}[Chebyshev's inequality V]\label{thm:chebyshev-cadic}
Let  $a,r\in (0,\infty)$ and $b, c\in (1,\infty)$.
Let $f ,g: (0, \infty) \to[0,\infty)$ be two measurable functions.  Assume $g$ is $c$-adically decreasing.
If  $f$ is $b$-almost decreasing, then the following estimates hold
\begin{align*}
\int_0^r  g(t) f(t) t^{a-1} \d t
&	\geq
\frac{a(c r)^{-a}}{b^{3}} \Big( \int_0^r  g(t) t^{a-1} \d t \Big)
\Big( \int_0^{c r} f(t) t^{a-1}\d t \Big),\\
\int_r^\infty g(t) f(t) t^{-a-1} \d t
&\geq
\frac{a(c r)^{a}}{b^{3}} \Big( \int_r^\infty  g(t) t^{-a-1} \d t \Big)
\Big( \int^\infty _{c r} f(t) t^{-a-1}\d t \Big).
\end{align*}

If $f$ is $b$-almost increasing then the estimates are reversed as follows
\begin{align*}
\int_0^r  g(t) f(t) t^{a-1} \d t
&	\leq
a(c r)^{-a} b^3\Big( \int_0^r  g(t) t^{a-1} \d t \Big)
\Big( \int_0^{c r} f(t) t^{a-1}\d t \Big),\\
\int_r^\infty g(t) f(t) t^{-a-1} \d t
&\leq
a(c r)^{a} b^3\Big(  \int_r^\infty  g(t) t^{-a-1} \d t \Big)
\Big( \int^\infty _{c r} f(t) t^{-a-1}\d t \Big).
\end{align*}
Moreover, if $g$ is merely decreasing then one can simply replace $b^3$ by $b$.
\end{theorem}

\begin{proof}
Note that, if $g$ is monotone then the sought inequalities follow immediately  from Theorem \ref{thm:almost-cheby-sing}.  Due to scaling invariance, it is enough to consider the case  $r=1$. For a function $w:(0,\infty)\to[0,\infty)$, we use the following $c$-adic decomposition
\begin{align*}
\int_0^1 t^{a -1}w(t) \d t
&= \sum_{n \geq 0} \int_{c^{-(n+1)}}^{c^{-n}} t^{a -1} w(t) \d t
= \int_{\tfrac{1}{c}}^{1} \tau^{a -1} \sum_{n \geq 0} \frac{1}{c^{a n}}
w(c^{-n} \tau) \d \tau,\\
\int_1^\infty t^{-a -1} w(t)\, \d t
&= \sum_{n \geq 0} \int_{c^n}^{c^{n+1}} t^{-a -1} w(t)\, \d t
= \int_{1}^{c} \tau^{-a -1} \sum_{n \geq 0} \frac{1}{c^{a n}} w(c^n \tau)\, \d \tau.
\end{align*}
Since $f$ is $b$decreasing,  we have
$b^{-1}f(c^{-n})\leq f(c^{-n} \tau)\leq bf(c^{-n-1})$ for $\tau \in (\tfrac{1}{c},1)$,
while $b^{-1}f(c^{n+1})\leq f(c^{n} \tau)\leq bf(c^{n})$ for $\tau \in (1,c)$.
Therefore using again monotonicity and the formulas  $\int_{c^{-k}}^{c^{-(k-1)}} t^{a -1}  \d t
= \int_{c^{k-1}}^{c^{k}} t^{-a -1}  \d t
= \frac{c^{a}-1}{a c^{a k}},$ $k\in \{n,n+1\} $
we find that
\begin{align*}
\frac{ab^{-2}}{c^a-1}\int_{c^{-n}}^{c^{-(n-1)}}\hspace*{-1ex} f(t)t^{a -1} \d t &
\leq
\frac{1}{c^{a n}} f(c^{-n} \tau)\leq  	\frac{a b^2 c^{2a}}{c^a-1}\int_{c^{-n-2}}^{c^{-n-1}} \hspace*{-1ex} f(t)t^{a -1} \d t
\quad \,\, (\tau \in (\tfrac{1}{c},1)),\\
\frac{a b^{-2}c^{2a}}{c^a-1}\int_{c^{n+1}}^{c^{n+2}} \hspace*{-1ex} f(t)t^{-a -1}  \d t& \leq
\frac{1}{c^{a n}} f(c^{n} \tau)
\leq \frac{a b^2}{c^a-1} \int_{c^{n-1}}^{c^{n}} \hspace*{-1ex}f(t)t^{-a -1}  \d t \,\,\,\qquad (\tau \in (1,c)).
\end{align*}

Hence it follows that
\begin{align*}
&b^{-2}\int_{0}^{c} f(t)t^{a -1} \d t
\leq
\frac{c^a-1}{a}\sum_{n \geq 0}\frac{ f(c^{-n} \tau)}{c^{a n}}
\leq  	c^{2a}b^2\int_{0}^{c^{-1}} \hspace*{-1ex}f(t)t^{a -1} \d t
\quad  \,\, (\tau \in (\tfrac{1}{c},1)),\\
&
c^{2a}b^{-2}\int_{c}^{\infty}\hspace*{-1ex}
f(t)t^{-a -1}  \d t
\leq
\frac{c^a-1}{a} \sum_{n \geq 0}\frac{f(c^{n} \tau) }{c^{a n}}
\leq b^2\int_{c^{-1}}^{\infty} \hspace*{-1ex}f(t)t^{-a -1}  \d t
\quad (\tau \in (1,c)).
\end{align*}
Since $f $  is $b$-almost decreasing and $g(ct) \leq g(t)$, if we  fix $\tau>0$,
then the sequence $(g(c^{-n} \tau))_n$ is increasing and $(f(c^{-n} \tau))_n$ $b$-almost increasing,
while the  sequence $(g(c^{n} \tau))_n$ is decreasing and $(f(c^{n} \tau))_n$ is $b$-almost decreasing.
Whence,  applying  Theorem \ref{thm:cheby-weak-sync} for discrete sums,  with weights $\alpha_n = \frac{1}{c^{a n}}$ and using
$\sum_{n \geq 0} \frac{1}{c^{a n}} = \frac{c^{a}}{c^{a}-1}$, we obtain
\begin{align*}
\sum_{n \geq 0} \frac{1}{c^{a n}} g(c^{-n} \tau)f(c^{-n} \tau)
&\geq  \frac{c^a-1}{bc^a}\Big( \sum_{n \geq 0} \frac{1}{c^{a n}} g(c^{-n} \tau) \Big)\Big(\sum_{n \geq 0} \frac{1}{c^{a n}} f(c^{-n} \tau)\Big),\\
\sum_{n \geq 0} \frac{1}{c^{a n}} g(c^{n} \tau)f(c^{n} \tau)
&\geq  \frac{c^a-1}{bc^a} \Big(\sum_{n \geq 0} \frac{1}{c^{a n}} g(c^{n} \tau)\Big) \Big(\sum_{n \geq 0} \frac{1}{c^{a n}} f(c^{n} \tau)\Big).
\end{align*}
By exploiting these estimates inn conjunction with the $c$-adic decomposition we obtain
\begin{align*}
\int_0^1g(t)f(t) t^{a -1} \d t
&\geq   \frac{c^a-1}{bc^a} \int_{\tfrac{1}{c}}^{1} \tau^{a -1} \Big(\sum_{n \geq 0} \frac{1}{c^{a n}} g(c^{-n} \tau) \Big)\Big( \sum_{n \geq 0} \frac{1}{c^{a n}} f(c^{-n} \tau)\Big) \d \tau\\
&\geq \frac{a}{b^{3}c^a} \int_{\tfrac{1}{c}}^{1} \tau^{a -1} \Big(\sum_{n \geq 0} \frac{1}{c^{a n}} g(c^{-n} \tau)\Big)  \d \tau\,  \Big(\int_0^c f(t)t^{a -1} \d t\Big)\\
&= \frac{a}{b^{3}c^a}	\Big(\int_0^1g(t)t^{a -1} \d t\Big) \Big(\int_0^c f(t)t^{a -1} \d t\Big),
\end{align*}
while, similarly,  for the second case we get
\begin{align*}
\int_1^\infty g(t)t^{-a -1} f(t) \d t
&\geq     \frac{c^a-1}{b^{3}c^a} \int_{1}^{c} \tau^{-a -1}
\Big(\sum_{n \geq 0} \frac{1}{c^{a n}} g(c^{n} \tau)  \Big)
\Big(\sum_{n \geq 0} \frac{1}{c^{a n}} f(c^{n} \tau)\Big) \d \tau\\
&\geq \frac{a c^a}{b^{3}} \int_{1}^{c} \tau^{-a -1}\Big( \sum_{n \geq 0} \frac{1}{c^{a n}} g(c^{n} \tau) \Big)  \d \tau\, \Big(\int_c^\infty  f(t)t^{-a -1}  \d t\Big)\\
&= \frac{a c^a}{b^{3}}\Big(\int_1^\infty g(t)t^{-a -1} \d t \Big)
\Big(\int_c^\infty  f(t)t^{-a -1}  \d t\Big).
\end{align*}
\end{proof}

\begin{theorem}[Chebyshev's inequality V$'$]\label{thm:chebyshev-cadic-bis}
Let  $a,r\in (0,\infty)$ and $b, c\in (1,\infty)$.
Let $f ,g: (0, \infty) \to[0,\infty)$ be two measurable functions. Assume $g$ is $c$-adically increasing

If  $f$ is $b$-almost increasing, then the following estimates hold
\begin{align*}
\int_0^r  g(t) f(t) t^{a-1} \d t
&	\geq
\frac{a(r/c)^{-a}}{b^3}\Big( \int_0^r  g(t) t^{a-1} \d t \Big)
\Big( \int_0^{r/c} f(t) t^{a-1}\d t \Big),\\
\int_r^\infty g(t) f(t) t^{-a-1} \d t
&\geq
\frac{a(r/c)^{a}}{b^3} \Big(  \int_r^\infty  g(t) t^{-a-1} \d t \Big)
\Big( \int^\infty _{r/c} f(t) t^{-a-1}\d t \Big).
\end{align*}
If $f$ is $b$-almost decreasing then the estimates are reversed as follows
\begin{align*}
\int_0^r  g(t) f(t) t^{a-1} \d t
&	\leq
a(r/c)^{-a}b^{3} \Big( \int_0^r  g(t) t^{a-1} \d t \Big)
\Big( \int_0^{r/c} f(t) t^{a-1}\d t \Big),\\
\int_r^\infty g(t) f(t) t^{-a-1} \d t
&\leq
a(r/c)^{a}b^{3} \Big( \int_r^\infty  g(t) t^{-a-1} \d t \Big)
\Big( \int^\infty _{r/c} f(t) t^{-a-1}\d t \Big).
\end{align*}
Moreover, if $g$ is merely monotone, then one can simply replace $b^3$ by $b$.
\end{theorem}

\begin{proof}
Since $\tilde{g}(t) = g(1/t)$ is $c$-adically decreasing, apply Theorem~\ref{thm:chebyshev-cadic} for $a>0$, $1/r>0$ and the function $\tilde{f}(t) = f(1/t)$ is $b$-almost monotone oppositely to $f$.
\end{proof}

\begin{remark}
Chebyshev-type inequalities similar to those in Theorem~\ref{thm:chebyshev-cadic} and Theorem~\ref{thm:chebyshev-cadic-bis} can also be obtained in more complex settings, where the function $g:X\to[0,\infty)$ is simultaneously $c$-adically and $b$-almost monotone.
For instance, assume that $g:X\to[0,\infty)$ is both $c$-adically decreasing and $b$-almost decreasing, meaning that
$g(ct)\leq b\,g(t)$  for all $t\in X.$
Setting $q=\frac{\log b}{\log c},$
so that $b=c^{q}$, the function $t\mapsto g_q(t)=\frac{g(t)}{t^{q}}$
is $c$-adically decreasing that is, $g_q(ct)\leq g_q(t)$  for all $t\in X.$ Hence, Theorem~\ref{thm:chebyshev-cadic} and Theorem~\ref{thm:chebyshev-cadic-bis} apply to $g_q$ as well. We leave the derivation of the corresponding inequalities to the interested reader.
\end{remark}

\subsection{Mononicity of singular integrals}
\label{subsec:singular-integral}
In this section, we investigate various monotonicity properties involving integral operators. In the next result, we shall see that Chebyshev's inequality surprisingly yields consequences for families of density functions with the Monotone Likelihood Ratio (MLR) property, a concept of genuine significance in statistics and probability theory.  The study of the monotone likelihood ratio property dates back at least to the foundational statistical work by Karlin \& Rubin \cite{KaRu56}; a comprehensive exposition can be found in Lehmann \& Romano \cite{LeRo21}.

\begin{theorem}
\label{thm:mlr-almost}
Let $(X, \mu)$ be a measure space with $X\subset \R$ and $I\subset \R$ be a  nonempty set of indices. Let $b\geq 1$ and $c\geq 1$.  Assume $(p_\alpha)_{\alpha\in I}$ is a family of probability densities on $X$ with $b$-almost monotone likelihood ratio property in $x$, that is,  for each $\alpha<\beta$ the ratio mapping $x\mapsto g_{\alpha,\beta}(x)= \frac{p_\beta(x)}{p_\alpha(x)} $ is $b$-almost increasing on $X$. Let $f ,g: X\to [0,\infty]$ be such that $f$ is $c$-almost increasing and $g$ is $c$-almost decreasing. Then we have
\begin{align*}
 \begin{split}
\int_X f(x)p_\beta(x)\d\mu(x)
&\geq \frac{b^{-1}+c^{-1}}{1+bc}
\int_X f(x)p_\alpha (x)\d \mu(x),\\
\int_X g(x)p_\beta(x)\d\mu(x)
&\leq \frac{b+c}{1+b^{-1}c^{-1} }
\int_Xg(x)p_\alpha (x)\d \mu(x)
\end{split} && (\alpha,\beta\in I,\,\, \alpha<\beta).
\end{align*}
In other words,  if $\mathbf{X}_\alpha$ is a random variable with law $p_\alpha(x)\d \mu(x)$ then  the mappings $\alpha\mapsto \mathbb{E}[f(\mathbf{X}_\alpha)]$ and  $\alpha\mapsto \mathbb{E}[g(\mathbf{X}_\alpha)]$ are almost monotone on $I$ and we have
\begin{align*}
 \begin{split}
\mathbb{E}[f(\mathbf{X}_\beta)]&\geq\frac{b^{-1}+c^{-1}}{1+bc} \mathbb{E}[f(\mathbf{X}_\alpha )]
\quad \text{and}\quad \mathbb{E}[g(\mathbf{X}_\beta)] \leq  \frac{b+c}{1+b^{-1}c^{-1} } \mathbb{E}[g(\mathbf{X}_\alpha)]
\end{split}&&(\alpha<\beta).
\end{align*}
Analogously, if $x\mapsto g_{\alpha,\beta}(x)= \frac{p_\beta(x)}{p_\alpha(x)}$
is $b$-almost decreasing, then
\begin{align*}
 \begin{split}
\mathbb{E}[g(\mathbf{X}_\beta)]&\geq\frac{b^{-1}+c^{-1}}{1+bc} \mathbb{E}[g(\mathbf{X}_\alpha )]
\quad \text{and}\quad \mathbb{E}[f(\mathbf{X}_\beta)] \leq  \frac{b+c}{1+b^{-1}c^{-1} } \mathbb{E}[f(\mathbf{X}_\alpha)]
\end{split}&&(\alpha<\beta).
\end{align*}
\end{theorem}
\begin{proof}
The result follows by  applying Theorem~\ref{thm:almost-cheby} to the couples $(f,g_{\alpha,\beta})$ and $(g,g_{\alpha,\beta})$ with respect to the measure $\d \mu_\alpha(x)= p_\alpha(x)
 \d \mu(x)$; noting that
\begin{align*}
\int_X p_\alpha(x) \d \mu(x) = 1= \int_X p_\beta(x) \d \mu(x) = \int_X g_{\alpha,\beta}(x) p_\alpha(x) \d \mu(x).
\end{align*}
\end{proof}

The next result is a direct consequence of Theorem \ref{thm:mlr-almost} and shows that the monotone likelihood ratio property implies the  monotonicity of the expectations. 
\begin{theorem}[Monotone Likelihood Ratio]
\label{thm:mon-like-ratio} Let $(X, \mu)$ be a measure space with $X\subset \R$ and $I\subset \R$ be a  nonempty set. Assume $(p_a)_{a\in I}$ is a family of probability densities on $X$ with monotone likelihood ratio in $x$, that is,  for each $a<b$ the ratio mapping $x\mapsto \frac{p_b(x)}{p_a(x)} $ is increasing.
Let $c\geq1$ and $f ,g: X\to [0,\infty]$ be such that $f$ is $c$-almost increasing and $g$ is $c$-almost decreasing. Then we have
\begin{align*}
 \begin{split}
\int_X f(x) p_a(x)\d \mu(x)&\leq c\int_X f(x) p_b(x)\d \mu(x)\\
\int_X g(x) p_b(x)\d \mu(x)&\leq c\int_X g(x) p_a(x)\d \mu(x)
\end{split} && (a,b\in I,\,\, a<b).
\end{align*}
In other words,  if $\mathbf{X}_a$ is a random variable with law $p_a(x)\d \mu(x)$ then  the mappings $a\mapsto \mathbb{E}[f(\mathbf{X}_a)]$  and $a\mapsto \mathbb{E}[g(\mathbf{X}_a)]$ is $c$-almost monotone on $I$, and we have
\begin{align*}
\mathbb{E}[f(\mathbf{X}_a)]\leq c\mathbb{E}[f(\mathbf{X}_b)]\quad\text{and}\quad  \mathbb{E}[g(\mathbf{X}_b)]\leq c\mathbb{E}[g(\mathbf{X}_a)] && (a,b\in I,\,\,a<b).
\end{align*}
\end{theorem}

\begin{proof}
The claim follows by taking $b=1$ in Theorem~\ref{thm:mlr-almost}.  For the special case $c = 1$, see the different approach in \cite[Chapter 3]{LeRo21}.
\end{proof}

The next lemma provides an Abelian-type averaging result, establishing the limiting behavior of some singular integrals near $0$ and $\infty$. 
\begin{lemma}
[Averaging  near $0$ and $\infty$]
\label{lem:asymp-averag}
Let $r > 0$ and  $g\in L^q_{\loc}((0,\infty))$, with $1<q\leq \infty$, $g: (0, \infty)\to \R$.
If $\lim\limits_{t \to 0^+} g(t) = g(0)$ exists in $\R\cup \{ -\infty, \infty\}$, then we have
\begin{align*}
&\lim_{a\to0^+} a\int_0^r g(t) t^{a-1}\d t = g(0).
\end{align*}
If $\lim\limits_{t \to \infty} g(t) = g(\infty)$ exists in $\R\cup \{ -\infty, \infty\}$, then we have
\begin{align*}
&\lim_{a\to0^+} a\int_r^\infty g(t) t^{-a-1}\d t= g(\infty).
\end{align*}
\end{lemma}

\begin{proof}
Since $q\in (1,\infty]$, we put $q'= \frac{q}{q-1}\in [1,\infty)$.  Let us start with the cases where $g(0)$ and $g(\infty)$ are finite.  
Given $\varepsilon > 0$, there exists $0 < \delta < r$ such that $|g(t) - g(0)| < \varepsilon$ for all $ t \in (0, \delta).
$   Then for $a>0$ H\"older inequality implies
\begin{align*}
a\Big|\int_0^r g(t)t^{a-1} \d t &- \frac{g(0)}{a} r^{a} \Big|
= a \Big|\int_0^r (g(t) - g(0)) t^{a-1} \d t \Big|\\
&\leq \Big(a \int_0^\delta + a \int_\delta^r \Big)|g(t) - g(0)| t^{a-1} \d t \\
&\leq \varepsilon \delta^a + \|g-g(0)\|_{L^q( (\delta,r))}
\Big(\frac{a(r^{(a-1)q'+1} - \delta^{(a-1)q'+1})}{(a-1)q'+1}\Big)^{1/q'}.
\end{align*}
Letting $a \to 0^+$, we get
\begin{align*}
\limsup_{a \to 0^+} \Big|a\int_0^r g(t)t^{a-1} \d t - g(0)r^a \Big|\leq \varepsilon.
\end{align*}
Thus, using again the fact that $r^a \to 1$, we conclude
\begin{align*}
\lim_{a\to0^+} a\int_0^r g(t)t^{a-1}\d t = g(0).
\end{align*}
Analogously, given $\varepsilon > 0$, if  $ g(t) \to g(\infty)$ as $t\to\infty$, so there exists $R > r$ such that
$|g(t) - g(\infty)| < \eps$  for all  $t > R.$ Then for $a>0$ H\"older inequality implies

\begin{align*}
a \Big|\int_r^\infty g(t)t^{-a-1}\d t &- \frac{g(\infty)}{a} r^{-a} \Big|
= a \Big|\int_r^\infty (g(t) - g(\infty))t^{-a-1}\d t \Big|\\
&\leq \Big(a \int_r^R+ a \int_R^\infty  \Big) |g(t) - g(\infty)|t^{-a-1}\d t\\
&\leq \|g -g(\infty)\|_{L^q( (r,R))}\Big(\frac{a\big( r^{(-a-1)q'+1}-R^{(-a-1)q'+1} \big)}{(a+1)q'-1}\Big)^{1/q'}
+ \varepsilon R^{-a}.
\end{align*}
Letting $a\to 0^+$, we obtain
\begin{align*}
\limsup_{a\to 0^+}  \Big| a\int_r^\infty g(t)t^{-a-1}\d t - g(\infty) r^{-a} \Big| \leq \varepsilon.
\end{align*}
Hence, since  $r^{-a} \to 1$, we conclude
\begin{align*}
\lim_{a\to0^+} a\int_r^\infty g(t)t^{-a-1}\d t = g(\infty).
\end{align*}
Last upon replacing $g$ with $-g$, it is sufficient to deal with cases $g(0)= \infty$ and $g(\infty)=\infty $. As previously we can choose $0<\delta<r$ and $R>r$ such that
\begin{align*}
&g(t)\geq 1/\eps \quad\text{for all $0<t<\delta$},\quad \text{resp.}\quad g(t)\geq 1/\eps \quad\text{for all $t>R$}.
\end{align*}
By splitting the integrals as previously, the following estimates hold respectively
\begin{align*}
a \int_0^r g(t) t^{a-1} \,\d t
&\geq \frac{\delta^a}{\varepsilon} - a \|g\|_{L^q((\delta,r))} \Big( \frac{r^{(a-1)q'+1} - \delta^{(a-1)q'+1}}{(a-1)q'+1} \Big)^{1/q'}, \\[1.5ex]
a \int_r^\infty g(t) t^{-a-1} \,\d t
&\geq \frac{R^{-a}}{\varepsilon} - a \|g\|_{L^q((r,R))} \Big( \frac{r^{(-a-1)q'+1} - R^{(-a-1)q'+1}}{(-a-1)q'+1} \Big)^{1/q'}.
\end{align*}
From these, we deduce
\begin{align*}
&\liminf_{a\to 0^+}  a\int_0^r g(t)t^{-a-1}\d t \geq \frac{1}{\varepsilon},\quad \text{while}\quad
\liminf_{a\to 0^+}  a\int_r^\infty g(t)t^{-a-1}\d t \geq \frac{1}{\varepsilon}.
\end{align*}
Since $\eps>0$ was arbitrary, we conclude:
\begin{align*}
&\lim_{a\to0^+} a\int_0^r g(t)t^{a-1}\d t = g(0)=\infty
\quad \text{and}\quad \lim_{a\to0^+} a\int_r^\infty g(t)t^{-a-1}\d t = g(\infty)=\infty.
\end{align*}
\end{proof}

Motivated by Lemma \ref{lem:asymp-averag} can introduce  the following definition.
\begin{definition}\label{def:tail-sing-int}
Given  $r>0$ and a measurable function $g: (0,\infty)\to [0,\infty]$ we define the function $T,L: [0,\infty) \to[0,\infty]$ with
\begin{align*}
L(z)&= zr^{-z}\int_0^r  g(t)t^{z-1} \d t\qquad\text{with}\qquad L(0)= g(0), \\
T(z)&= zr^{z} \int_r^\infty g(t) t^{-z-1} \d t\qquad\text{with}\qquad T(0)=g(\infty).
\end{align*}
The conventions $L(0) = g(0)$ and $T(0) = g(\infty)$ are justified by Lemma~\ref{lem:asymp-averag}.
\end{definition}

\begin{theorem}[Monotonicity of singular integrals I]
\label{thm:almost-monotone}
Let $r>0$ and $g: (0,\infty)\to [0,\infty]$ be measurable. If $g$ is almost monotone, then so are $L(z)$ and $T(z)$ defined in Definition~\ref{def:tail-sing-int}.  Namely if $g$ is $c$-almost increasing, $c \geq 1$, then we have
\begin{align*}
L(a)\leq cL(b)\qquad \text{and}\qquad T(b)\leq cT(a) && 0\leq a\leq b.
\end{align*}
Analogously, if $g$ is $c$-almost decreasing, $c \ge 1$, then we have
\begin{align*}
L(b)\leq cL(a)\qquad \text{and}\qquad T(a)\leq cT(b)&& 0\leq a\leq b.
\end{align*}
\end{theorem}

\begin{proof}
One can either apply Theorem~\ref{thm:mon-like-ratio} using the probability densities $p_a(t)= ar^{-a} t^{a-1}\mathds{1}_{(0,r)}(t)$ and $p^a(t)= ar^{a} t^{-a-1}\mathds{1}_{(r,\infty)}(t)$ or  apply Theorem~\ref{thm:almost-cheby-sing} with $f(t)=t^{a-b}$ which is decreasing on $(0,\infty)$
while $ f(t) t^{b-1}=  t^{a-1}$ and $ f(t) t^{-a-1}=  t^{-b-1}$.
\end{proof}
\begin{theorem}[Monotonicity of singular integrals II]\label{thm:cadic-monotonicity}
If a measurable function $g: (0,\infty)\to [0,\infty]$ is $c$-adically monotone ($c>1$), then the maps $z\mapsto c^{-z}L(z)$ and $z\mapsto c^zT(z)$, with $L$ and $T$ as in Definition~\ref{def:tail-sing-int}, are monotone. Namely, if  $g$ is $c$-adically decreasing, we have
\begin{align*}
c^{-b}L(b)\leq c^{-a}L(a)\quad \text{and}\quad  c^{a}T(a)\leq c^{b}T(b)&&0\leq a\leq b.
\end{align*}
If $g$ is $c$-adically increasing, then the inequalities are reversed, i.e., we have
\begin{align*}
c^{-a}L(a)\leq c^{-b}L(b)\quad \text{and}\quad c^{b}T(b)\leq c^{a}T(a)&&0\leq a\leq b.
\end{align*}
\end{theorem}

\begin{proof}
Both estimates are implied by  Theorem~\ref{thm:chebyshev-cadic} since $f(t)=t^{a-b}$ is decreasing  $(0,\infty)$, $g(t) f(t) t^{b-1}= g(t) t^{a-1}$ while $g(t) f(t) t^{-a-1}= g(t) t^{-b-1}$ and
\begin{align*}
\int_0^{c r} f(t)t^{b-1}\d t= \frac{(cr)^{a}}{a}
\quad\text{and}\quad
\int_{c r}^\infty f(t)t^{-a-1}\d t =  \frac{(cr)^{-b}}{b}.
\end{align*}
The same analogy can applied to  $f(t)=t^{b-a}$ which is increasing on $(0,\infty)$.
\end{proof}

\section{Asymptotics of first-order differences in \texorpdfstring{$L^p$}{Lp}-space}\label{sec:Lp-moduos-of-cont}
In this section we  study, for $u\in L^p(\R^d)$, the asymptotic of the averaged $L^p$-modulus of continuity function
$U: (0,\infty)\to  (0,\infty)$ given by
\begin{align*}
U(t) &= \int_{\mathbb{S}^{d-1}} \int_{\R^d} |u(x + t\omega) - u(x)|^p \d x \d\sigma_{d-1}(\omega).
\end{align*}
\subsection{Funk--Radon transform} 
We will make use of the spherical slice transform identity, via  the so called Funk--Radon transform introduced by Paul Funk in \cite{Fun11,Fun13}, and which has since undergone significant development; see, for instance, \cite{Rub15,Que20,Que20-bis}. The Funk--Radon transform (also known as the Minkowski–Funk or spherical Radon transform) is an integral operator is defined  in its normalized form for a function $g : \mathbb{S}^{d-1} \to \R$ by
\begin{align*}
\mathbf{F}g(v) := \int_{\mathbb{S}^{d-2}_v} g(w) \frac{\d\sigma_{d-2}(w)}{|\mathbb{S}^{d-2}|} \qquad (v\in \mathbb{S}^{d-1}),
\end{align*}
where $\mathbb{S}^{d-2}_v = \{ w \in \mathbb{S}^{d-1} : \langle v, w \rangle = 0 \}$ denotes the great subsphere orthogonal to $v$.
The following result is likely well-known in the literature; however, as we were unable to find a precise reference, we provide a proof for the convenience of the reader.
\begin{theorem}[Slicing method]
\label{thm:slicing-integ-on-sphere}
For $g \in L^1( \mathbb{S}^{d-1})$ with $d \geq 2$ there holds that
\begin{align*}
\int_{\mathbb{S}^{d-1}} \hspace*{-1ex}g(w) \d\sigma_{d-1}(w)
=
\int_{\mathbb{S}^{d-1}} \hspace*{-1ex}
\Big(
\int_{\mathbb{S}^{d-2}_v} g(w) \frac{\d\sigma_{d-2}(w)}{|\mathbb{S}^{d-2}|}
\Big)
\d\sigma_{d-1}(v)= \int_{\mathbb{S}^{d-1}}\hspace*{-1ex} \mathbf{F}g(v) \d\sigma_{d-1}(v).
\end{align*}
\end{theorem}

\begin{proof}
it suffices to prove the claim for indicator functions
$g = \mathds{1}_A$, where  $A \subset \mathbb{S}^{d-1} $ is a Borel measurable set. This can be done by considering the measure defined by
\begin{align*}
\mu(A)=  \int_{\mathbb{S}^{d-1}} \sigma_{d-2}(A\cap  \mathbb{S}^{d-2}_v ) \frac{ \d \sigma_{d-1}(v)}{|\mathbb{S}^{d-2}|}.
\end{align*}
Given that the Hausdorff measures $\sigma_{d-1}$ and $\sigma_{d-2}$ are respectively rotation invariant, so does the measure $\mu$ on $\mathbb{S}^{d-1}$. Indeed, for each $v\in \mathbb{S}^{d-1}$ and $O\in \mathcal{O}(d)$, an orthogonal mapping on $\mathbb{R}^d$, since $O^\top  O= I_d$, one finds that
\begin{align*}
O(A) \cap  \mathbb{S}^{d-2}_v  = O(A \cap  \mathbb{S}^{d-2}_{O^\top (v)}).
\end{align*}
The rotation invariance of $\sigma_{d-2}$ implies
\begin{align*}
\sigma_{d-2}(O(A)\cap  \mathbb{S}^{d-2}_v ) = \sigma_{d-2}(O(A\cap  \mathbb{S}^{d-2}_{O^\top (v)})) = \sigma_{d-2}(A\cap  \mathbb{S}^{d-2}_{O^\top (v)}).
\end{align*}
The rotation invariance of $\sigma_{d-1}$ yields that
\begin{align*}
\int_{\mathbb{S}^{d-1}} \sigma_{d-2}(O(A)\cap  \mathbb{S}^{d-2}_v ) \frac{ \d \sigma_{d-1}(v)}{|\mathbb{S}^{d-2}|} &=  \int_{\mathbb{S}^{d-1}} \sigma_{d-2}(A\cap  \mathbb{S}^{d-2}_{O^\top (v)}) \frac{ \d \sigma_{d-1}(v)}{|\mathbb{S}^{d-2}|}\\
&=     \int_{\mathbb{S}^{d-1}} \sigma_{d-2}(A\cap  \mathbb{S}^{d-2}_{v}) \frac{ \d \sigma_{d-1}(O(v))}{|\mathbb{S}^{d-2}|},
\end{align*}
in other words $\mu(O(A)) =  \mu(A) $. It follows that, $\mu= c\sigma_{d-1}$ since rotation invariant Borel measures on $\mathbb{S}^{d-1} $ are unique up to a multiplicative constant.  On the other hand, once again the rotation invariance of $\sigma_{d-2}$ implies that $\sigma_{d-2}(\mathbb{S}^{d-1}\cap \mathbb{S}^{d-2}_v)  =  \sigma_{d-2}(\mathbb{S}^{d-2})$ so  that $\mu(\mathbb{S}^{d-1})= \sigma_{d-1} (\mathbb{S}^{d-1}) =|\mathbb{S}^{d-1}|$, yielding  $\mu= \sigma_{d-1}$.
\end{proof}

We need the following preparation lemma borrowed from \cite{Pon04}.

\begin{lemma}\label{lem:lp-diff-growth-spher}  Assume $d\geq 2$. For all $u\in L^p(\R^d)$, $s>0$ and $\theta\in [0,1]$ we have
\begin{align*}
U(\theta s)\leq 2^pU(s).
\end{align*}
\end{lemma}

\begin{proof}
For $u\in L^p(\R^d)$ define $U_*(h) = \|u(\cdot+h)-u\|^p_{L^p(\R^d)}$. 
Thus, we must show 
\begin{align*}
\int_{\mathbb{S}^{d-1}} U_*(\theta sw)\,\d \sigma_{d-1}(w)\leq 2^{p} \int_{\mathbb{S}^{d-1}} U_*(sw)\d \sigma_{d-1}(w).
\end{align*}
\noindent If $d$ is even then using the block matrices
$\pm\big(\begin{smallmatrix}
0&-1\\1&0
\end{smallmatrix}\big)$ we can  construct a rotation $O\in \mathcal{O}(d)$ on $\R^d$,  such that $Ow \cdot w= 0$ for all $w\in \R^d$. Equivalently $O$ is skew symmetric, i.e., $O^T=-O$.  Indeed, since $ w^\top Ow =w^\top O^\top w$ we deduce $Ow\cdot w= w^\top  O w = \frac{1}{2} w^\top  (O + O^\top ) w= w^\top O^{\operatorname{sym}}w$.
Hence $Ow\cdot w=0$ if and only if $w^\top O^{\operatorname{sym}} w=0$. Thus the symmetric part  $O^{\operatorname{sym}} $ of $O$ vanishes, i.e., $O^\top =-O$. Thereupon, consider
\begin{align*}
	O_+ w = \frac{\theta}{2} w+ \frac{\theta'}{2} Ow\quad\text{and}\quad O_- w = \frac{\theta}{2} w- \frac{\theta'}{2}Ow\;\;\text{ with  $\theta' = \sqrt{4-\theta^2}$}.
\end{align*}
It appears that  $\theta w= O_-w+O_+w$ and  $O_-,O_+ \in \mathcal{O}(d)$ are rotations. Indeed,  since
$|Ow|= |w|$ and $Ow.\cdot w=0$ we have
\begin{align*}
| O_\pm w |^2= \frac{\theta^2}{4}|w|^2+ \frac{\theta'^2}{4}|Ow|^2 \pm \frac{\theta\theta'}{2}O w\cdot w= |w|^2.
\end{align*}

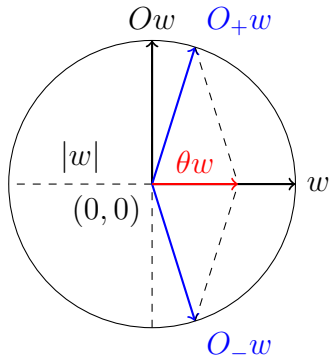
\begin{figure}[ht!]
\centering
\begin{tikzpicture}[scale=1.9]
\def\r{1}
\def\thetaval{0.6} 
\pgfmathsetmacro{\thetap}{sqrt(4-\thetaval*\thetaval)}
\coordinate (O) at (0,0);
\coordinate (w) at (\r,0);
\coordinate (Ow) at (0,\r);
\coordinate (Om) at ({\thetaval/2*\r},{\thetap/2*\r});
\coordinate (Op) at ({\thetaval/2*\r},{-\thetap/2*\r});
\coordinate (Sum) at ({\thetaval*\r},0);
\coordinate (Left) at (-\r,0);
\coordinate (Below) at (0,-\r);
\coordinate (Above) at (0,\r);
\draw[thin] (O) circle (\r);
\draw[dashed] (Below) -- (Above);
\draw[dashed] (O) -- (Left);
\node[above] at (-\r/2,0) {$|w|$};
\draw[->, thick] (O) -- (w) node[right] {$w$};
\draw[->, thick] (O) -- (Ow) node[above] {$Ow$};
\draw[->, thick, blue] (O) -- (Om) node[above right] {$O_+ w$};
\draw[->, thick, blue] (O)-- (Op) node[below right] {$O_- w$};
\draw[dashed] (Om) -- (Sum);
\draw[dashed] (Op) -- (Sum);
\draw[->, thick, red] (O) -- (Sum) node[above] at (\thetaval/2,0) {$\theta w$};
   \draw (O) node[below left] {$(0,0)$};
\end{tikzpicture}
\vspace{-4mm}
\caption{Two-dimensional illustration with $\theta = 0.6$.}
\label{fig:theta-construction}
\end{figure}
Given that $\d \sigma_{d-1}$ is invariant under rotation, we obtain
\begin{align*}
\int_{\mathbb{S}^{d-1}}
U_*(\theta sw)\d \sigma_{d-1}(w)
&\leq 2^{p-1} \int_{\mathbb{S}^{d-1}}
U_*(sO_-w)+ U_*(sO_+w)\d \sigma_{d-1}(w) \\
&= 2^{p} \int_{\mathbb{S}^{d-1}}
U_*(sw)\, \d \sigma_{d-1}(w).
\end{align*}
\noindent If $d$ is odd then $d-1\geq 2$ is even thus the previous case reveals that
\begin{align*}
\int_{\mathbb{S}^{d-2}_v}
U_*(\theta sw)\d \sigma_{d-2}(w)\leq 2^{p} \int_{\mathbb{S}^{d-2}_v} U_*(sw)  \d \sigma_{d-2}(w)\quad\text{for all $v\in \mathbb{S}^{d-1}$}.
\end{align*}
Thus,  the slicing integration formula from Theorem \ref{thm:slicing-integ-on-sphere}  yields
\begin{align*}
\int_{\mathbb{S}^{d-1}} U_*(\theta sw) \d \sigma_{d-1}(w)
&= \int_{\mathbb{S}^{d-1}}
\hspace*{-1ex}\Big(
\int_{\mathbb{S}^{d-2}_v} \hspace*{-1ex} U_*(\theta sw) \frac{\d\sigma_{d-2}(w)}{|\mathbb{S}^{d-2}|}
\Big)
\d\sigma_{d-1}(v)\\
&\leq 2^{p} \int_{\mathbb{S}^{d-1}}
\hspace*{-1ex}\Big(\int_{\mathbb{S}^{d-2}_v} \hspace*{-1ex}U_*(sw) \frac{\d\sigma_{d-2}(w)}{|\mathbb{S}^{d-2}|}
\Big) \d\sigma_{d-1}(v)\\
&= 2^{p}\int_{\mathbb{S}^{d-1}} U_*(sw)
\d \sigma_{d-1}(w).
\end{align*}
\end{proof}

\subsection{Basics Properties of $U$} We now gather several properties of the function
$U$ that will play a crucial role in the subsequent analysis.
\begin{theorem}\label{thm:lp-first-order-diff}
Let $u\in L^p(\R^d)$ with $1\leq p<\infty$. Let  $AU(t)=\frac{1}{t}\int_0^t U(z)\d z$  be the  Hardy average of $U$ and $g(t)= U(t)/t^p$.
The following assertions are true.
\begin{enumerate}
\item \label{itme:lp-diff1}   For all $s, t>0$ we have $U(-t)=U(t)$, $AU(-t)= AU(t)$ and
\begin{align*}
U(t+s)\leq2^{p-1} ( U(t)+U(s))&\qquad\text{and}\qquad U(t)\leq 2^p AU(t), \\
U(\ell t) \leq \ell^pU(t)&\qquad \text{and} \qquad g(\ell t) \leq g(t)\qquad\text{ for all \, $\ell\in \mathbb{N}$ }
\end{align*}
\item  \label{itme:lp-diff2} For $0<s\leq t$ such that $t/s= \ell +\theta$ with $\theta\in [0,1]$ and $\ell \in \mathbb{N}$ we have
\begin{align*}
\frac{U(t)}{t^p}&\leq 2^{p-1}\Big(\frac{U(s)}{s^p}+ \frac{U(\theta s)}{t^p}\Big), \\
\frac{U(t)}{t^p}&\leq 4^{p-1}(1+\ell^p)\Big(\frac{U(s)}{s^p}+  \frac{U\big((1-\tfrac{\theta }{\ell})s\big) }{t^p}\Big).
\end{align*}
\begin{align*}
U(s)\leq 2^{p+1}AU(t)\quad\text{and}\quad
AU(s)\leq 2^{p+1}AU(t).
\end{align*}
\item \label{itme:lp-diff2-bis}  Assume that  $d\geq 2$. Then for $0<s\leq t$ we have
\begin{align*}
&U(s)\leq 2^pU(t)\qquad \text{and}\qquad  \frac{U(t)}{t^p}\leq 4^{p}\frac{U(s)}{s^p},\\
&U(t+s)\leq 2^{p-1}(U(t)+U(s))\leq 4^p U(t),\\
&2^{-p} U(t)\leq AU(t)\leq 2^p U(t).
\end{align*}
\item  \label{itme:lp-diff3} If $u\in L^p(\R^d)$ then $AU(t),\,U(t)\leq 2^p |\mathbb{S}^{d-1}| \|u\|^p_{L^p(\R^d)}$. Moreover,  $U$, $AU$ and $ g(t)=\frac{U(t)}{t^p}$ are continuous on $(0,\infty)$ and
\begin{align*}
&\lim_{t \to 0^+} AU(t)=\lim_{t \to 0^+} U(t)=0,\\
& \lim_{t \to \infty} AU(t)=\lim_{t \to \infty} U(t)=2|\mathbb{S}^{d-1}|\|u\|^p_{L^p(\R^d)}.
\end{align*}
In fact, for almost all $\omega\in \mathbb{S}^{d-1}$, there holds that
\begin{align*}
\lim_{t \to \infty} \|u(\cdot+ t\omega)\pm u\|^p_{L^p(\R^d)}=2\|u\|^p_{L^p(\R^d)}.
\end{align*}
\item  \label{itme:lp-diff4} If $u\in W^{1,p}(\R^d)$ then $(p+1)AU(t),\, U(t)\leq |\mathbb{S}^{d-1}| K_{d,p}\|\nabla u\|^p_{L^p(\R^d)}t^p $ and
\begin{align*}
 \lim_{t \to 0^+} \, (p+1)\frac{AU(t)}{t^p}=\lim_{t \to 0^+} \frac{U(t)}{t^p}=K_{d,p}|\mathbb{S}^{d-1}|\|\nabla u\|^p_{L^p(\R^d)}.
\end{align*}
\end{enumerate}
\end{theorem}
\begin{proof}
\eqref{itme:lp-diff1} The parity of  $U$ and $AU$ is easy to verify. By Convexity  we obtain
\begin{align*}
U(t+r) &= \int_{\mathbb{S}^{d-1}} \int_{\R^d} |u(x + (t+r)\omega) - u(x)|^p \d x \d\sigma_{d-1}(\omega)\\
&\leq 2^{p-1}\int_{\mathbb{S}^{d-1}} \int_{\R^d} |u(x + (t+r)\omega) - u(x + r\omega)|^p \d x \d\sigma_{d-1}(\omega)\\
&+2^{p-1}\int_{\mathbb{S}^{d-1}} \int_{\R^d} |u(x + r\omega) - u(x)|^p \d x \d\sigma_{d-1}(\omega)\\
&=2^{p-1}(U(t)+U(r)).
\end{align*}
By the same token we obtain the estimate $U(t)\leq 2^p|\mathbb{S}^{d-1}|\|u\|^p_{L^p(\R^d)}$. Furthermore, the above  estimate  implies that
\begin{align*}
U(t)\leq \frac{2^{p-1}}{t} \int_0^t (U(z)+ U(t-z) )\d z
= \frac{2^p}{t}\int_0^t U(z)\d z = 2^p AU(t).
\end{align*}
By  convexity again  for each  $ x\in \R^d$ and $\omega\in \mathbb{S}^{d-1}$ we have
\begin{align*}
|u(x+\ell t\omega )-u(x)|^p
&= \ell ^p \big|\frac{1}{\ell}\sum_{i=0}^{\ell-1}u(x+(i+1) t\omega)-u(x+i t\omega) \big|^p\\
&\leq \ell^{p-1} \sum_{i=0}^{\ell-1}|u(x+(i+1) t\omega)-u(x+i t\omega)|^p.
\end{align*}
Integrating both sides over $\mathbb{S}^{d-1}\times \R^d$ implies
\begin{align*}
U(\ell t)\leq \ell^pU(t)\qquad \text{and hence } \qquad g(\ell t) \leq g(t).
\end{align*}
\eqref{itme:lp-diff2} Since $t= s\ell+\theta s$ we deduce from the foregoing that
\begin{align*}
U(t) \leq 2^{p-1}\big(U(\ell s) +U(\theta s)\big)
&\leq
2^{p-1}\big(\ell^p U(s) +U(\theta s)\big)\leq
2^{p-1}\big( (t/s)^p U(s) +U(\theta s)\big).
\end{align*}
We conclude that
\begin{align*}
\frac{U(t)}{t^p}\leq 2^{p-1}\Big(\frac{U(s)}{s^p}+ \frac{U(\theta s)}{t^p}\Big).
\end{align*}
Moreover, we note that
\begin{align*}
U(\theta s)\leq \ell^p U\Big(\tfrac{\theta s}{\ell}\Big)
\leq 2^{p-1}\ell^p\Big( U(s)+U\big((1-\tfrac{\theta }{\ell})s\big)\Big).
\end{align*}
Inserting this in the previous expression gives
\begin{align*}
\frac{U(t)}{t^p}\leq 4^{p-1}(1+\ell^p)\Big(\frac{U(s)}{s^p}+  \frac{U\big((1-\tfrac{\theta }{\ell})s\big) }{t^p}\Big).
\end{align*}
Using the subadditivity and parity of $U$ we find that
\begin{align*}
 U(s)
&\leq \frac{2^{p-1}}{t} \int_0^t\big(U(z) +  U(s-z)\d z \big)= 2^{p-1}\big( AU(t)+ \frac{1}{t} \int^s_{s-t} U(r)\d r\big)\\
&\leq  2^{p-1}\big( AU(t)+ \frac{1}{t} \int^t_{-t} U(r)\d r\big)= 3.2^{p-1}AU(t)\leq 2^{p+1}AU(t).
\end{align*}
That is $U(r)\leq 2^{p+1} AU(t)$. Using $U(r)\leq 2^{p+1} AU(t)$ for $0<r\leq s\leq  t$ we get
\begin{align*}
AU(s)= \frac{1}{s}\int_0^s U(r)\d r \leq 2^{p+1}AU(t).
\end{align*}
\eqref{itme:lp-diff2-bis} Now assuming $d\geq 2$, by  Lemma \ref{lem:lp-diff-growth-spher}, we implies that  $U(s)\leq 2^p U(t)$ for $0<s\leq t$. Moreover, while $U(t)\leq 2^{p} AU(t)$has already been proved, combining the latter with the previously established estimates gives
\begin{align*}
\frac{U(t)}{t^p}&\leq  2^{p-1}\Big(\frac{U(s)}{s^p}+ \frac{U(\theta s)}{t^p}\Big)\leq 4^{p}\frac{U(s)}{s^p},\\
U(t+s)&\leq 2^{p-1}(U(t)+U(s))\leq 4^p U(t),\\
AU(t)&= \frac{1}{t}\int_0^t U(z)\d z \leq 2^p U(t).
\end{align*}

\eqref{itme:lp-diff3} The continuity of the shift, i.e.,  $\|u(\cdot+h)-u\|_{L^p(\R^d)}\to 0$ as $|h|\to0$ implies $\lim_{t \to 0^+} U(t)= \lim_{t \to 0^+} AU(t)=0$ and the continuity of $U$  and hence of $AU$, $\frac{U(t)}{t^p}$. To compute the limit of $U(t)$ as $t\to\infty$ let us consider $(u_n)_n\subset C_c^\infty(\R^d)$  be a sequence converging to $u$ in $L^p(\R^d)$ and let  $U_n$ be defined for $u_n$ in the same manner as $U$ is defined for $u$. For fixed $u_n$, suppose that $\supp u_n \subset B_R(0)$ for some $R>0$ so that $\supp ~u_n(\cdot+t\omega)  \subset B_R(-t\omega)$.
Therefore,  for $\omega \in \mathbb{S}^{d-1}$ and $t>2R$
\begin{align*}
\supp u_n \;\cap\; \supp\big(u_n(\cdot+t\omega)\big) = \emptyset.
\end{align*}
Consequently, since for $t>2R$, $x\in\R^d$  and $\omega\in \mathbb{S}^{d-1}$ we have
\begin{align*}
|u_n(x+t\omega)\pm u_n(x)|^p = |u_n(x+t\omega)|^p + |u_n(x)|^p.
\end{align*}
Hence we deduce that for all $t\geq 2R$ we have
\begin{align*}
\int_{\mathbb{S}^{d-1}}\|u_n(\cdot+t\omega)\pm u_n\|^p_{L^p(\R^d)}\d\sigma_{d-1}(\omega)
= 2|\mathbb{S}^{d-1}| \|u_n\|^p_{L^p(\R^d)}.
\end{align*}
This  shows that  $\|u_n(\cdot+ t\omega)\pm u_n\|^p_{L^p(\R^d)}\to 2\|u_n\|^p_{L^p(\R^d)}$  and $U_n(t)\to 2|\mathbb{S}^{d-1}|\|u_n\|^p_{L^p(\R^d)}$ as $t\to \infty$. Next, by reverse triangular inequality, we obtain
\begin{align*}
\big|U^{\frac{1}{p}}_n(t)-U^{\frac{1}{p}}(t)\big|&\leq 2|\mathbb{S}^{d-1}|^{\frac{1}{p}}\| u_n-u\|_{L^p(\R^d)}.
\end{align*}
Therefore, for each  $u_n$, we find that
\begin{align*}
\limsup_{t\to\infty}|U^{\frac{1}{p}}(t)&- 2|\mathbb{S}^{d-1}|^{\frac{1}{p}}\|u\|_{L^p(\R^d)}|\\
&\leq  \limsup_{t\to\infty}\Big(|U^{\frac{1}{p}}_n(t)- 2|\mathbb{S}^{d-1}|^{\frac{1}{p}}\|u\|_{L^p(\R^d)}|+\big|U^{\frac{1}{p}}_n(t)- U^{\frac{1}{p}}(t)\big|\Big)\\
&=2|\mathbb{S}^{d-1}|^{\frac{1}{p}} | \|u_n\|_{L^p(\R^d)}-\|u\|_{L^p(\R^d)}|+ \limsup_{t\to\infty}\big|U^{\frac{1}{p}}_n(t)- U^{\frac{1}{p}}(t)\big|\\
&\leq 4|\mathbb{S}^{d-1}|^{\frac{1}{p}}\| u_n-u\|_{L^p(\R^d)}.
\end{align*}
Letting $n\to \infty$ so that $\| u_n-u\|_{L^p(\R^d)}\to 0$ thereby yielding
\begin{align*}
\lim_{t\to\infty}U(t)=U(\infty):= 2|\mathbb{S}^{d-1}|\|u\|^p_{L^p(\R^d)}\quad\text{and hence}\quad \lim_{t\to\infty}\frac{U(t)}{t^p}=0.
\end{align*}
Analogously, one establishes that that for almost all $\omega\in \mathbb{S}^{d-1}$,
\begin{align*}
\lim_{t\to\infty}\|u(\cdot+ t\omega)\pm u\|^p_{L^p(\R^d)}= 2\|u\|^p_{L^p(\R^d)}.
\end{align*}
For the case $AU$, let $\varepsilon>0$ and $t_0$ such that $|U(t) -U(\infty)|<\varepsilon$ for all $t>t_0$.
\begin{align*}
 |AU(t) -U(\infty)|= \Big|\frac{1}{t}\int_0^t(U(z) -U(\infty))\d z\Big|\leq \frac{\varepsilon(t-t_0) }{t}+ \frac{1}{t}\int_{0}^{t_0} |U(z) -U(\infty)|\d z.
\end{align*}
It follows that $\limsup\limits_{t\to\infty} |AU(t) -U(\infty)|\leq  \varepsilon$. We deduce $AU(t)\xrightarrow{t\to\infty} U(\infty)$.
 \smallskip

\noindent \eqref{itme:lp-diff4} To show that $\frac{U(t)}{t^p}\to g(\infty):=K_{d,p}|\mathbb{S}^{d-1}|\|\nabla u\|^p_{L^p(\R^d)}$ as $t\to 0$, note that 
\begin{align*}
\int_{ \mathbb{S}^{d-1}}\int_{\R^d}|\nabla u(x)\cdot \omega|^p\d x  \d\sigma_{d-1}(\omega)= |\mathbb{S}^{d-1}|K_{d,p}\|\nabla u\|^p_{L^p(\R^d)}= g(\infty).
\end{align*}
The fundamental theorem of calculus  implies
$U(t)\leq |\mathbb{S}^{d-1}| K_{d,p}\|\nabla u\|^p_{L^p(\R^d)}t^p $ and 
\begin{align*}
\Big|(|\mathbb{S}^{d-1}|&K_{d,p})^{\frac{1}{p}}
\|\nabla u\|_{L^p(\R^d)}-\frac{U^{\frac{1}{p}}(t)}{t} \Big|
=\left|\Big( \int_{\mathbb{S}^{d-1}} \int_{\R^d} \Big| \int_0^1\nabla u(x)\cdot\omega\ d \tau \Big|^p \d x \d\sigma_{d-1}(\omega)\Big)^{\frac{1}{p}}\right.
\\ &\qquad\left.-\Big( \int_{\mathbb{S}^{d-1}} \int_{\R^d} \Big| \int_0^1\nabla u(x+ \tau t\omega)\cdot \omega\ d \tau \Big|^p \d x \d\sigma_{d-1}(\omega)\Big)^{\frac{1}{p}}\right|. 
\end{align*}
The reverse triangular inequality, Jensen's inequality and Fubini's imply

\begin{align*} 
\Big|[g(\infty)]^{\frac{1}{p}}
-\frac{U^{\frac{1}{p}}(t)}{t} \Big|
\leq \Big( \int_0^1 \int_{\mathbb{S}^{d-1}} \|\nabla u(\cdot + \tau t\omega)-\nabla u\|^p_{L^p(\R^d)}  \d\sigma_{d-1}(\omega) \d \tau\Big)^{\frac{1}{p}}\xrightarrow{t\to0}0,
\end{align*}
where the conclusion is implied by the  dominated convergence theorem, since $\|\nabla u(\cdot + \tau t\omega)-\nabla u\|^p_{L^p(\R^d)} \leq 2^p\|\nabla u\|^p_{L^p(\R^d)}$ and
$\|\nabla u(\cdot + \tau t\omega) - \nabla u\|^p_{L^p(\R^d)} \to 0$ as $t \to 0$  by the continuity of the shift operator in  $L^p(\R^d)$. Next let  $\varepsilon>0$ and consider $t_0>0$ such that $\big|\frac{U(t)}{t^p} -g(\infty)\big|<\varepsilon$  for all $0<t<t_0$. For all $0<t<t_0$ we find that
\begin{align*}
 \Big |\frac{(p+1)}{t^p}AU(t) -g(\infty)\Big|= \frac{(p+1)}{t^{p+1}}\int_0^t z^p\Big|\frac{U(z) }{z^p}-g(\infty)\Big|\d z
\leq (p+1)\varepsilon.
\end{align*}
We conclude that $(p+1)\frac{AU(t)}{t^p}\xrightarrow{t\to\infty} g(\infty)$.
\end{proof}
\begin{remark}\label{rem:almost increasing}
The approach used in Lemma~\ref{lem:lp-diff-growth-spher} strongly relies on the averaged slicing method from Theorem~\ref{thm:slicing-integ-on-sphere}, thereby precluding the one-dimensional case $d=1$. Thus, it remains an open question whether, in dimension $d=1$, there exists a constant $C_p>0$ such that the estimate
$U(\theta s) \leq C_p U(s)$ holds for all $s>0$ and $\theta \in [0,1]$. We believe this does not always hold for $d=1$. However, as we see in Theorem~\ref{thm:lp-first-order-diff}, one surrogate way to compensate for this one-dimensional gap consists in introducing the Hardy average 
$AU(s) = \frac{1}{s} \int_0^s U(z)\d z,$ 
which satisfies $AU(\theta s)\leq 2^{p+1}AU(s)$ for $d\geq 1$ while encapsulating the most essential properties inherited from $U$.
\end{remark}

\section{Asymptotics near $s=1$ and $s=0$}
\label{sec:assymp-s-near-1-0}
Let us  make a brief digression to establish asymptotic results in the regimes $s \to 0^+$ and $s \to 1^-$, for the normalized fractional $p$-Laplacian $(-\Delta)^s_p$ and the associated seminorm, with respect to normalizing  constant $\widetilde{C}_{d,p,s}$ defined in \eqref{eq:normalized-frac-cons}. We present a relatively simple approach to the proofs of the asymptotic limits of the fractional seminorm $ |\cdot|_{W^{s,p}(\R^d)} $ as $ s \to 1^- $, due to Brezis--Bourgain--Mironescu~\cite{BBM01}, and as $s \to 0^+ $, due to Maz'ya--Shaposhnikova~\cite{MS02}.

\begin{theorem}\label{thm:asymp-larg-short}
Let $u\in L^p(\R^d)$ and $r>0$. The  following assertions are true.
\begin{enumerate}
\item If $u\in L^p(\R^d)$ then we have
\begin{align*}
&\lim_{s\to0^+} s \iint_{\R^d\times B^c(0,r)} \frac{ |u(x) - u(y)|^p}{|x-y|^{d+s p}}\d y\d x= \lim_{s\to0^+} s\int_r^\infty \frac{U(t)}{t^{1+s p}}= \frac{2|\mathbb{S}^{d-1}|}{p}\|u\|^p_{L^p(\R^d)},\\
&\lim_{s\to1^-} (1-s) \iint_{\R^d\, B^c(0,r)} \frac{ |u(x) - u(y)|^p}{|x-y|^{d+s p}}\d y\d x= \lim_{s\to1^-} (1-s)\int_r^\infty \frac{U(t)}{t^{1+s p}}=0.
\end{align*}
\item If   $u\in \bigcup_{s\in (0,1)} W^{s, p}(\R^d)$ then we have
\begin{align*}
\lim_{s\to0^+}s\iint_{\R^d\times B(0,r)}\hspace{-2ex} \frac{ |u(x) - u(y)|^p}{|x-y|^{d+s p}}\d y \d x
&= \lim_{s\to0^+} s\int_0^r  \frac{U(t) }{t^{1+s p}}\d t=0.
\end{align*}
\item If $u\in W^{1,p}(\R^d)$ then we have
\begin{align*}
\lim_{s\to1^-}(1-s)&\iint_{\R^d \times B(0,r)}\hspace{-1ex} \frac{ |u(x) - u(y)|^p}{|x-y|^{d+s p}}\d y \d x\\
&= \lim_{s\to1^-} (1-s)\int_0^r  \frac{U(t) }{t^{1+s p}}\d =K_{d,p}\frac{|\mathbb{S}^{d-1}|}{p}\|\nabla u\|^p_{L^p(\R^d)}.
\end{align*}
\end{enumerate}
\end{theorem}

\begin{proof}
According to Theorem \ref{thm:lp-first-order-diff} we know that $U(t)\to |\mathbb{S}^{d-1}| \|u\|^p_{L^p(\R^d)}$ as $t\to\infty$ and $U(t)\leq 2^p |\mathbb{S}^{d-1}| \|u\|^p_{L^p(\R^d)}$.  Therefore, we deduce from Lemma \ref{lem:asymp-averag} that
\begin{align*}
&\lim_{s \to 0^+} \frac{sp}{p} \int_r^\infty \frac{U(t)}{t^{1+s p}} \d t = \frac{|\mathbb{S}^{d-1}|}{p} \|u\|^p_{L^p(\R^d)}.
\end{align*}
Furthermore, since $U(t)\leq 2^p|\mathbb{S}^{d-1}|\|u\|^p_{L^p(\R^d)} $ we have
\begin{align*}
\lim_{s\to1^-} (1-s)\int_r^\infty \frac{U(t)}{t^{1+s p}}\leq 2^p|\mathbb{S}^{d-1}|\|u\|^p_{L^p(\R^d)}\lim_{s\to1^-} \frac{(1-s) r^{-sp}}{sp}=0.
\end{align*}
Assume $u\in W^{s_0, p}(\R^d)$ for some  $0<s_0<1$ then, for $0<s<s_0$ we have
\begin{align*}
s\int_0^r \frac{U(t)}{t^{1+s p} }\leq sr^{(s_0-s)p} \int_0^r \frac{U(t)}{t^{s_0p+1} } \leq sr^{(s_0-s)p} |u|^p_{W^{s_0, p}(\R^d)}\xrightarrow{s\to 0}0.
\end{align*}
Now if  $u \in W^{1,p}(\R^d)$,  then  $g(t)=\frac{U(t)}{t^p}$ satisfies $g(t)\leq |\mathbb{S}^{d-1}| K_{d,p} \|\nabla u\|^p_{L^p(\R^d)}$ and
\begin{align*}
\lim_{t \to 0^+} \frac{U(t)}{t^p} = |\mathbb{S}^{d-1}| K_{d,p} \|\nabla u\|^p_{L^p(\R^d)} = g(0).
\end{align*}
Therefore, putting $a = (1 - s)p$  we deduce from Lemma \ref{lem:asymp-averag} that
\begin{align*}
\lim_{s \to 1^-}(1 - s) \int_0^r \frac{U(t)}{t^{1+s p}} \d t = \lim_{a \to 0^+} \frac{a}{p} \int_0^r g(t)t^{a-1}\d t=\frac{g(0)}{p}=   \frac{|\mathbb{S}^{d-1}|}{p} K_{d,p} \|\nabla u\|^p_{L^p(\R^d)}.
\end{align*}
\end{proof}

\begin{definition}\label{def:bvp-space}
 For $\Omega\subset \R^d$ open,  we define  the space $BV_p(\Omega)$ with $BV_p(\Omega)= W^{1,p}(\Omega)$ if $1< p<\infty$ and $BV_1(\Omega)= BV(\Omega)$, equipped with the norm $\|u\|_{ BV_p(\Omega)} =(\|u\|^p_{L^p(\Omega)} +|u|^p_{BV_p(\Omega)})^{1/p}$ with the seminorm $|\cdot|_{BV_p(\Omega)}$ defined by
\begin{align*}
|u|_{BV_1(\Omega)} &= |u|_{BV(\Omega)} \\
 |u|_{BV_p(\Omega)} &= |u|_{W^{1,p}(\Omega)} = \|\nabla u\|_{L^p(\Omega)} \quad \text{for } 1 < p < \infty.
\end{align*}
Here  we emphasize that $BV(\Omega)$ is the space of functions of bounded variation and  $|u|_{BV_1(\Omega)}= |u|_{BV(\Omega)} $ is the total variation of $u$.
\end{definition}
\begin{theorem}\label{thm:mazy'a-bbm-limit}
Let $u\in L^p(\R^d)$, $1\leq p<\infty$. The following assertions are true.
\smallskip

\textbf{Maz'ya-Shaposhnikova:} We have  $u\in \bigcup_{s\in (0,1)} W^{s, p}(\R^d)$ if and only if
\begin{align*}
&\lim_{s\to0^+} s \iint_{\R^d \times \R^d} \frac{ |u(x) - u(y)|^p}{|x-y|^{d+s p}}\d y\d x=  \frac{2|\mathbb{S}^{d-1}|}{p}\|u\|^p_{L^p(\R^d)}.
\end{align*}
\textbf{Brezis-Bourgain-Mirunescu:}
Under the convention that $|u|_{BV_p(\R^{d})}=\infty$ whenever
$u\in L^p(\R^{d})\setminus BV_p(\R^{d})$ we have
\begin{align*}
&\lim_{s\to1^-}s(1-s)\iint_{\R^d \times \R^d}\hspace{-2ex} \frac{ |u(x) - u(y)|^p}{|x-y|^{d+s p}}\d y \d x=K_{d,p}\frac{|\mathbb{S}^{d-1}|}{p}\|\nabla u\|^p_{L^p(\R^d)}.
\end{align*}
In particular, the right-hand side limit  is finite if and only if
$u\in BV_p(\R^d)$.

\end{theorem}
\begin{proof}
For $u\in \bigcup_{s\in (0,1)} W^{s, p}(\R^d)$ Theorem \ref{thm:asymp-larg-short} implies
\begin{align*}
\lim_{s\to0^+}
[u]^p_{W^{s,p}(\R^d)}&= \lim_{s\to0^+}\Big( s\int_1^\infty \frac{U(t)}{t^{1+s p}}
+s\int_0^1 \frac{U(t) }{t^{1+s p}}\d t\Big) =\frac{2|\mathbb{S}^{d-1}|}{p}\|u\|^p_{L^p(\R^d)},
\end{align*}
where we recall  that
\begin{align*}
[u]^p_{W^{s,p}(\R^d)}= s(1-s)\iint_{\R^d \times \R^d} \frac{ |u(x) - u(y)|^p}{|x-y|^{d+s p}}\d y\d x.
\end{align*}
Conversely, if the above limit holds then  $u\in \bigcup_{0 < s < 1} W^{s,p}(\R^d)$ since
\begin{align*}
s\iint_{\R^d \times \R^d} \frac{ |u(x) - u(y)|^p}{|x-y|^{d+s p}}\d y\d x\leq \big(1+\tfrac{2|\mathbb{S}^{d-1}|}{p}\|u\|^p_{L^p(\R^d)}\big)\quad\text{for all $0<s<s_0$},
\end{align*}
for some $s_0\in (0,1)$. Next for $u\in W^{1,p}(\R^d)$, $p\geq1$, Theorem \ref{thm:asymp-larg-short} implies
\begin{align*}
\lim_{s\to0^+}
[u]^p_{W^{s,p}(\R^d)}
&= \lim_{s\to1^-} (1-s) \Big( \int_0^1 \frac{U(t) }{t^{1+s p}}\d t+ \int_1^\infty  \frac{U(t) }{t^{1+s p}}\d t\Big)
= K_{d,p}\frac{|\mathbb{S}^{d-1}|}{p}\|\nabla u\|^p_{L^p(\R^d)}.
\end{align*}
The general case $u\in BV(\R^d)$ when $p=1$ follows
analogously, see \cite{Fog23} and \cite{Dav02}. Recall that by \cite{BBM01}, see also \cite[Theorem 1.1]{Fog23}, we find that
\begin{align*}
&\liminf_{s\to1^-}
[u]^p_{W^{s,p}(\R^d)}<\infty \implies \begin{cases}
u\in W^{1,p}(\R^d), & \text{if } 1<p<\infty\\
u\in BV(\R^d),&\text{if } p=1.
\end{cases}
\end{align*}
In particular, if $u\notin W^{1,p}(\R^d)$ for $1<p<\infty$ (resp. $u\notin BV(\R^d)$ for p=1) then
\begin{align*}
\liminf_{s\to1^-}s(1-s)\iint_{\R^d \times \R^d}\hspace{-2ex} \frac{ |u(x) - u(y)|^p}{|x-y|^{d+s p}}\d y \d x=\infty,
\end{align*}
that is we have $\|\nabla u\|^p_{L^{p}(\R^d)}=\infty$ (resp. $|u|_{BV(\R^d)}=\infty$).
\end{proof}

 \begin{corollary}
For $u\in L^p(\R^d)$, $1\leq p<\infty$ we have
\begin{align*}
&\lim_{s\to1^-} \frac{C_{d,p,s}}{2} \iint_{\R^d \times \R^d} \frac{ |u(x) - u(y)|^p}{|x-y|^{d+s p}}\d y\d x=  \|\nabla u\|^p_{L^p(\R^d)}\qquad \text{for $u\in  W^{1,p}(\R^d)$},\\
&\lim_{s\to0^+} \frac{C_{d,2,\frac{sp}{2}}}{2} \iint_{\R^d \times \R^d} \frac{ |u(x) - u(y)|^p}{|x-y|^{d+s p}}\d y\d x=  \|u\|^p_{L^p(\R^d)} \qquad \text{for  $u\in \bigcup_{0 < s < 1} W^{s,p}(\R^d)$}. 
\end{align*}
\end{corollary}

\begin{proof}
This is directly implied by Theorem \ref{thm:mazy'a-bbm-limit} and the fact that
\begin{align*}
\lim_{s\to 1^-} \frac{K_{d,p}C_{d,p,s} }{2s(1-s)}
=\lim_{s\to 0^+} \frac{C_{d,2,\frac{sp}{2}}}{s(1-s)}=\frac{p}{|\mathbb{S}^{d-1}|}.
\end{align*}
\end{proof}
\begin{lemma}
\label{lem:ptwise-s-1-s-asymp}
Let $\Omega \subset \R^d$ be any open set and $x\in \Omega$.
Given  $f\in L^\infty_{\mathrm{loc}}(\R^d)$  and $g\in L^\infty(\Omega)$, we  define
\begin{align*}
f_\infty(x,w)=\lim_{r\to \infty} f(x+rw)\quad \text{and}\quad
g_0(x,w)=\lim_{r\to 0^+} g(x+rw)\qquad w\in \mathbb{S}^{d-1}.
\end{align*}
Assume $f_\infty(x,w)$ and $g_0(x,w)$ exist in $\R\cup \{-\infty,\infty\}$ for a.e. $w\in \mathbb{S}^{d-1}$ then
\begin{align*}
\lim_{s \to 1^-}(1-s) \int_{\Omega} \frac{ g(y)\, \d y}{|x-y|^{d-(1-s)p}}
&= \frac{1}{p} \int_{\mathbb{S}^{d-1}} g_0(x,w)\d\sigma_{d-1}(w),\\
\lim_{s \to 0^+}s \int_{\R^d \setminus \Omega} \frac{ f(y)\, \d y}{|x-y|^{d+sp}}
&= \frac{1}{p} \int_{\mathbb{S}^{d-1}} f_\infty(x,w)\mathds{1}_{S_\Omega(x)} (w)\d\sigma_{d-1}(w)
\end{align*}
where $S_\Omega(x) =
 \{ w \in \mathbb{S}^{d-1} :  \lim_{r\to \infty} \mathds{1}_{\Omega^c}(x+rw)=1\}$ is \emph{ the escape sector at $x$}.  In particular there hold that $S_\Omega(x)= \mathbb{S}^{d-1}$ if $\Omega$ is bounded,  $S_\Omega(x)= \emptyset$ if $\Omega^c$ is bounded, while $S_{\R^d_+}(x)= \mathbb{S}^{d-1}_+$  with $\R^d_+=\{x\in \R^d\,:\, x_d>0\}$ and $\mathbb{S}^{d-1}_+= \mathbb{S}^{d-1}\cap \R^d_+$.
\end{lemma}

\begin{proof}
Since $x\in \Omega$ we have $\delta_x=\operatorname{dist}(x,\partial\Omega)>0$. Let $c_1= \|f\|_{L^\infty(B_\rho(x))}$ and $c_2= \|g\|_{L^\infty(\Omega)}$. For arbitrarily  large $\rho>\delta_x$ and arbitrarily small $0<\delta<\delta_x$ we have
\begin{align*}
&s \int_{ \Omega^c\cap B_\rho(x)}
\frac{|f(y)|\, \d y}{|x-y|^{d+sp}} \leq
\int_{B_\rho(x)\setminus B_{\delta_x}(x)}
\hspace{-1ex} \frac{s c_1\d y}{|x-y|^{d+sp}}
= c_1\frac{|\mathbb{S}^{d-1}|}{p} (\delta_x^{-sp}
-\rho^{-sp})\xrightarrow{s\to 0^+}0,\\
&(1-s) \int_{\Omega\cap B^c_{\delta} (x)}
\frac{ |g(y)|\d y}{|x-y|^{d-(1-s)p}} \leq  \int_{B^c_\delta(x)} \frac{c_2(1-s)\d y}{|x-y|^{d-(1-s)p}}
= c_2\frac{  (1-s) |\mathbb{S}^{d-1}|}{sp \delta^{(s-1)p}}
\xrightarrow{s\to 1^-}0. 
\end{align*}
Using  the previous remark, Lemma \ref{lem:asymp-averag}  and dominated convergence we get
\begin{align*}
\lim_{s \to 0^+}s \int_{\R^d \setminus \Omega} \frac{ f(y)\, \d y}{|x-y|^{d+sp}}
&=\lim_{s \to 0^+}s \int_{B^c_\rho(x)} \frac{ f(y)\mathds{1}_{\Omega^c}(y)\, \d y}{|x-y|^{d+sp}} \\
&=  \int_{\mathbb{S}^{d-1}}
 \lim_{s\to 0^+}s\int_\rho^\infty
f(x + rw)\mathbf{1}_{\Omega^c}(x + rw)
r^{-sp-1}\d r \, \d\sigma_{d-1}(w)\\
&= \frac{1}{p} \int_{\mathbb{S}^{d-1}}\lim_{r\to \infty} f(x+rw) \mathds{1}_{\Omega^c}(x+rw)\,  \d\sigma_{d-1}(w)\\
&=\frac{1}{p} \int_{\mathbb{S}^{d-1}}
f_\infty(x, w) \mathds{1}_{S_\Omega(x)}(w)\,  \d\sigma_{d-1}(w),
\end{align*}
where we note that, for each $w \in \mathbb{S}^{d-1}$,
\begin{align*}
\lim_{r\to \infty} \mathds{1}_{\Omega^c}(x+rw)&= \mathds{1}_{S_\Omega(x)} (w)=
\begin{cases}
1 & \text{if } x + rw \in \Omega^c \text{ for all sufficiently large } r, \\
0 & \text{otherwise}.
\end{cases}
\end{align*}
Likewise, since $B_\delta(x)\subset B_{\delta_x}(x)\subset \Omega$,  dominated convergence and  Lemma \ref{lem:asymp-averag} yield 
\begin{align*}
\lim_{s\to 1^-}(1-s) \int_{\Omega} \frac{ g(y)\d y}{|x-y|^{d-(1-s)p}}
&= \lim_{a \to 0^+}\frac{a}{p} \int_{B_{\delta} (x)} \frac{ g(y)\d y}{|x-y|^{d-a}} \\
&= 
\int_{\mathbb{S}^{d-1}}  \lim_{a \to 0^+}\frac{a}{p} \int_0^\delta  g(x+r w)   r^{a-1}\d r \d \sigma_{d-1}(w) \\
&=  \frac{1}{p} \int_{\mathbb{S}^{d-1}} g_0(x,w)\d\sigma_{d-1}(w).
\end{align*}
Finally, if $\Omega$ is bounded then  we have $ \Omega\subset B_R(x)$  for $R>0$ large enough, so that $x+rw\in B_R^c(x)\subset \Omega^c$ for all $w\in \mathbb{S}^{d-1}$ and $r>R.$ Hence one readily gets that  $S_\Omega(x) =\mathbb{S}^{d-1}$. Likewise if $\Omega^c$ is bounded  say $\Omega^c \subset B_R(x)$ for some $R>0$, then for all $r>R$ and $w\in \mathbb{S}^{d-1}$ we have $x+rw\in \Omega\subset B_R^c(x)$, that is, $\mathds{1}_{\Omega^c}(x+rw)=0$  and hence $S_\Omega(x)= \emptyset$. Analogously it is not difficult to verify that $S_{\R_+}(x)=\mathbb{S}^{d-1}_+$.
\end{proof}

\begin{theorem}\label{thm:mazya-on-dom}
Let $\Omega\subset \R^d$ be any open set and  $u\in L^p(\Omega)$, $1\leq p<\infty$.  For $x\in\Omega$ we denote   $\delta_x= \dist(x,\partial \Omega)$,   $S_\Omega(x) =
 \{ w \in \mathbb{S}^{d-1} :  \lim\limits_{r\to \infty} \mathds{1}_{\Omega^c}(x+rw)=1\}$ and
\begin{align*}
 |S_\Omega(x)|:=\sigma_{d-1}(S_\Omega(x))=\int_{ \mathbb{S}^{d-1}} \mathds{1}_{S_\Omega(x)}(w)\d \sigma_{d-1}(w).
\end{align*}
The following assertions are true.
\begin{enumerate}
\item\label{item:mazya-dom0} For  $u\in \bigcup_{0 < s < 1} L^p(\Omega, \delta_x^{-sp})$ we have
\begin{align*}
\lim_{s\to0^+}
s \int_{\Omega} |u(x)|^p
\int_{\R^d \setminus \Omega}
\frac{\d y }{|x-y|^{d+sp}}  \d x
=\frac{1}{p}\int_{\Omega}
|u(x)|^p\,  |S_\Omega(x)| \d x.
\end{align*}
\item  \label{item:mazya-dom} For $u\in \bigcup_{0 < s < 1} W^{s,p}(\Omega)\cap L^p(\Omega, \delta_x^{-sp})$ we have
\begin{align*}
\lim_{s \to 0^+} s(1-s) \iint_{\Omega \times \Omega}
\frac{|u(x) - u(y)|^p}{|x-y|^{d+sp}} \d y \d x=
\frac{2}{p} \int_\Omega
\big(|\mathbb{S}^{d-1}| - |S_\Omega(x)|\big)|u(x)|^p \d x.
\end{align*}
\item \label{item:mazya-dom1} Assume that either $\Omega$ has a compact Lipschitz boundary, or that $\Omega=\{x= (x',x_d): x_d>\zeta(x')\}$ is epigraph of a Lipschitz function $\zeta:\R^{d-1}\to \R$. Then for $u\in \bigcup_{0 < s < 1} W^{s,p}(\Omega)$ we have
\begin{align*}
\lim_{s \to 0^+} s(1-s) \iint_{\Omega \times \Omega}
\frac{|u(x) - u(y)|^p}{|x-y|^{d+sp}} \d y \d x=
\frac{2}{p} \int_\Omega
\big(|\mathbb{S}^{d-1}| - |S_\Omega(x)|\big) |u(x)|^p \d x.
\end{align*}
\item\label{item:mazya-dom2} Under the conditions of item \eqref{item:mazya-dom} or  item  \eqref{item:mazya-dom1}, the following hold
\begin{align*}
&\lim_{s\to0^+}
s(1-s) \iint_{\Omega \times\Omega}
\frac{|u(x)-u(y)|^p}{|x-y|^{d+sp}}  \d x =0
\qquad\qquad\qquad \text{if $\Omega$ is bounded},\\
&\lim_{s\to0^+}
s(1-s) \iint_{\Omega \times\Omega}
\frac{|u(x)-u(y)|^p}{|x-y|^{d+sp}}  \d x =  \frac{2|\mathbb{S}^{d-1}|}{p}
\|u\|^p_{L^p(\Omega)}
\qquad\text{if $\Omega^c$ is bounded},\\
&\lim_{s\to0^+}
s(1-s)\iint_{\R^d_+\times \R^d_+}
\frac{|u(x)-u(y)|^p}{|x-y|^{d+sp}}  \d x =  \frac{|\mathbb{S}^{d-1}|}{p}\|u\|^p_{L^p(\R^d_+)}\qquad\text{if $\Omega=\R^d_+$}.
\end{align*}
\end{enumerate}
\end{theorem}
\begin{proof}
$\eqref{item:mazya-dom0}$ Let $u\in L^p(\Omega, \delta_x^{-s_0p})$ for some $0<s_0<1$. Lemma \ref{lem:ptwise-s-1-s-asymp} implies 
\begin{align*}
 s \int_{\R^d \setminus \Omega} \frac{\d y}{|x-y|^{d+sp}} \xrightarrow{s\to0^+} \frac{1}{p}|S_\Omega(x)|\qquad x\in \Omega.
\end{align*}
Moreover,  for all $0<s\leq s_0$ we have
\begin{align*}
s \int_{\R^d \setminus \Omega} \frac{\d y}{|x-y|^{d+sp}}\leq s \int_{\R^d \setminus B_{\delta_x}(x)} \frac{\d y}{|x-y|^{d+sp}}= \frac{|\mathbb{S}^{d-1}|}{p}\delta_x^{-sp}
\leq \frac{|\mathbb{S}^{d-1}|}{p}(1+\delta_x^{-s_0p}).
\end{align*}
Since $|u(x)|^p(1+\delta_x^{-s_0p})\in L^1(\Omega)$, the dominated convergence theorem yields
\begin{align*}
\lim_{s\to0^+}
s \int_{\Omega} |u(x)|^p
\int_{\R^d \setminus \Omega}
\frac{\d y }{|x-y|^{d+sp}}  \d x=\frac{|\mathbb{S}^{d-1}|}{p}
 \int_{\Omega}
|u(x)|^p\,  |S_\Omega(x)|  \d x.
\end{align*}
$\eqref{item:mazya-dom}$ Assume $u\in W^{s_0,p}(\Omega)\cap L^p(\Omega, \delta_x^{-s_0p})$ for some $0<s_0<1$. It appears that $\widetilde{u}\in W^{s_0,p}(\R^d)$ where $\widetilde{u}$ is the  zero-extension of $u$, i.e., $\widetilde{u}(x)= u(x)$ for $x\in \Omega$ and $\widetilde{u}(x)=0$ for $x\in \Omega^c$.  Indeed, we have
\begin{align*}
\iint_{\R^d \times \R^d}
\frac{|\widetilde{u}(x) - \widetilde{u}(y)|^p}{|x-y|^{d+s_0p}} \d y \d x
&\leq \iint_{\Omega \times \Omega}
\frac{|u(x) - u(y)|^p}{|x-y|^{d+s_0p}} \d y \d x
+ \frac{2 |\mathbb{S}^{d-1}|}{s_0p}\int_{\Omega} \frac{|u(x)|^p}{\delta_x^{s_0p}}\d x.
\end{align*}
Therefore, since $\widetilde{u}\in W^{s_0,p}(\R^d)$, Theorem  \ref{thm:mazy'a-bbm-limit} implies
\begin{align*}
\lim_{s \to 0^+} s [\widetilde{u}]^p_{W^{s,p}(\mathbb{R}^d)}
= \frac{2}{p}|\mathbb{S}^{d-1}|
\int_{\R^d} |\widetilde{u}(x)|^p \d x=  \frac{2}{p}|\mathbb{S}^{d-1}|
\int_{\Omega} |u(x)|^p \d x.
\end{align*}
This in conjunction with the foregoing gives
\begin{align*}
\lim_{s\to 0^+} [u]^p_{W^{s,p}(\Omega)}
&= \lim_{s\to 0^+} [\widetilde{u}]^p_{W^{s,p}(\R^d)}-2\lim_{s\to 0^+}s \int_{\Omega} |u(x)|^p
\Big( \int_{\R^d \setminus \Omega}
\frac{\d y}{|x-y|^{d+sp}}  \Big) \d x\\
&= \frac{2}{p}\int_\Omega  (|\mathbb{S}^{d-1}|- |S_\Omega(x)| )|u(x)|^p\d x.
\end{align*}
$\eqref{item:mazya-dom1}$ If $\partial\Omega$ is compact Lipschitz or $\Omega=\{x= (x',x_d): x_d>\zeta(x')\}$, then
the fractional Hardy inequality \cite{Dyd04, AJR26}   entails the continuous embedding  $W^{s,p}(\Omega)\hookrightarrow L^p(\Omega, \delta_x^{-sp})$  for $0<s<\frac{1}{p}$. The desired results follow from the previous assertion.
\medskip

\eqref{item:mazya-dom2} This follows from Lemma \ref{lem:ptwise-s-1-s-asymp} since for all $x\in \R^d_+$ we have  $S_{\R^d_+}(x)=\mathbb{S}^{d-1}_+$, whereas for all $x\in \Omega$ we have  $S_{\Omega}(x)=\mathbb{S}^{d-1}$ if $\Omega$ is bounded and $S_{\Omega}(x)=\emptyset$ if $\Omega^c$ is bounded.
\end{proof}

\begin{corollary}
 Let $\Omega\subset \R^d$ be any open bounded set. Let $u\in \bigcup_{0 < s < 1} W^{s,p}(\R^d)$, $1\leq p<\infty$. Then we have
\begin{align*}
&\lim_{s\to 0^+}
s(1-s) \iint_{\Omega \times \Omega}
\frac{|u(x) - u(y)|^p}{|x-y|^{d+sp}} \d y \d x=0,\\
&\lim_{s\to 0^+}
s(1-s) \iint_{\R^d_- \times \R^d_+}
\frac{|u(x) - u(y)|^p}{|x-y|^{d+sp}} \d y \d x=\frac{|\mathbb{S}^{d-1}|}{2p}\|u\|^p_{L^p(\R^d)}.
\end{align*}
If in addition $\Omega$ is bounded Lipschitz then we have
\begin{align*}
&\lim_{s\to 0^+}
s(1-s) \iint_{\Omega^c \times \Omega^c}
\frac{|u(x) - u(y)|^p}{|x-y|^{d+sp}} \d y \d x=
\frac{2|\mathbb{S}^{d-1}|}{p}\|u\|^p_{L^p(\Omega^c)},\\
&\lim_{s\to 0^+}
s(1-s) \iint_{\Omega^c \times \Omega}
\frac{|u(x) - u(y)|^p}{|x-y|^{d+sp}} \d y \d x=\frac{|\mathbb{S}^{d-1}|}{p}\|u\|^p_{L^p(\Omega)}.
\end{align*}
\end{corollary}
\begin{proof}
 Since $B_R(0)$ is Lipschitz and $\Omega\subset B_R(0)$ for some $R>0$, Theorem \ref{thm:mazya-on-dom} implies
\begin{align*}
&\limsup_{s\to 0^+}
s \iint_{\Omega \times \Omega}
\frac{|u(x) - u(y)|^p}{|x-y|^{d+sp}} \d y \d x\leq \limsup_{s\to 0^+}
s \iint_{B_R(0) \times B_R(0)}
\frac{|u(x) - u(y)|^p}{|x-y|^{d+sp}} \d y \d x=0.
\end{align*}
Next, combining Theorem \ref{thm:mazy'a-bbm-limit} and Theorem \ref{thm:mazya-on-dom} we find that
\begin{align*}
\lim_{s\to 0^+}
s\iint_{\R^d_- \times \R^d_+}
\hspace{-2ex}\frac{|u(x) - u(y)|^p}{|x-y|^{d+sp}} \d y \d x
&=\lim_{s\to 0^+}
\frac{s}{2}\Big(\iint_{\R^d \times \R^d}- \iint_{\R^d_+ \times \R^d_+}- \iint_{\R^d_- \times \R^d_-}\Big)[\cdots]\\
&=\frac{|\mathbb{S}^{d-1}|}{2p}\Big( 2\|u\|^p_{L^p(\R^d)}-
\|u\|^p_{L^p(\R^d_+)}-\|u\|^p_{L^p(\R^d_-)}\Big)\\
&=\frac{|\mathbb{S}^{d-1}|}{2p}\|u\|^p_{L^p(\R^d)}.
\end{align*}
Analogously if $\Omega$ is bounded Lipschitz so that $|\partial \Omega|=0$, then Theorem \ref{thm:mazya-on-dom} implies
\begin{align*}
&\lim_{s\to 0^+}
s(1-s) \iint_{\Omega^c \times \Omega^c}
\frac{|u(x) - u(y)|^p}{|x-y|^{d+sp}} \d y \d x=
\frac{2|\mathbb{S}^{d-1}|}{p}\|u\|^p_{L^p(\Omega^c)},
\end{align*}
while, splitting the integral, we deduce from the foregoing  that
\begin{align*}
\lim_{s\to 0^+}
s\iint_{\Omega^c \times \Omega}
\hspace{-2ex} \frac{|u(x) - u(y)|^p}{|x-y|^{d+sp}} &\d y \d x
= \lim_{s\to 0^+}
\frac{s}{2}\Big(\iint_{\R^d \times \R^d} \hspace{-2ex}- \iint_{\Omega \times \Omega}- \iint_{\Omega^c \times \Omega^c}\Big)[\cdots]\\
&=\frac{|\mathbb{S}^{d-1}|}{2p}
\Big( \|u\|^p_{L^p(\R^d)}-\|u\|^p_{L^p(\Omega^c)}\Big)=\frac{|\mathbb{S}^{d-1}|}{p}\|u\|^p_{L^p(\Omega)}.
\end{align*}
\end{proof}
For completeness, we also include the asymptotic behavior as $s \to 1^-$, which can be recovered from the results established in \cite{BBM01,Dav02,Fog23,Fog25}.
\begin{theorem}\label{thm:xBBM-frac}
Assume  $\Omega\subset \R^d$ is any $BV_p$-extension domain. Then there is $C= C(d,p,\Omega)>0$ such that for $u\in L^p(\Omega)$ we have
\begin{align*}
\lim_{s\to 1^-}\frac{C_{d,p,s}}{2}\iint_{ \Omega \times \Omega}\frac{|u(x) - u(y)|^p}{|x-y|^{d+sp}} \d y \d x&= |u|^p_{BV_p(\Omega)}\\
\sup_{s\in (0,1)} \frac{\widetilde{C}_{d,p,s}}{2}\iint_{ \Omega \times \Omega}\frac{|u(x) - u(y)|^p}{|x-y|^{d+sp}} \d y \d x &\leq C \|u\|^p_{BV_p(\Omega)}.
\end{align*}
\end{theorem}

Next, we provide the pointwise asymptotics for the fractional $p$-Laplacian.
\begin{theorem}
\label{thm:asymp-frac-plaplace}
Assume   $1<p<\infty$, $\frac{1}{p-1}<q\leq\infty$ and $x\in\R^d$. Consider the normalized fractional $p$-Laplacian
\begin{align*}
(-\Delta)^s_p u(x)
:= \widetilde{C}_{d,p,s}\,
\mathrm{p.v.}\int_{\mathbb{R}^d}
\frac{|u(x)-u(y)|^{p-2}(u(x)-u(y))}
{|x-y|^{d+sp}}\d y.
\end{align*}
If $L^\infty(\R^d)\cap C^2(B_1(x))$ then
\begin{align*}
\lim_{s\to 1^-}(-\Delta)^s_p u(x) = \lim_{s\to0^+} \frac{2(1-s)p}{|\mathbb S^{d-1}|K_{d,p}}\int_{\R^d}
&\frac{|u(x)-u(y)|^{p-2}(u(x)-u(y))}
{|x-y|^{d+sp}}\d y = -\Delta_p u(x).
\end{align*}
If  $u\in L^q_{\mathrm{loc}}(\R^d)\cap C^{\varepsilon}(B_1(x))$
with  $\varepsilon\in (0,1)$ and $u(\infty)=\lim\limits_{|x|\to\infty} u(x)$  exists then
\begin{align*}
\lim_{s\to0} (-\Delta)^s_p u(x)= \lim_{s\to0^+} \frac{sp}{|\mathbb S^{d-1}|}\int_{\R^d}
&\frac{|u(x)-u(y)|^{p-2}(u(x)-u(y))}
{|x-y|^{d+sp}}\d y= I_p u(x).
\end{align*}
where  we define $I_p u(x)=
|u(x)-u(\infty)|^{p-2}(u(x)-u(\infty))$. In particular, if $u\in L^q(\mathbb{R}^d)\cap C^{\varepsilon}(B_1(x))$ with $1<q<\infty$ and $\lim\limits_{|x|\to\infty} u(x)=0$   then
\begin{align*}
\lim_{s\to0} (-\Delta)^s_p u(x)=I_p u(x)=
|u(x)|^{p-2}u(x).
\end{align*}
\end{theorem}

\begin{proof}
The first limit near $s=1$ can be found in \cite[Section 9]{Fog25}. If $u(\infty)$ exists and $u\in L^q( \R^d)$ with $1<q<\infty$, necessarily $u(\infty)=0.$ Moreover, note that $\lim_{s\to0}\frac{ \widetilde{C}_{d,p,s}}{s(1-s)}
=\frac{p}{|\mathbb S^{d-1}|}.$  Since $u\in C^\varepsilon(B_1(x))$, there exists $C>0$ such that $|u(x)-u(y)| \leq C |x-y|^\varepsilon$, $y\in B_1(x)$. Therefore, for  $0<s<\varepsilon(p-1)$, we find that
\begin{align*}
\left|
s\int_{B_1(x)}\hspace{-2ex}
\frac{|u(x)-u(y)|^{p-2}(u(x)-u(y))}
{|x-y|^{d+sp}}dy
\right|
&\leq C
\int_{B_1(x)}
|x-y|^{-d-sp+\varepsilon(p-1)}dy \\
&=C\frac{ s|\mathbb S^{d-1}|}{ -sp+\varepsilon(p-1)} \to 0
\quad \text{as } s\to0^+.
\end{align*}
Next,  define $u_\infty(x)= u(x)-u(\infty)$ and 
$$g_{x, w}(t)= |u(x)-u(x+t w)|^{p-2}(u(x)-u(x+t w))$$ 
so that $g_{x,w}(\infty)
=
|u_\infty(x)|^{p-2}u_\infty(x)$  and $g_{x, w}\in L^{(p-1)q}_{\mathrm{loc}}((0,\infty))$ with $1<(p-1)q\leq \infty$. In view of Lemma \ref{lem:asymp-averag} we find that
\begin{align*}
\lim_{s\to0^+}s\int_{B^c_1(x)}\hspace{-3ex}
\frac{|u(x)-u(y)|^{p-2}(u(x)-u(y))}
{|x-y|^{d+sp}}\d y&=
\lim_{s\to0^+}s\int_{\mathbb{S}^{d-1}} \int_1^\infty \frac{g_{x, w}(t)}{ t^{1+sp}}  \d t\d \sigma_{d-1}(w)\\
&=\frac{|\mathbb S^{d-1}|}{p}g_{x, w}(\infty)= \frac{|\mathbb S^{d-1}|}{p}|u_\infty(x)|^{p-2}u_\infty(x),
\end{align*}
with $u_\infty(x)= u(x)-u(\infty)$.  Altogether we obtain
\begin{align*}
\lim_{s\to0^+} \frac{sp}{|\mathbb S^{d-1}|}\int_{\R^d}
&\frac{|u(x)-u(y)|^{p-2}(u(x)-u(y))}
{|x-y|^{d+sp}}\d y= |u_\infty(x)|^{p-2}u_\infty(x).
\end{align*}
\end{proof}

\section{Proof of Theorem \ref{thm:robust-interpo-esti} and Theorem \ref{thm:r-split-monot-semin} }
\label{sec:robust-interpo}
In this section, we establish the proofs of Theorem~\ref{thm:robust-interpo-esti} and Theorem~\ref{thm:r-split-monot-semin}. Let us recall that the function $U \colon [0,\infty) \to [0,\infty)$, given for $u \in L^p(\mathbb{R}^d)$ by
\begin{align*}
U(t) = \int_{\mathbb{S}^{d-1}} \int_{\mathbb{R}^d} |u(x + t\omega) - u(x)|^p \d x \, \d \sigma_{d-1}(\omega).
\end{align*}
\subsection{Proof of Theorem \ref{thm:r-split-monot-semin}} \label{sec:r-split-monot-semin}
We begin with the monotonicity of the short-range interaction $s \mapsto L_p(s,u,r)$ and the long-range map $s \mapsto T_p(s,u,r)$ with respect to $s$ (see Definition \ref{def:long-short-range}), starting with the case $p=2$.
\begin{lemma}\label{lem:short-cos}
Let  $\xi\in\R$ and $r>0$.  For $0\leq \eta\leq s\leq 1$ we have
\begin{align*}
(1-\eta) (2r)^{-2(1-\eta)} \int_0^r\frac{(1-\cos (\xi t))}{t^{1+2\eta}}\d t&\leq (1-s)  (2r)^{-2(1-s)} \int_0^r \frac{(1-\cos (\xi t))}{t^{1+2s}}\d t,
\end{align*}
with the extreme cases $s=1$ and $\xi=\infty$  given by
\begin{align*}
&\lim_{s\to 1^-}(1-s)  \int_0^1 \frac{(1-\cos (\xi t))}{t^{1+2s}}\d t=\frac{\xi^2}{4},\\
&\lim_{\xi\to \infty}(1-s) \int_0^1 \frac{(1-\cos (\xi t))}{t^{1+2s}}\d t =\infty.
\end{align*}
\end{lemma}
\begin{proof}
For $\xi\in \R$ the mapping $t\mapsto g_\xi(t) =\frac{1-\cos (\xi t)}{t^2}$ with $ g_\xi(0)= \frac{\xi^2}{ 2}$ is continuous and  dyadically decreasing, i.e., $g_\xi(2t)\leq g_\xi(t)$ for all $t\geq 0$. Applying the Chebyshev type inequality in Theorem \ref{thm:cadic-monotonicity} with $a=2(1-s)\leq b=2(1-\eta)$ implies that
\begin{align*}
2(1-\eta)(2 r)^{-2(1-\eta)} \int_0^{r} g_\xi(t) t^{2(1-\eta)-1}\d t\leq 2(1-s)(2 r)^{-2(1-s)}\int_0^r  g_\xi(t)t^{2(1-s)-1} \d t,
\end{align*}
yielding the sought inequality.
\end{proof}

\begin{lemma}\label{lem:tail-short-cos}
Let  $\xi\in\R$ and $0\leq \sigma\leq \eta<1$. The following estimates hold true
\begin{align*}
\eta  r^{2\eta}\int_r^\infty \frac{(1-\cos (\xi t))}{t^{1+2\eta}}\d t
&\leq 4\sigma  r^{2\sigma}\int_r^\infty \frac{(1-\cos (\xi t))}{t^{1+2\sigma}}\d t,\\
\eta(1-\eta) r^{2\eta} \int_r^\infty \frac{(1-\cos (\xi t))}{t^{1+2\eta}}\d t
&\leq 3\sigma(1-\sigma)  r^{2\sigma}\int_r^\infty \frac{(1-\cos (\xi t))}{t^{1+2\sigma}}\d t .
\end{align*}
with the extreme cases $\sigma=0$ and $\xi=\infty$ given by
\begin{align*}
&\lim_{\xi\to \infty}\sigma  \int_1^\infty \frac{(1-\cos (\xi t))}{t^{1+2\sigma}}\d t =
\frac{1}{2} \quad\text{and}\quad
\lim_{\sigma\to 0^+}\sigma  \int_1^\infty \frac{(1-\cos (\xi t))}{t^{1+2\sigma}}\d t =\begin{cases}
\frac{1}{2}&\xi>0\\
0&\xi=0.
\end{cases}
\end{align*}
\end{lemma}
\begin{proof}
Up to replacing $\xi$ by $r\xi$ we can assume by scaling argument that $r=1$ and by parity that $\xi\geq 0$.
 Define the function
\begin{align*}
f(\eta,\xi):=\eta  \int_1^\infty \frac{(1-\cos (\xi t))}{t^{1+2\eta}}\d t&= \eta \xi^{2\eta}  \int_{\xi}^\infty \frac{1-\cos (t)}{t^{1+2\eta}}\d t.
\end{align*}
In most cases we  will only show that $f(\eta,\xi)\leq 3 f(\sigma,\xi)$ as this readily implies
the second desired estimate $ (1-\eta)f(\eta,\xi)\leq 3 (1-\sigma)f(\sigma,\xi)$, since $1-\eta\leq 1-\sigma$. We begin by performing integration by parts, yielding
\begin{align}\label{eq:xexpansion-away-zero}
f(\eta,\xi)= \frac{1}{2}- \eta  \int_1^\infty \frac{\cos (\xi t)}{t^{1+2\eta}}\d t
&=\frac{1}{2} + \frac{\eta\sin (\xi)}{\xi} -\frac{\eta(1+2\eta)}{\xi}   \int_1^\infty \frac{\sin (\xi t)}{t^{2+2\eta}}\d t.
\end{align}
For fixed $\xi_0>0$  then we have
\begin{align}\label{eq:xuniform-away-zero}
\sup_{\xi\geq \xi_0}\Big| f(\eta,\xi)-\frac12\Big| \leq \frac{\eta}{\xi_0}+ \frac{\eta (1+2\eta)}{\xi_0}\int_1^\infty\frac{\d t}{t^{2+2\eta}}=
\frac{2\eta}{\xi_0}.
\end{align}
We note in passing that $f(\sigma, 0)=0$ while for $\xi>0$ the estimate \eqref{eq:xuniform-away-zero} readily implies
\begin{align*}
f(0,\xi)=& \lim_{\sigma\to 0^+}\sigma  \int_1^\infty \frac{(1-\cos (\xi t))}{t^{1+2\sigma}}\d t
=\frac{1}{2},
\\
f(\sigma, \infty)=&\lim_{\xi\to \infty}\sigma  \int_1^\infty \frac{(1-\cos (\xi t))}{t^{1+2\sigma}}\d t =
\frac{1}{2}.
\end{align*}
From now on,  we take $\xi_0=8$ hence  we arrive at
\begin{align}\label{eq:xgen-bound}
\frac{1}{2}- \frac{\eta}{4}\leq f(\eta,\xi)\leq \frac{1}{2}+ \frac{\eta}{4}&& \text{for all $\eta\in (0,1) $ and $\xi\geq 8$}.
\end{align}
In particular, since $0\leq\eta\leq 1$ we find that
\begin{align*}
\frac{1}{4}\leq f(\eta,\xi)\leq \frac{3}{4}&& \text{for all $\eta\in (0,1) $ and $\xi\geq 8$}.
\end{align*}
The latter estimate clearly implies
\begin{align*}
f(\eta,\xi)\leq 3f(\sigma,\xi) && \text{for all $\sigma, \eta\in (0,1) $ and $\xi\geq 8$}.
\end{align*}
Therefore, it remains to deal with the case $0\leq \xi\leq 8$ distinguished in tree cases.
\vspace{1mm}

\textbf{Case 1: }
Assume $0\leq \sigma\leq \eta\leq \frac{1}{4}$ and $0\leq \xi\leq 8$. Let us obverse that
\begin{align*}
f(\eta,\xi)=\eta  \xi^{2\eta}  \int_{\xi}^\infty \frac{1-\cos (t)}{t^{1+2\eta}}\d t\leq \frac{\eta \xi^{2\eta}}{2}\int_{\R} \frac{1-\cos (t)}{t^{1+2\eta}}\d t=  \frac{\eta \xi^{2\eta} }{2 C_{1,\eta}} ,
\end{align*}
where $C_{1,\eta}$,  is normalization constant of the one dimensional fractional Laplacian, see \cite[Section 2.3]{guy-thesis} or \cite[Chapter 1]{BuVa16},  given by
\begin{align*}
\frac{1}{C_{1,\eta}} = \int_{\R} \frac{1-\cos (t)}{t^{1+2\eta}}\d t =\frac{\pi^{1/2} \Gamma(1 - \eta)}{2^{2\eta} \eta\Gamma\left( \frac{1 + 2\eta}{2} \right)}=  \frac{ \pi}{ \Gamma(1+2\eta) \sin(\pi \eta)}.
\end{align*}
The last expression follows from   the Schwartz's reflection and duplication formulas
\begin{align*}
\Gamma(1-\eta)\Gamma (\eta)= \frac{\pi}{\sin (\pi\eta)}\qquad \text{ and }\qquad  \Gamma(2\eta)= \frac{2^{2\eta} \Gamma(\eta)\Gamma(\tfrac{1+2\eta}{2})}{2\pi^{1/2}}.
\end{align*}
Note that $\frac{\sqrt{\pi}}{2}= \Gamma(1+\tfrac{1}{2})\leq \Gamma(1+2x)\leq 1$  for $x\in (0,\frac{1}{4})$,  while since $x\mapsto \frac{\sin (x)}{x} $ is decreasing on $ [0, \frac{\pi}{4}]$ we have,
\begin{align*}
1\leq \frac{x}{\sin (x)}\leq \frac{\pi\sqrt{2}}{4},\qquad x\in [0, \frac{\pi}{4}].
\end{align*}
Inserting this in the previous expression,  implies that
\begin{align}\label{eq:eta-less-quart-upb}
f(\eta,\xi)\leq \frac{ \pi\eta \xi^{2\eta}}{ 2\Gamma(1+2\eta) \sin(\pi \eta)}\leq \frac{\sqrt{2\pi}}{4}\xi^{2\eta}
&&\quad \text{ for all  \,\, $0\leq \eta\leq \frac{1}{4}$ and $\xi\geq 0$}.
\end{align}
Now, since $0\leq \xi\leq \xi_0=8$ and $0\leq \sigma\leq \frac{1}{4}$, the lower estimate in \eqref{eq:xgen-bound} implies
\begin{align*}
f(\sigma,\xi)\geq \sigma \xi^{2\sigma}  \int_{\xi_0}^\infty \frac{1-\cos (t)}{t^{1+2\sigma}}\d t= \big(\frac{\xi}{\xi_0}\big)^{2\sigma}f(\sigma,\xi_0)\geq\big(\frac{\xi}{\xi_0}\big)^{2\sigma} \frac{2-\sigma}{4}\geq \frac{7}{16} \big(\frac{\xi}{\xi_0}\big)^{2\sigma}.
\end{align*}
In other words, we have 
\begin{align}\label{eq:sigma-less-quart-lob}
f(\sigma,\xi)\geq\big(\frac{\xi}{\xi_0}\big)^{2\sigma} \frac{2-\sigma}{4}\geq \frac{7}{16} \big(\frac{\xi}{\xi_0}\big)^{2\sigma}&&\text{$0\leq \sigma\leq \frac{1}{4}$ and $0\leq \xi\leq \xi_0=8$ }.
\end{align}

\begin{itemize}
\item
If $0\leq \sigma\leq \eta\leq \frac{1}{4}$ and $0\leq \xi\leq 3$, the estimates \eqref{eq:eta-less-quart-upb} and \eqref{eq:sigma-less-quart-lob} imply
\begin{align*}
\frac{f(\eta,\xi) }{f(\sigma,\xi)}\leq\xi^{2\eta} \big(\frac{\xi}{\xi_0}\big)^{2(\eta-\sigma)}
\frac{4\sqrt{2\pi}}{7}\leq \xi^{1/2}
\frac{4\sqrt{2\pi}}{7}\leq \frac{4\sqrt{6\pi}}{7}<\frac{5}{2}<3.
\end{align*}
\item If $0\leq \sigma\leq  \eta\leq\frac{1}{4}$
and $3\leq \xi\leq 8$, the lower estimate  in \eqref{eq:xuniform-away-zero} implies that
\begin{align}\label{eq:sigma-less-quart-lob-bis}
f(\sigma,\xi)\geq\big( \frac{1}{2}-\frac{2\sigma}{\xi}\big)\geq \big( \frac{1}{2}-\frac{1}{6}\big)=\frac{1}{3}&& \text{for all $0\leq \sigma\leq \frac{1}{4}$ and $\xi\geq 3$}.
\end{align}
This together with the fact that $ f(\eta,\xi) \leq 1$ implies
\begin{align*}
f(\eta,\xi) \leq 1\leq3f(\sigma,\xi).
\end{align*}
\end{itemize}
\textbf{Case 2: } Assume  that $\frac{1}{4}\leq\sigma\leq \eta\leq 1$ and $0\leq \xi\leq 8$.
\begin{itemize}
\item Assume $\frac{1}{4}\leq\sigma\leq \eta\leq \frac{3}{4}$ or  $\frac{3}{4}\leq\sigma\leq \eta\leq 1$ and $0\leq \xi\leq 8$. Since $t^{-2\eta}\leq t^{-2\sigma} $, $t\geq 1$ and $\frac{1}{3}\leq \frac{\eta}{\sigma} \leq 3$, it follows that
\begin{align*}
f(\eta,\xi)\leq 3f(\sigma,\xi)  && \text{for all $\xi\geq 0$}.
\end{align*}
\item Assume $\frac{1}{3}\leq\sigma\leq  \frac{3}{4}\leq \eta\leq 1$ and $0\leq \xi\leq 8$. Since  $\frac{1}{3}\leq \frac{\eta}{\sigma} \leq 3$ and $t^{-2\eta}\leq t^{-2\sigma} $, $t\geq 1$, it follows that
\begin{align*}
f(\eta,\xi)\leq 3f(\sigma,\xi)  && \text{for all $\xi\geq 0$}.
\end{align*}
\item  Assume  $\frac{1}{4}\leq\sigma\leq \frac{1}{3}$ and $ \frac{3}{4}\leq \eta\leq 1$  and $0\leq \xi\leq 8$. Analogously, it follows that
\begin{align*}
f(\eta,\xi)\leq 4f(\sigma,\xi)  && \text{for all $ \xi\geq 0$}.
\end{align*}
On the other hand, using again  $t^{-2\eta}\leq t^{-2\sigma}$, $t\geq1$, we immediately have
\begin{align*}
(1-\eta) f(\eta, \xi)\leq (1-\sigma) f(\sigma,\xi),&& \text{for all $ \xi\geq 0$},
\end{align*}
where we note that $\frac{\eta(1-\eta)}{\sigma(1-\sigma)}\leq 1$ since
\begin{align*}
\max_{\tau\in\left[\frac{3}{4},\,1\right]} \tau(1-\tau)
= \frac{3}{16} =
\min_{\tau\in\left[\frac{1}{4},\,\frac{1}{3}\right]} \tau(1-\tau).
\end{align*}

\end{itemize}
\textbf{Case 3: } Assume $0\leq \sigma\leq \frac{1}{4}\leq \eta\leq 1$ and $0\leq \xi\leq 8$.
\begin{itemize}
\item
If $0\leq \sigma\leq \frac{1}{4}\leq \eta\leq 1$ and $3\leq \xi\leq 8$, the estimate \eqref{eq:xuniform-away-zero} yields
\begin{align}\label{eq:sigma-less-quart-lob-bis}
f(\sigma,\xi)\geq\big( \frac{1}{2}-\frac{2\sigma}{\xi}\big)\geq \big( \frac{1}{2}-\frac{1}{6}\big)=\frac{1}{3}&& \text{for all $0\leq \sigma\leq \frac{1}{4}$ and $\xi\geq 3$}.
\end{align}
This together with the fact that $ f(\eta,\xi) \leq 1$
implies
\begin{align*}
f(\eta,\xi) \leq 1\leq 3f(\sigma,\xi).
\end{align*}
\item
If $0\leq \sigma\leq \frac{1}{4}\leq \eta\leq 1$ and $0\leq \xi\leq 3$. We will show that $f(\eta,\xi)\leq 4f(\sigma,\xi)$ which implies
$(1-\eta)f(\eta,\xi)\leq 3(1-\sigma)f(\sigma,\xi)$ since $(1-\eta)\leq \frac{3}{4}\leq (1-\sigma)$. Let us observe that
\end{itemize}
\begin{align*}
f(\eta, \xi)
&= \eta \int_{1}^{\infty}  (1-\cos(\xi t))
\left( \frac{1}{\Gamma(1+2\eta)} \int_{0}^{\infty} x^{2\eta} e^{-xt} \, \d x \right)  \d t \\
&=  \frac{1}{2} \int_{0}^{\infty} \frac{x^{2\eta-1} }{\Gamma(2\eta)}  \left( \int_{1}^{\infty} (1-\cos(\xi t)) xe^{-xt} \, \d t \right)  \d x \\
&=  \frac{1}{2} \int_{0}^{\infty} K(x,\xi) \frac{x^{2\eta-1}}{\Gamma(2\eta)} e^{-x} \d x
\end{align*}
where we consider the kernel
\begin{align*}
K(x,\xi)=e^x \operatorname{Re}\left( \int_{1}^{\infty} (1-e^{i\xi t})  xe^{-xt} \, \d t \right) =
1- \frac{x^2\cos (\xi) -x\xi  \sin(\xi)}{x^2+\xi^2 }.
\end{align*}
\noindent Cauchy-Schwartz inequality implies $-\xi\sin\xi + x\cos\xi
\leq \sqrt{x^2+\xi^2}$, so that
\begin{align*}
\frac{x\xi\sin\xi}{x^2+\xi^2}
\leq
\frac{x}{\sqrt{x^2+\xi^2}}- \frac{x^2\cos\xi}{x^2+\xi^2}
\leq 1-\frac{x^2\cos\xi}{x^2+\xi^2}.
\end{align*}
Since $\sin\xi\geq 0$ as  $0\leq \xi\leq 3<\pi$ if follows that
\begin{align*}
 g(x, \xi) \leq K(x,\xi)= g(x,\xi) +
\frac{x\xi\sin\xi}{x^2+\xi^2}\leq 2g(x, \xi)
\end{align*}
where $g(\cdot,\xi)$ is given by
\begin{align*}
g(x,\xi)= 1-\frac{x^2\cos\xi}{x^2+\xi^2}= 1-\cos\xi + \frac{\xi^2\cos\xi}{x^2+\xi^2}.
\end{align*}
If $0\leq \xi\leq \frac{\pi}{2}$ so that  $\cos \xi\geq 0$ then clearly $g(x,\xi)$ is decreasing. Hence
\begin{align*}
K(y,\xi)\leq 2K(x,\xi)\qquad \text{$x\leq y$ and $0\leq \xi\leq \frac{\pi}{2}$}.
\end{align*}
If $\frac{\pi}{2}<\xi\leq \pi,$ then $g(x,\xi)$  is increasing, $g(\infty,\xi) =2(1-\cos\xi ) g(0, \xi) \leq 4  g(x,\xi)$. Then
\begin{align*}
 K(y,\xi)\leq 2 g(y,\xi)\leq 2 g(\infty,\xi) \leq 4  g(x,\xi) \leq 4 K(x,\xi)
\end{align*}
for all $x,y\in (0,\infty)$. Altogether, we deduce that $K(\cdot,\xi)$ is $4$-almost decreasing for $0\leq \xi\leq 3$, i.e., we have $K(y,\xi)\leq 4K(x,\xi) $ for $0<x<y$.  The  family $(\frac{x^{2\eta-1}}{\Gamma(2\eta)} e^{-x})_\eta$  of probability densities function of  $(2\eta,1)$-Gamma distributions  on $(0,\infty)$ and
possesses the monotone likelihood ratio property in $x$, i.e.,  for $\sigma<\eta$, the mapping
\begin{align*}
x\mapsto  \frac{\Gamma(2\sigma)}{\Gamma(2\eta)}x^{2(\eta-\sigma)}
\quad\text{is increasing on $(0,\infty)$}.
\end{align*}
Accordingly, Theorem \ref{thm:mon-like-ratio} implies
that  $f(\eta,\xi)\leq 4f(\sigma,\xi)$.
\end{proof}

\begin{theorem}
\label{thm:r-split-monot-semin-l2}
Let $u\in L^2(\R^d)$ and $r>0$. The following estimates  hold
\begin{align*}
L_2(\eta, u,r)&\leq  L_2(s, u,r)&&0\leq  \eta<s\leq 1,\\
T_2(\eta, u,r)&\leq 2 T_2(\sigma , u,r), && 0\leq \sigma< \eta \leq 1,\\
(1-\eta)^{1/2}T_2(\eta, u,r)&\leq \sqrt{3} (1-\sigma)^{1/2}T_2(\sigma , u,r),
&& 0\leq \sigma< \eta \leq 1.
\end{align*}
\end{theorem}
\begin{proof}
Using the Fourier transform, Plancherel's identity implies
\begin{align*}
U(t)&= \int_{ \mathbb{S}^{d-1}} \int_{\R^d} | u(x+\omega t)-u(x)|^2\d x\d \sigma_{d-1}(\omega)\\
&= 2\int_{\mathbb{S}^{d-1}}  \int_{\R^d} |\widehat{u}(\xi)|^2(1-\cos (|\omega \cdot \xi| t))\d \xi\d \sigma_{d-1}(\omega).
\end{align*}
Accordingly,  Lemma \ref{lem:short-cos} and Lemma \ref{lem:tail-short-cos} imply all desired estimates since
\begin{align*}
T^2_2 (\eta, u,r) &=2 \eta \int_{\mathbb{S}^{d-1}}  \int_{\R^d}  |\widehat{u}(\xi)|^2 r^{2\eta}\int_r^\infty \frac{(1-\cos (|\omega\cdot \xi| t))}{t^{1+2\eta}}\d t \d \xi \d \sigma_{d-1}(\omega),\\
L^2_2 (\eta, u,r) &=2 (1-\eta )\int_{\mathbb{S}^{d-1}}  \int_{\R^d}  |\widehat{u}(\xi)|^2 (2r)^{2\eta}\int_0^r \frac{(1-\cos (|\omega\cdot \xi| t))}{t^{1+2\eta}}\d t \d \xi \d \sigma_{d-1}(\omega).
\end{align*}
\end{proof}
\begin{proof}[Proof of Theorem \ref{thm:r-split-monot-semin}]
 The extreme cases $\sigma=0$ and $s=1$ follow by letting $\sigma\to 0^+$ and $s\to1^-$ using Theorem \ref{thm:mazy'a-bbm-limit}, thus we only consider $0<\sigma<\eta<s<1$.
The continuities are not difficult to establish. By Theorem \ref{thm:lp-first-order-diff}
we know that $t \mapsto g(t) = \frac{U(t)}{t^p}$ is dyadically decreasing, i.e.,  $g(2t) \leq g(t)$ and satisfies
\begin{align*}
g(0)=\lim_{t\to0^+} g(t)= \frac{|\mathbb{S}^{d-1}|}{p}K_{d,p} \|\nabla u\|^p_{L^p(\R^d)} =[u]^p_{W^{1,p}(\R^d)}.
\end{align*}
Therefore,  for $0<\eta<s\leq 1$ and $1\leq p<\infty$,  Theorem~\ref{thm:cadic-monotonicity} implies that
\begin{align*}
(1-\eta)(2r)^{-(1-\eta )p}\int_0^r g(t)t^{(1-\eta) p-1} \d t
&\leq (1-s) (2r)^{-(1-s) p}\int_0^r g(t)t^{(1-s) p-1}\d t.
\end{align*}
In other words we have
\begin{align*}
(1-\eta)(2r)^{-(1-\eta) p}\int_0^r \frac{U(t)}{t^{1+\eta p}} \d t
&\leq (1-s) (2r)^{-(1-s) p}\int_0^r \frac{U(t)}{t^{1+s p}}\d t.
\end{align*}
That is, $L_p(\eta, u,r) \leq L_p(s, u,r)$. Now if $p=\infty$  and $0<\eta<s\leq  1$, we get
\begin{align*}
L_\infty(\eta, u,r)&=
(2r)^{\eta-s} (2r)^{-(1-s)} \sup_{0<|x-y|\leq r} \frac{|u(x)-u(y)|}{|x-y|^s} |x-y|^{s-\eta}\\
&\leq 2^{\eta-s}L_\infty(s, u,r)\leq L_\infty(s, u,r).
\end{align*}
In particular, $ 2^{-\eta}L_\infty(\eta, u,r) \leq 2^{-s}L_\infty(s, u,r)$. 
Next,  we turn our attention to long range. Since   $U(t)\leq 2^p |\mathbb{S}^{d-1}| \|u\|^p_{L^p(\R^d)}$ by Theorem \ref{thm:lp-first-order-diff}, for $\sigma=0$ we get 
\begin{align*}
T_p(\eta, u,r)\leq \Big(2^{p-1}\frac{2|\mathbb{S}^{d-1}|}{p}\|u\|^p_{L^p(\R^d)}
\Big)^{\frac{1}{p}}=2^{1-\frac{1}{p}}T_p(0, u,r)\leq 2T_p(0, u,r).
\end{align*}
For  $p=\infty$,  we get  $T_\infty(\eta, u, r)\leq 2T_\infty(0, u,r)$, while for $0<\sigma<\eta<1$ we have
\begin{align*}
T_\infty(\eta, u,r)&= r^{\sigma}r^{\eta-\sigma}\sup_{|x-y|\geq r} \frac{|u(x)-u(y)|}{|x-y|^{\sigma}}|x-y|^{\sigma-\eta} \leq T_\infty(\sigma, u,r).
\end{align*}
The case $p=2$ and $d\geq 1$ follows immediately from Lemma \ref{lem:tail-short-cos}.
Now assume $d \geq 2$. In view of Theorem \ref{thm:lp-first-order-diff} we know that
\begin{align*}
U(\infty)=\lim_{t\to^\infty} U(t)= \frac{2|\mathbb{S}^{d-1}|}{p} \|u\|^p_{L^p(\R^d)} =[u]^p_{W^{0,p}(\R^d)},
\end{align*}
and $U$ is $2^p$-almost increasing, i.e., $U(s) \leq 2^p U(t)$ for all $s < t$. Accordingly, Theorem~\ref{thm:almost-monotone} implies  that for $0\leq \sigma<\eta $ we have $T_p(\eta, u,r) \leq 2 T_p(\sigma, u,r)$, that is,
\begin{align*}
\eta r^{\eta p}\int_r^\infty \frac{U(t)}{t^{1+\eta p}}\d t
&\leq 2^p\sigma r^{ \sigma p}\int_r^\infty \frac{U(t)}{t^{ 1+\sigma p}} \d t.
\end{align*}

Finally, it remains to show that  $T_p(\eta, u,r) \leq 2^{1+\frac{1}{p}}  r^\sigma(1-\sigma)^{-\frac{1}{p}}
[u]_{W^{\sigma,p}(\R^d)}$ for $d\geq1$,  $\sigma>0$ and $p\neq \infty$.  Since $z^{-\frac{1}{\eta p}} \leq z^{-\frac{1}{\sigma p}}$ for $z\in (0,1)$, Theorem \ref{thm:lp-first-order-diff} implies  $U(rz^{-\frac{1}{\eta p}})\leq 2^{p+1}AU(rz^{-\frac{1}{\sigma p}})$, yielding
\begin{align*}
\int_0^1 U(rz^{-\frac{1}{\eta p}}) \d z \leq 2^{p+1} \int_0^1 AU(rz^{-\frac{1}{\sigma p}})\d z .
\end{align*}
The latter estimate  is equivalent to saying
\begin{align*}
\eta r^{\eta p}\int_r^\infty \frac{U(t)}{t^{1+\eta p}}\d t
&\leq 2^{p+1}\sigma r^{\sigma p}\int_r^\infty \frac{AU(t)}{t^{ 1+\sigma p}} \d t.
\end{align*}
Now, applying Fubini's theorem to interchange the order of integration gives
\begin{align*}
 \int_r^\infty \frac{AU(t)}{t^{1+\sigma p}}\d t
&=
 \int_r^\infty
\Big(\frac{1}{t}\int_0^t U(z)\d z\Big)
\frac{\d t}{ t^{1+\sigma p}}=
 \int_0^\infty U(z) \Big( \int_{\max(r,z)}^\infty \frac{\d t}{
t^{2+\sigma p}}\Big)
\d z\\
&\leq   \int_0^\infty
U(z) \Big(\int_{z}^\infty t^{-(2+\sigma p)}\d t\Big)
\d z = \frac{1}{1+\sigma p}  \int_0^\infty
\frac{U(z)}{z^{1+\sigma p}} \d z.
\end{align*}
Altogether  yields the estimate $T_p(\eta, u,r) \leq 2^{1+\frac{1}{p}}  r^\sigma(1-\sigma)^{-\frac{1}{p}}
[u]_{W^{\sigma,p}(\R^d)}$ since 
\begin{align*}
\eta r^{\eta p}\int_r^\infty \frac{U(t)}{t^{1+\eta p}}\d t \leq2^{p+1} \sigma r^{\sigma p}  \int_r^\infty \frac{AU(t)}{t^{1+\sigma p}}\d t\leq
2^{p+1} \sigma r^{\sigma p} \int_0^\infty \frac{U(t)}{t^{1+\sigma p}}\d t.
\end{align*}
\end{proof}
 
\begin{corollary}[Characterization of $W^{s,p}(\R^d)$]
For $0< s\leq 1$, $r>0$ and $u\in L^p(\R^d)$, $1\leq p<\infty$ we have
\begin{align*}
L_p(s,u,r) = \sup_{0\leq \eta< s} L_p(\eta,u,r).
\end{align*}
Furthermore, with the understanding that the space $W^{1,1}(\mathbb{R}^d)$ is replaced by $BV(\R^d)$ when $s = 1$ and $p = 1$, we have that $u \in W^{s,p}(\R^d)$ if and only if
\begin{align*}
\sup_{0 \leq \eta < s} L_p(\eta,u,r) < \infty.
\end{align*}
\end{corollary}
\begin{proof}
Assume  $0<s<1$. By Theorem \ref{thm:r-split-monot-semin}, $s\mapsto L_p(s,u,r)$ is increasing. Therefore, using this and Fatou's Lemma imply that
\begin{align*}
L_p(s,u,r)\leq \liminf_{n\to\infty} L_p(s-\tfrac{1}{n},u,r)\leq \sup_{0\leq \eta< s} L_p(\eta,u,r)\leq L_p(s,u,r).
\end{align*}
That is, $L_p(s,u,r) = \sup_{0\leq \eta< s} L_p(\eta,u,r).$ In particular  if $u\in W^{s,p}(\R^d)$ then
\begin{align*}
\sup_{0\leq \eta< s} L_p(\eta,u,r)=L_p(s,u,r) <\infty.
\end{align*}
Conversely if $
\sup_{0\leq \eta< s} L_p(\eta,u,r)<\infty$ then $u\in W^{s,p}(\R^d)$. Indeed, since
$T_p^p(s, u,r)\leq 2^pT_p^p(0, u,r)$ with $T_p^p(0, u,r) =\frac{2}{p}|\mathbb{S}^{d-1}|\|u\|^p_{L^p(\R^d)}$ we have
\begin{align*}
[u]^p_{W^{s,p}(\R^d)}
&= (2r)^{(1-s)p} s L_p^p(s, u,r) + r^{-sp} (1-s)T_p^p(s, u,r)\\
&\leq (2r)^{(1-s)p} s \sup_{0\leq \eta< s} L^p_p(\eta,u,r) + r^{-sp} 2^pT_p^p(0, u,r)<\infty.
\end{align*}
The case $s = 1$, which provides a standard nonlocal characterization of $W^{1,p}(\R^d)$, is less obvious.
It follows from
\cite[Theorems 1.4 and 1.4$'$]{Fog23}; see also \cite{BBM01}.
\end{proof}
\subsection{Proof of Theorem \ref{thm:robust-interpo-esti} } \label{sec:robust-interpo-esti}
We also need the following elementary lemma.
\begin{lemma}
\label{lem:min-power-interpolation}
Given $a, b, \alpha, \beta > 0$,   the function $r\mapsto F(r) := a r^\alpha + b r^{-\beta},$ $r>0$ attains a unique global minimum at $r_* = \left( \frac{b \beta}{a \alpha} \right)^{\frac{1}{\alpha + \beta}}$ and satisfies
\begin{align*}
\min_{r > 0} F(r)
&= F(r_*) = (\alpha + \beta) \left( \frac{a}{\alpha} \right)^{\frac{\beta}{\alpha + \beta}} \left( \frac{b}{\beta} \right)^{\frac{\alpha}{\alpha + \beta}} = \kappa(\theta)\, a^{\theta} b^{1 - \theta},
\end{align*}
where $\theta = \frac{\beta}{\alpha + \beta} \in (0,1)$ and  the function $\theta \mapsto \kappa(\theta)=\theta^{-\theta}(1 - \theta)^{-(1-\theta)} $ satisfies  $\kappa(0^+) =\kappa(1^-) =1$, $\kappa(1/2)=2 $  and the universal bounds
\begin{align*}
1\leq \kappa(\theta) \leq 2 \quad \text{for all } \theta \in (0,1).
\end{align*}
Analogously the function $r\mapsto G(r) := \max(a r^\alpha, b r^{-\beta}),$ $r>0$ attains a unique global minimum at $r_0 = \left( \frac{b}{a} \right)^{\frac{1}{\alpha + \beta}}$ and satisfies
\begin{align*}
\min_{r > 0} G(r)
&= G(r_0) = a^{\frac{\beta}{\alpha + \beta}} b^{\frac{\alpha}{\alpha + \beta}} = a^{\theta} b^{1 - \theta}.
\end{align*}
\end{lemma}

\begin{proof}
The unique critical point is $ r_* = \big(\frac{b \beta}{a \alpha}\big)^{\frac{1}{\alpha + \beta}}$ satisfying  $F'(r_*) = 0$ with $F'(r) = a \alpha r^{\alpha - 1} - b \beta r^{-\beta - 1}.$
Moreover, evaluating  the second derivative $F''(r)$  at $r=r_*$ reveals that $r_*$ is a global minimum since
\begin{align*}
F''(r_*)
&= a \alpha (\alpha - 1) r_*^{\alpha - 2} + b \beta (\beta + 1) r_*^{-\beta - 2}
=  (a\alpha)^{\frac{\alpha+2}{\alpha + \beta} } (b\beta)^{\frac{\beta-2}{\alpha + \beta} } (\alpha+\beta)>0
\end{align*}
It follows that, the minimal value is given by
\begin{align*}
F(r_*)
&= a \Big(\frac{b \beta}{a \alpha}\Big)^{\frac{\alpha}{\alpha + \beta}} + b \Big(\frac{b \beta}{a \alpha}\Big)^{-\frac{\beta}{\alpha + \beta}}
=(\alpha + \beta)  \Big(\frac{a}{\alpha}\Big)^{\frac{\beta}{\alpha + \beta}}  \Big(\frac{b }{\beta}\Big)^{\frac{\alpha}{\alpha + \beta}}= \kappa(\theta)\, a^{\theta} b^{1 - \theta}.
\end{align*}
Next,  taking $\rho(\theta)= \log(\theta^{\theta} (1-\theta)^{(1-\theta)})$ we find that
\begin{align*}
\rho'(\theta)
&= \log \theta - \log(1 - \theta) \quad\text{and}\quad
\rho''(\theta) = \frac{1}{\theta} + \frac{1}{1 - \theta}>0.
\end{align*}
This implies that $\rho'$ is strictly increasing and  for all $\theta\in (0, \frac12)$ we have
\begin{align*}
-\infty=\lim_{\theta\to 0^+} \rho' (\theta)< \rho'(\theta)< \rho'(1/2)=0< \rho'(1-\theta)<\lim_{\theta\to 0} \rho' (1-\theta)=\infty.
\end{align*}
Thus,  $\rho$ is strictly decreasing on $(0, \tfrac{1}{2})$ and strictly increasing on $(\tfrac{1}{2}, 1)$, yielding that  the  minimum is achieved at $\theta = \tfrac{1}{2}$  and the maximum at $\theta \in \{0,1\}$. We find that
\begin{align*}
\frac{1}{2}\leq \theta^{\theta} (1-\theta)^{(1-\theta)} \leq 1 \quad \text{for all } \theta \in (0,1).
\end{align*}
Therefore, since $1/\kappa(\theta) =
\theta^{\theta}(1 - \theta)^{1-\theta}$ we deduce
\begin{align*}
\kappa(0^+) =\kappa(1^-) =1\leq \kappa(\theta) \leq 2=\kappa(1/2) \quad \text{ for all }
\theta \in (0,1).
\end{align*}
The claim for $G$ follows since $G$ is decreasing on $(0, r_0)$ and increasing on $(r_0,\infty)$.
\end{proof}

\begin{proof}[Proof of Theorem \ref{thm:robust-interpo-esti}]
Let $0\leq \sigma<\eta<s\leq 1$ and $\theta=\tfrac{\eta-\sigma}{s-\sigma}\in(0,1)$. We only prove the estimates for $\R^d$ as those of $\R^d_+$ follow by even reflection; see Lemma~\ref{lem:even-reflec}. For $1 \leq p < \infty$ and $r > 0$, setting $\rho = r^{s-\sigma}$ in Theorem~\ref{thm:r-split-monot-semin} implies that
\begin{align*}
(1-\eta)^{\frac{1}{p}} T_p(\eta, u, r) \leq 2\varkappa r^{\sigma} [u]_{W^{\sigma,p}(\R^d)} \quad \text{and} \quad s^{\frac{1}{p}} L_p(\eta, u, r) \leq (2r)^{-(1-s)} [u]_{W^{s,p}(\R^d)},
\end{align*}
where we recall $\tiny{\varkappa = \begin{cases} 2^{\frac{1}{p}}& \text{if } d = 1, \\ 1 & \text{if } d \geq 2. \end{cases}}$. Therefore, we deduce
\begin{align*}
[u]^p_{W^{\eta,p}(\R^d)}
&=\eta(1-\eta)\int_0^r \frac{U(t)}{t^{1+\eta p}} \d t + \eta(1-\eta)\int_r^\infty \frac{U(t)}{t^{1+\eta p}} \d t\\
&= (2r)^{(1-\eta)p} \eta L_p^p(\eta, u,r) + r^{-\eta p} (1-\eta)T_p^p(\eta, u,r) \\
&\leq (2r)^{(1-\eta)p} sL_p^p(s, u,r) + 2^pr^{-\eta p} (1-\sigma)T_p^p(\sigma, u,r) \\
& \leq 2^{(s-\eta) p}r^{(s-\eta) p}  [u]^p_{W^{s,p}(\R^d)}+ 2^p\varkappa^p r^{-(\eta-\sigma) p}  [u]^p_{W^{\sigma,p}(\R^d)}\\
&=2^{(s-\sigma)(1-\theta)p} \rho^{(1-\theta)p }  [u]^p_{W^{s,p}(\R^d)}+ 2^p\varkappa^p\rho^{-\theta p}  [u]^p_{W^{\sigma,p}(\R^d)},
\end{align*}
According to  Lemma \ref{lem:min-power-interpolation} the value $\rho_*$ that minimizes the right-hand side is
\begin{align*}
\rho_*
& = \Big( \frac{2^{p} \varkappa^p \theta }{2^{(s-\sigma)(1-\theta)p}(1-\theta) } \frac{[u]^p_{W^{\sigma,p}(\R^d)}}{[u]^p_{W^{s,p}(\R^d)}} \Big)^{\frac{1}{p}}.
\end{align*}
Moreover, for   this choice $\rho_*$, the minimal value of the right-hand side becomes:
\begin{align*}
[u]_{W^{\eta,p}(\R^d)}
&\leq	\min_{\rho>0}\Big(2^{(s-\sigma)(1-\theta)p} \rho^{(1-\theta)p }  [u]^p_{W^{s,p}(\R^d)}+ 2^p\rho^{-\theta p }  [u]^p_{W^{\sigma,p}(\R^d)}\Big)^{\frac{1}{p}} \\
&= \varkappa 2^{(s-\eta)\theta+(1-\theta)}\kappa^{\frac{1}{p}}(\theta) [u]_{W^{s,p} (\R^d)}^{\theta } \cdot
[u]_{W^{\sigma,p} (\R^d)}^{1 - \theta}\\
&=\varkappa N_p(\sigma, \eta, s) [u]_{W^{s,p} (\R^d)}^{\theta } \cdot
[u]_{W^{\sigma,p} (\R^d)}^{(1 - \theta)}.
\end{align*}
where we note that  $\kappa(\theta)=\theta^{-\theta}(1-\theta)^{-(1-\theta)}\in [1,2]$ and define
\begin{align*}
N_p(\sigma,\eta,s)
=2^{(s-\eta)\theta+(1-\theta)}
\kappa(\theta)^{\frac{1}{p}}
= 2^{\frac{(s-\eta)}{s-\sigma}(\eta+1-\sigma)}
\left(\frac{s-\sigma}{\eta-\sigma}\right)^{\frac{\eta-\sigma}{(s-\sigma)p}}
\left(\frac{s-\sigma}{s-\eta}\right)^{\frac{s-\eta}{(s-\sigma)p}},
\end{align*}
which satisfies   $1\leq N_p(\sigma, \eta, s)\leq 2^{1+\frac{1}{p}}$.  Analogously for the case $p=\infty$ we have
\begin{align*}
|u|_{W^{\eta,\infty}(\R^d)} &=
\max\Big(\sup_{|x-y|< r}\frac{| u(x)-u(y)|}{|x-y|^\eta}, \,\,  \sup_{|x-y|\geq r}\frac{| u(x)-u(y)|}{|x-y|^\eta}\Big)\\
&= \max \Big((2r)^{(1-\eta)}  L_\infty(\eta, u,r) ,\,\,   r^{-\eta} T_\infty(\eta, u,r)\Big) \\
&\leq  \max\Big((2r)^{(1-\eta)}  L_\infty(s, u,r) ,  \,\, 2r^{-\eta} T_\infty(\sigma, u,r)\Big) \\
&\leq \max \Big((2r)^{s-\eta}[u]_{W^{1,\infty}(\R^d)}, \,\,  2r^{\sigma-\eta}[u]_{W^{0,\infty}(\R^d)}\Big)\\
&=\max\Big(2^{(s-\sigma)(1-\theta)} \rho^{1-\theta}  [u]_{W^{s,\infty}(\R^d)}, \,\, 2\rho^{-\theta }  [u]_{W^{\sigma,\infty}(\R^d)}\Big).
\end{align*}
By Lemma \ref{lem:min-power-interpolation} taking $
\rho= \rho_*= 2^{\eta+ 1-s}\frac{ [u]_{W^{\sigma,\infty}(\R^d)}}{ [u]_{W^{s,\infty}(\R^d)}}$
yields that
\begin{align*}
|u|_{W^{\eta,\infty}(\R^d)}
&\leq \min_{\rho>0} \max\Big(2^{(s-\sigma)(1-\theta)} \rho^{1-\theta}  [u]_{W^{s,\infty}(\R^d)}, \,\, 2\rho^{-\theta }  [u]_{W^{\sigma,\infty}(\R^d)}\Big)\\
&=2^{(s-\eta)\theta+(1-\theta)}[u]_{W^{s,\infty} (\R^d)}^{\theta } \cdot
[u]_{W^{\sigma,\infty} (\R^d)}^{1 - \theta}\\
&= N_\infty(\sigma, \eta, s)[u]_{W^{s,\infty} (\R^d)}^{\theta } \cdot
[u]_{W^{\sigma,\infty} (\R^d)}^{1 - \theta}.
\end{align*}
On the other side, given that $[u]^{*\,p}_{W^{t,p}(\R^d)}= B_{d,p,t} [u]^p_{W^{t,p}(\R^d)}$, it merely follows that
\begin{align*}
[u]^*_{W^{\eta,p} (\R^d)}
&\leq \Upsilon^*_{d,p}(\sigma, \eta, s)
[u]^{*\, \theta }_{W^{s,p} (\R^d)} \cdot
[u]^{*\,(1 - \theta)}_{W^{\sigma,p} (\R^d)},
\end{align*}
with $\Upsilon^*_{d,p}(\sigma, \eta, s)= N_p(\sigma, \eta, s) \Big(\frac{B_{d,p,\eta}}{B_{d,p,\sigma}^{\theta}B_{d,p,s}^{1-\theta}} \Big)^{\frac{1}{p}}$. Likewise, by substituting the expression of $[\cdot]_{W^{t,p}(\R^d)}$ in term of $|\cdot|_{W^{t,p}(\R^d)}$ yields:

\textbf{Case $\sigma\neq 0,\,\, s\neq 1$.}
\begin{align*}
|u|_{W^{\eta,p}(\R^d)} \leq
\frac{ N_p(\sigma,\eta,s)}{[\eta(1-\eta)]^{\frac{1}{p}}}  \Big[ (s(1-s))^\theta (\sigma(1-\sigma))^{1-\theta}\Big]^{\frac{1}{p}}
|u|_{W^{s,p}(\R^d)}^{\theta} |u|_{W^{\sigma,p}(\R^d)}^{1-\theta}.
\end{align*}
\textbf{Case  $\sigma=0, s\neq 1$.}
\begin{align*}
|u|_{W^{\eta,p}(/R^d)} \leq\frac{ N_p(\sigma,\eta,s)}{[\eta(1-\eta)]^{\frac{1}{p}}}  \Big[ (s(1-s))^\theta (2|\mathbb{S}^{d-1}|/p)^{1-\theta} \Big]^{\frac{1}{p}}
|u|_{W^{s,p}(\R^d)}^\theta |u|_{W^{0,p}(\R^d)}^{1-\theta}.
\end{align*}
\textbf{Case $\sigma\neq 0,\,\,  s=1$.}
\begin{align*}
|u|_{W^{\eta,p}(\R^d)} \leq \frac{ N_p(\sigma,\eta,s)}{[\eta(1-\eta)]^{\frac{1}{p}}}  \Big[ (K_{d,p} |\mathbb{S}^{d-1}|/p)^\theta (\sigma(1-\sigma))^{1-\theta}\Big]^{\frac{1}{p}}
|u|_{W^{1,p}(\R^d)}^\theta |u|_{W^{\sigma,p}(\R^d)}^{1-\theta}.
\end{align*}
\textbf{Case $\sigma=0, s=1$.}
\begin{align*}
|u|_{W^{\eta,p}(\R^d)} \leq \frac{ N_p(\sigma,\eta,s)}{[\eta(1-\eta)]^{\frac{1}{p}}}
 \Big[  \frac{|\mathbb{S}^{d-1}|}{p}K_{d,p}^\theta 2^{1-\theta} \Big]^{\frac{1}{p}}
|u|_{W^{1,p}(\R^d)}^\theta |u|_{W^{0,p}(\R^d)}^{1-\theta}.
\end{align*}
\end{proof}

The following corollary is a direct consequence of Theorem~\ref{thm:robust-interpo-esti} and provides a stronger resolution to \cite[Open Problem 1.10]{BSY22} where one notes that $s(1-s)\leq\min(s,1-s)\leq 2s(1-s)$, $s\in [0,1]$;  see also the alternative answer in \cite[Theorem 3.9]{DoMi23}.
\begin{corollary}
\label{cor:open-question}
For $u\in L^p(\R^d)$ and $0\leq \sigma<\eta<s\leq 1$ we have
\begin{align*}
[u]_{W^{\eta,p}(\R^d)}
&\leq
\varkappa N_p(\sigma,\eta,s)
\,\big([u]^p_{W^{\sigma,p}(\R^d)} +[u]^p_{W^{s,p}(\R^d)}\big)^{\frac{1}{p}}\\
[u]_{W^{\eta,p}(\R^d_+)}
&\leq
4^{1/p}\varkappa N_p(\sigma,\eta,s)
\big( [u]^p_{W^{\sigma,p}(\R^d_+)} + [u]^p_{W^{s,p}(\R^d_+)}
\big)^{\frac{1}{p}},
\end{align*}
where we recall that $N_p(\sigma,\eta,s)\leq 2^{1+\frac{1}{p}}$.  In particular for $\sigma=0$ we have
\begin{align*}
[u]_{W^{\eta,p}(\R^d)} &\leq
2^{1+\frac{2}{p}} (2|\mathbb{S}^{d-1}|+1)
\,\big(\|u\|^p_{L^p(\R^d)} +[u]^p_{W^{s,p}(\R^d)}\big)^{\frac{1}{p}}\\
[u]_{W^{\eta,p}(\R^d_+)} &\leq
2^{1+\frac{4}{p}} (2|\mathbb{S}^{d-1}|+1)
\big( \|u\|^p_{L^p(\R^d_+)} + [u]^p_{W^{s,p}(\R^d_+)}
\big)^{\frac{1}{p}}.
\end{align*}
\end{corollary}


\begin{theorem}\label{thm:robust-interpo-dom}
Let $\Omega \subset \R^d$ be a robust $W^{s,p}$-extension domain with corresponding constant $C = C(d,p,\Omega) > 0$ as in Theorem \ref{thm:robust-extension} wherein we consider the norm
\begin{align*}
\|u\|_{W^{s,p}(\Omega)}&:=  \Big(\|u\|^p_{L^p(\Omega)}+ [u]^p_{W^{s,p}(\Omega)} \Big)^{1/p}\qquad s\in [0,1].
\end{align*}
 Then  for $0\leq\sigma <\eta< s\leq1$ there holds the estimate
\begin{align*}
\|u\|_{W^{\eta,p}(\Omega)}
\leq C\,\big(\varkappa N_p(\sigma,\eta,s)+1\big)\,
\|u\|_{W^{s,p}(\Omega)}^{\theta}
\|u\|_{W^{\sigma,p}(\Omega)}^{1-\theta}\leq 9C
\|u\|_{W^{s,p}(\Omega)}^{\theta}
\|u\|_{W^{\sigma,p}(\Omega)}^{1-\theta}.
\end{align*}
This holds in  particular,  if $\Omega = \R^d_+$ or if $\partial\Omega$ is compact and Lipschitz.
\end{theorem}
\begin{proof}
By Theorem \ref{thm:robust-extension} (see Appendix \ref{sec:robust-exten}), there exist a  $C=C(d,p,\Omega)>0$ and  $\overline{u}\in L^p(\R^d)$ a suitable extension of $u\in L^p(\Omega)$ such that $\overline{u}|_\Omega=u$ a.e. in $\Omega$ and $\|\overline{u}\|_{W^{t, p}(\R^d)}\leq C \|u\|_{W^{t, p}(\Omega)}$ for all $t\in [0,1]$.
 In view of Theorem \ref{thm:robust-interpo-esti} we find that
\begin{align*}
\|u\|_{W^{\eta, p}(\Omega)}
&=\Big(\|u\|^p_{L^p(\Omega)}+ [u]^p_{W^{\eta,p}(\Omega)} \Big)^{1/p}\leq  \Big(\|u\|^p_{L^p(\Omega)}+ [\overline{u}]^p_{W^{\eta,p}(\R^d)} \Big)^{1/p}
\\
&\leq \Big(\|u\|^p_{L^p(\Omega)}+N_p(\sigma, \eta, s)
[\overline{u}]^\theta_{W^{s, p}(\R^d)} [\overline{u}]^{1-\theta}_{W^{\sigma, p}(\R^d)}\Big)^{1/p}\\
&\leq  \Big(\|u\|^{\theta p}_{L^p(\Omega)}\|u\|^{(1-\theta)p}_{L^p(\Omega)}+ CN^p_p(\sigma, \eta, s)\|u\|^{\theta p}_{W^{s, p}(\Omega)}\|u\|^{(1-\theta)p}_{W^{\sigma, p}(\Omega)} \Big)^{1/p}\\
&\leq C(N_p(\sigma, \eta, s)+1)
\|u\|^{\theta }_{W^{s, p}(\Omega)}\|u\|^{1-\theta}_{W^{\sigma, p}(\Omega)} \leq 9C
\|u\|^{\theta }_{W^{s, p}(\Omega)}\|u\|^{1-\theta}_{W^{\sigma, p}(\Omega)}.
\end{align*}
\end{proof}

\subsection{Special case $p=2$}
The cactus case $p = 2$ is simpler to analyzed due to the availability of the Fourier transform.
Recall that the Fourier transform extends uniquely as a unitary operator on $L^2(\R^d)$, defined for
$u \in L^2(\R^d) \cap L^1(\R^d)$ by
\begin{align*}
\hat{u}(\xi):= \frac{1}{(2\pi)^{d/2}}\int_{\R^d} e^{-i\xi\cdot x} u(x)\d x\qquad (\xi\in \R^d).
\end{align*}
It is well-known that by Plancherel's theorem, we have
\begin{align}\label{eq:xxplancherel}
[u]^{*\, 2}_{H^s(\R^d)}:=\frac{ C_{d,s}}{2}\iint_{\R^d \times \R^d}\frac{|u(x+h)-u(x)|^2}{|h |^{d+2s}}\d h= \int_{\R^d}|\hat{u}(\xi)|^2|\xi|^{2s}\d \xi,
\end{align}
where $C_{d,s}$ is the usual normalization constant of the fractional Laplacian
\begin{align*}
C_{d,s}
= \widetilde{C}_{d,2,s}
= \frac{s (1-s) 2^{2s} \Gamma\left( \frac{d + 2s}{2} \right)}{\pi^{d/2} \Gamma(2 - s)}=\Big(\int_{\R^d}\frac{1-\cos(h_1)}{|h|^{d+2s}}\d h\Big)^{-1},
\end{align*}
see \cite[Section 2.3]{guy-thesis} or \cite[Chapter 1]{BuVa16}.
We need the following lemma.
\begin{lemma}
\label{lem:bounds-renorm-fract-cst} Recalling  $|\mathbb{S}^{d-1}|=\frac{2\pi^{d/2}}{\Gamma(d/2)}$, $d\geq1$, there holds that
\begin{align*}
\lim_{s \to 0^+} \frac{C_{d,s}}{s(1 - s)} =
\frac{2}{|\mathbb{S}^{d-1}|}\quad\text{and}\quad
\lim_{s \to 1^-} \frac{C_{d,s}}{s(1 - s)} = \frac{4d}{|\mathbb{S}^{d-1}|}.
\end{align*}
For $d\geq 2$, the map $s\mapsto \frac{C_{d,s}}{s(1-s)} $ is strictly increasing on $(0,1)$. In particular,
\begin{align*}
 \frac{2}{|\mathbb{S}^{d-1}|} < \frac{C_{d,s}}{s(1-s)}  < \frac{4d}{|\mathbb{S}^{d-1}|}\qquad \quad  \text{ for } \,\,  s\in(0,1).
\end{align*}
For $d=1$  the map $s\mapsto \frac{C_{1,s}}{s(1-s)} $ is strictly increasing on $(\tfrac{1}{16},1)$.  Moreover we have
\begin{align*}
\frac{25}{26}
&<\frac{2^{1/8}\Gamma\Big(\frac{9}{16}\Big)}{\sqrt{\pi}} \leq \frac{C_{d,s}}{s(1-s)} \leq2\quad \qquad \text{ for\,\,  $s\in (0,1)$}.
 \end{align*}
\end{lemma}

\begin{proof}
Let us recall that the digamma  function given by $\psi(x)=\frac{\Gamma'(x)}{\Gamma(x)}= \frac{d}{d x}\log \Gamma (x)$ is strictly increasing on $(0,\infty)$ where it derivative $\psi'(x)=\sum_{n=0}^\infty \frac{1}{(x+n)^2}>0$ is also called the trigamma function. Moreover $\psi$ has only one positive root $x_0 \approx 1.461632 \in (1,\frac{3}{2})$. Hence the function  $x\mapsto \Gamma(x)$ is strictly increasing on $(x_0,\infty)$ and strictly decreasing on $(0, x_0)$, in particular $\Gamma(x_0)=\min_{x\in (0,\infty)}\Gamma(x)$. Furthermore, let $\gamma \approx 0.5772$ be  Euler-Mascheroni constant then some special values  of $\psi$ are given by
\begin{align*}
 \psi(1)&= -\gamma \approx -0.5772,\qquad && \psi(1/2)= -\gamma -2\log2 \approx -1.9635, \\
\psi(0^+)&= \lim_{x\to 0^+} \psi(x) = -\infty,&& \psi(x+1)=\psi(x)+\frac{1}{x},\quad x>0.
\end{align*}
Define the function $G: (0,1)\to (0,\infty)$ with
\begin{align*}
G(s) :=\frac{s(1-s)}{C_{d,s}}=\frac{\pi^{d/2} \, \Gamma(2 - s)}{2^{2s} \, \Gamma\left( \frac{d + 2s}{2} \right)}.
\end{align*}
For  $d\geq 2$ or $\frac12\leq s<1$, since $\psi$ is increasing  we have $\psi(\frac{d+2s}{2})\geq \psi(1)$  and $\psi(2-s)\geq \psi(1)$. Thus, the logarithmic derivative of $G$ satisfies
\begin{align*}
\frac{d}{ds}\log G(s)= \frac{G'(s)}{G(s)} &= -\psi(2-s) - 2 \log 2 - \psi\Big(\frac{d+2s}{2}\Big) \\
&\leq
 -2\psi(1) - 2 \log 2 =2\gamma- 2 \log 2<0 .
\end{align*}
Therefore, for $d\geq 2,$ $G$ is strictly decreasing on $(0,1)$ yielding that
\begin{align*}
\frac{\pi^{d/2}}{2d\,\Gamma(d/2)}
<G(s)<
\frac{\pi^{d/2}}{\Gamma(d/2)},
\qquad
\text{for } s \in (0,1), \; d \geq 2.
\end{align*}
For $d=1$  and $\frac{1}{2}\leq s\leq 1$ we have  $\frac{G'(s)}{G(s)}<0$. Next, since the trigamma function $\psi'(x)= \frac{d^2}{dx^2}\log\Gamma(x)$ is decreasing, the function $s\mapsto H(s)$ with $H(s):=\big(\tfrac{G'(s)}{G(s)} \big)'= \psi'(2-s)  -\psi'(\frac{1+2s}{2})$ is increasing. Thus, using $\psi'(1)= \frac{\pi^2}{6}$,  $\psi'(\frac{1}{2})= \frac{\pi^2}{2}$ and the formula $\psi'(x+1)=\psi'(x)-\frac{1}{x^2}$, we have
\begin{align*}
H(s)\leq H(\tfrac12)=\psi'(\frac{3}{2})- \psi'(1)= \frac{\pi^2}{3} - 4<0&& \text{for $0<s<\frac{1}{2}$}.
\end{align*}
This implies that $s\mapsto \frac{G'(s)}{G(s)}$ is decreasing on $(0,\frac{1}{2})$, in particular
\begin{align*}
 \frac{G'(s)}{G(s)}\leq  \frac{G'(1/16)}{G(1/16)}<0,\qquad \frac{1}{16}<s<\frac{1}{2}.
\end{align*}
Thus we have  $\frac{G'(s)}{G(s)}<0$ for  $\frac{1}{16}<s<1$. Altogether we find that $G$ is strictly decreasing on $(\frac{1}{16},1)$. Recalling $G(s)= \frac{\sqrt{\pi}\,\Gamma(2-s)}{2^{2s}\,\Gamma\!\left(\frac{1+2s}{2}\right)}$ and $\Gamma\left(\tfrac{31}{16}\right)= \tfrac{15}{16}\Gamma\left(\tfrac{15}{16}\right)\approx 0.97<1$  we get
\begin{align*}
\frac{1}{2}=G(1)<G(s)\leq G(\frac{1}{16})= \frac{\sqrt{\pi}}
{2^{1/8}\,\Gamma\big(\frac{9}{16}\big)} \Gamma\big(\frac{31}{16}\big) <\frac{\sqrt{\pi}}
{2^{1/8}\,\Gamma\big(\frac{9}{16}\big)}, \quad\text{$\frac{1}{16}\leq s<1$}.
\end{align*}
 It remains  the case $d=1$ and $0<s<\tfrac{1}{16}$
we have $2-s \in (\tfrac{31}{16},2) \subset (x_0,\infty)$ and $\tfrac{1+2s}{2} \in (\tfrac12,\tfrac{9}{16}) \subset (0,x_0).$
Since the digamma function is strictly increasing we find that
\begin{align*}
\frac{d}{d s} \log \Big( \frac{1}{2^{2s}\Gamma\big(\frac{1+2s}{2}\big)} \Big)
= -2\log 2 -\psi\big(\tfrac{1+2s}{2}\big)
\geq -2\log 2 -\psi\big(\frac{9}{16}\big)\approx 0.26>0.
\end{align*}
 Whence,  $s \mapsto \frac{1}{2^{2s}\,\Gamma\!\left(\frac{1+2s}{2}\right)}$
is strictly increasing on $ (0,\tfrac{1}{16}).$ On the other hand, since $\Gamma$ is  increasing on $(x_0,\infty)$ we deduce that
\begin{align*}
\Gamma\Big(\frac{31}{16}\Big) &\leq\Gamma(2-s) \leq 1\quad \text{and}\quad
1 \leq \frac{\sqrt{\pi}}{2^{2s}\Gamma\Big(\frac{1+2s}{2}\Big)} \leq \frac{\sqrt{\pi}}{2^{1/8}\Gamma\Big(\frac{9}{16}\Big)}.
\end{align*}
It follows that, for $0<s<\frac{1}{16}$,
\begin{align*}
\frac{24}{25}<0.97\approx  \Gamma\Big(\frac{31}{16}\Big)
\leq G(s)
= \frac{\sqrt{\pi}\,\Gamma(2-s)}{2^{2s}\,\Gamma\big(\tfrac{1+2s}{2}\big)}
\leq \frac{\sqrt{\pi}}{ 2^{1/8}\Gamma\big(\tfrac{9}{16}\big)}< 1.04 =\frac{26}{25}.
\end{align*}
Hence we have
\begin{align*}
\frac{1}{2}<G(s)\leq  \frac{\sqrt{\pi}}
{2^{1/8}\,\Gamma\left(\frac{9}{16}\right)}<\frac{26}{25} \qquad\text{for \,\, $0<s<1$}.
\end{align*}
\end{proof}
\begin{theorem}
\label{thm:robust-interpo-L2}
Let  $u\in L^2(\R^d)$  and $0\leq \sigma<\eta<s\leq 1$  say $\eta =  \theta s+(1-\theta) \sigma$ with $\theta\in (0,1)$. Then we can choose $\Upsilon^*_{2,p}(\sigma, \eta, s)=1$. In particular we have
\begin{align*}
[u]^*_{H^\eta (\R^d)} &\leq[u]^{*\,\theta}_{H^s (\R^d)} \cdot  [u]^{*\,(1-\theta)}_{H^\sigma(\R^d)},
\\
[u]_{H^\eta (\R^d)}
&\leq \widetilde{N}_2(\sigma, \eta, s) \, [u]_{H^s (\R^d)}^{\theta}  \cdot [u]_{H^\sigma(\R^d)}^{1-\theta}
\end{align*}
where $\widetilde{N}_2(\sigma, \eta, s)=\big(\tfrac{B_{d,\sigma}^{\theta}B_{d,s}^{1-\theta}}{B_{d,\eta}} \big)^{1/2}\leq \sqrt{2d}$
and we define $s\mapsto B_{d,2,s}$ by
\begin{align*}
B_{d,2,s}= \frac{C_{d,s}}{2s(1-s)} =\frac{2^{2s-1} \Gamma\left( \frac{d + 2s}{2} \right)}{\pi^{d/2} \Gamma(2 - s)},
\quad
B_{d,2,0}=\frac{1}{|\mathbb{S}^{d-1}|}\quad\text{and}\quad B_{d,2,1}=\frac{2d}{|\mathbb{S}^{d-1}|}.
\end{align*}
\end{theorem}

\begin{proof}
Since  $\eta= \theta\sigma +(1-\theta) s$ then H\"older inequality implies
\begin{align*}
\int_{\R^d}|\hat{u}(\xi)|^2
|\xi|^{2\theta \sigma} |\xi|^{2(1-\theta)s}\d \xi
&\leq
\Big(  \int_{\R^d}|\hat{u}(\xi)|^2
|\xi|^{2\sigma}\d \xi\Big)^{\theta}\Big(  \int_{\R^d}|\hat{u}(\xi)|^2
|\xi|^{2s}\d \xi\Big)^{1-\theta}.
\end{align*}
In other words,  by the identity \eqref{eq:xxplancherel} we get
$[u]^*_{H^\eta (\R^d)}\leq[u]^{*\,\theta}_{H^s (\R^d)} [u]^{*\,(1-\theta)}_{H^\sigma(\R^d)}. $
Given that $[u]^{*\,2}_{H^{t}(\R^d)}= B_{d,2,t} [u]^2_{H^{t}(\R^d)}$, we conclude using Lemma \ref{lem:bounds-renorm-fract-cst} that
\begin{align*}
[u]_{H^\eta (\R^d)}\leq \big(\tfrac{B_{d,2,\sigma}^{\theta}B_{d,2,s}^{1-\theta}}{B_{d,2,\eta}} \big)^{1/2} [u]_{H^\sigma(\R^d)}^{\theta}\cdot [u]_{H^s (\R^d)}^{(1-\theta)} \leq  \sqrt{2d} [u]_{H^\sigma(\R^d)}^{\theta}\cdot [u]_{H^s (\R^d)}^{(1-\theta)}.
\end{align*}
\end{proof}

\subsection{One dimensional conjecture gap}\label{sec:conjecture}
From our analysis, the following question remains open: Prove that for $d=1$ there exists a constant $C_p>0$ such that
\begin{align}\label{eq:desired-ineq}
\eta \int_1^\infty \frac{U(t)}{t^{1+\eta p}}\d t
\leq
C_p\, \sigma \int_1^\infty \frac{U(t)}{t^{1+\sigma p}}\d t
\qquad\text{for }\,\, 0\leq \sigma<\eta<1,
\end{align}
for all  $u\in L^p(\R)$, where we note that $d=1$ we have
\begin{align*}
U(t)=\int_{\R}|u(x+t)-u(x)|^p+ |u(x-t)-u(x)|^p \d x=2 \int_{\R}|u(x+t)-u(x)|^p\d x.
\end{align*}
We make the following observations.
\begin{enumerate}
\item \textit{Natural guess for $C_p$.}
Theorem \ref{thm:r-split-monot-semin} establishes that $C_p = 2^p$ when $d \geq 2$, which is thus a natural guess for the constant in our conjecture for $d = 1$.
\item  \textit{Special case $p=2\, \&\, d\geq 1$.}
Thanks to the Fourier transform, the case $p=2$ for $d\geq 1$ boils down to proving the following inequality for $\xi\geq 0$:
\begin{align*}
\eta \int_1^\infty \frac{1-\cos(t\xi)}{t^{1+2\eta }}\d t\leq C_2\, \sigma \int_1^\infty \frac{1-\cos(t\xi)}{t^{1+2\sigma }}\d t\quad\text{for }\,\, 0\leq \sigma<\eta<1.
\end{align*}
We postulate that the sharper bound $C_2 = 2$ can be obtained, thereby improving upon the constant $C_2 = 4$ established in Lemma~\ref{lem:tail-short-cos}.
\item \textit{Better constant.} In general, for $d \ge 1$, the case $\sigma = 0$ suggests that the sought estimate could hold with the sharper constant $C_p = 2^{p-1}$. Indeed, by Theorem~\ref{thm:lp-first-order-diff}, we have $U(t) \leq 2^p |\mathbb{S}^{d-1}| \|u\|^p_{L^p(\R^d)} =: 2^{p-1} U(\infty).$ and hence , applying Lemma~\ref{lem:asymp-averag}, we obtain
\begin{align*}
\qquad \eta \int_1^\infty \frac{U(t)}{t^{1+\eta p}}\d t
\leq
\frac{2^{p-1} }{p} U(\infty) =2^{p-1} \Big(\lim_{\sigma\to 0^+} \sigma \int_1^\infty  \frac{U(t)}{t^{1+\sigma p}} \d t\Big).
\end{align*}
The same observation aligns with the conjecture in the case  $p=2$.

\item \textit{From $d=1$ to $d\geq 2$}.
A relevant motivation for finding a better constant $C_p$ in dimension $d=1$ is that the same constant $C_p$ also works in all dimensions $d\geq 2$ via successive applications of the slicing method from Theorem~\ref{thm:slicing-integ-on-sphere}.

\item \textit{Relaxed inequality.}
Although the inequality \eqref{eq:desired-ineq} is interesting in its own right, the interpolation inequality can also be obtained with a better constant through the following relaxed version:
\begin{align}\label{eq:relaxed-gap-ineq}
(1-\eta)^{\frac{1}{p}}T_p(\eta, u,1)\leq 2 (1-\sigma)^{\frac{1}{p}}T_p(\sigma,u, 1).
\end{align}
A closer look at our proof of Lemma \ref{lem:tail-short-cos}, for  the case $p=2$, hints that inequality \eqref{eq:relaxed-gap-ineq} could be easier to prove than \eqref{eq:desired-ineq}.

\item \textit{Another open question.}
 Although we are unable to construct a suitable counterexample, a related open question is to prove that $U$ is not always $2^p$-almost increasing when $d = 1$. Note that in Theorem~\ref{thm:r-split-monot-semin}, the estimate $T_p(\eta, u,1)\leq 2^p T_p(\sigma,u,1)$ for $d \ge 2$ relies on the fact that $U$ is $2^p$-almost increasing.
\end{enumerate}

\section{Stability regional Dirichlet boundary value problem} \label{sec:application-interpola}
In this section, let $\Omega \subset \mathbb{R}^d$ be an open, bounded Lipschitz domain, and assume that $sp>1$ with $0<s\leq 1$ and $1<p<\infty$, so that $\widetilde{C}_{d,p,s}= C_{d,p,s}$.  Our goal is to establish the convergence of weak solutions to the Dirichlet problem associated with the normalized regional fractional $p$-Laplacian $(-\Delta)_{p,\Omega}^s$.
 The form  $\cE^{s,p}_\Omega(\cdot,\cdot)$ associated to $(-\Delta)_{p,\Omega}^s$ is given, for $u,v\in W^{s,p}(\Omega)$, by
\begin{align*}
\cE^{s,p}_{\Omega}(u,v) &= \frac{C_{d,p,s}}{2} \iint_{\Omega \times \Omega}  \frac{|u(y)-u(x)|^{p-2}(u(y)-u(x)) }{|x-y|^{d+sp}} (v(y)-v(x))\d y \, \d x,\\
\cE^{1,p}_{\Omega}(u,v)&:= \int_\Omega |\nabla u(x)|^{p-2}\nabla u(x)\cdot \nabla v(x)\d x.
\end{align*}
Note that Theorem \ref{thm:xBBM-frac} implies that
$\cE^{s,p}_{\Omega}(u,v)\xrightarrow{s\to 1^-} \cE^{1,p}_{\Omega}(u,v)$ for $u,v\in W^{1,p}(\Omega)$;  see, for instance \cite{Fog26}.
Let us also recall that for $sp > 1$, the trace operator $\gamma^s_0: W^{s,p}(\Omega) \to W^{s-\frac{1}{p},p}(\partial\Omega)$ is continuous and surjective. Moreover, for $g \in W^{s-\frac{1}{p},p}(\partial\Omega)$, we consider the norm defined by
\begin{align*}
\|g\|_{W^{s-\frac{1}{p},p}(\partial\Omega)}
= \inf \big\{ \|u\|^*_{W^{s,p}(\Omega)} : \gamma^s_0 u = g \big\},
\end{align*}
where, for practical reasons, we consider renormalized norm
\begin{align*}
\|u\|^*_{W^{s,p}(\Omega)}
:= \big( \|u\|_{L^p(\Omega)}^p + \cE^{s,p}_{\Omega}(u,u) \big)^{\frac{1}{p}}.
\end{align*}
The fractional case is treated, for instance, in \cite{DMT22} and \cite[Chapter~9]{Leo23}, while the classical case $s = 1$ can be found in \cite[Chapter~III]{BF13}.

\subsection{Asymptotic compactness of Lions--Peetre type}
The following result, originally proved in \cite[Corollary~7]{BBM01}, establishes a form of asymptotic Lions--Peetre-type compactness (cf. \cite[Theorems~V.2.1 and~V.2.2]{LiPe64}). 
\begin{theorem}\label{thm:asymp-compact-frac}
Assume $\Omega \subset \R^d$ is open bounded Lipschitz. Let $s_0 \in [0,1)$ and $1\leq p<\infty$. Assume $(u_s)_{s\in[0,1)} \subset L^p(\Omega)$ is a bounded family such that
\begin{align*}
\sup_{s\in(s_0,1)}
\Big(\int_\Omega |u_s(x)|^p\d x+ \frac{C_{d,p,s}}{2}
\iint_{\Omega \times \Omega} \frac{|u_s(x)-u_s(y)|^p}{|x-y|^{d+sp}} \d y\d x\Big)<\infty.
\end{align*}
 Then there exist a subsequence $s_n\to 1^-$ and $u \in BV_p(\Omega)$ such that
\begin{itemize}
\item $u_{s_n}\to u$ in $L^p(\Omega)$ and  for all $0\leq \eta<1$ we have
\begin{align*}
\|u_s-u\|_{W^{\eta,p}(\Omega)}
\xrightarrow{s\to 1^-}0\quad \text{for all $0\leq \eta<1$}.
\end{align*}
\item  
We have $u\in BV_p(\Omega)$ (see Definition \ref{def:bvp-space}) and satisfies
\begin{align*}
|u|^p_{BV_p(\Omega)}
&\leq \liminf_{s_n\to 1^-} \frac{C_{d,p,s_n}}{2}
\iint_{\Omega \times \Omega} \frac{|u_{s_n}(x)-u_{s_n}(y)|^p}{|x-y|^{d+s_np}} \d y\d x.
\end{align*}
\item If $1<p<\infty$ then for all $0< \tau <1$ and $\tau p>1$ we have
\begin{align*}
\|\gamma^{s_n}_0(u_{s_n})-\gamma^1_0(u)\|_{L^{p}(\partial \Omega)}+ \|\gamma_0^{s_n} (u_{s_n})-\gamma^1_0(u)\|_{W^{\tau-\frac{1}{p},p}(\partial \Omega)}
\xrightarrow{s_n\to 1^-}0.
\end{align*}
\end{itemize}
\end{theorem}
\begin{proof}
The existence of subsequence $s_n\to 1^-$ such that $ \|u_{s_n}-u\|_{L^p(\Omega)}\xrightarrow{s_n\to 1^-}0$ for some $u\in BV_p(\Omega)$, which additionally satisfies the second claim,  can be found in \cite[Theorem 1.2]{Pon04}; see also \cite{BBM01}.  Next, note that by  \cite[Lemma 2.12]{Fog23} we have
\begin{align}\label{eq:xxuniform-bound}
s(1-s)
\iint_{\Omega \times \Omega} \frac{|u(x)-u(y)|^p}{|x-y|^{d+sp}} \d y\d x\leq C(d,p,\Omega)
\begin{cases}
\|u\|^p_{W^{1,p}(\Omega)}&1<p<\infty\\
\|u\|_{BV(\Omega)} &p=1.
\end{cases}
\end{align}
Thus by assumption we find that $\|u_{s_n}\|_{W^{s_n,p}(\Omega)}+\|u\|_{W^{s_n,p}(\Omega)}\leq C(d,p,\Omega)$. Let us consider $0\leq \eta < s_n <1$, for $n$ large enough, so that $u_{s_n}, u \in W^{\eta,p}(\Omega)$. The robust interpolation, Theorem \ref{thm:robust-interpo-dom}, with $\theta=\frac{\eta}{s_n}$ implies 
\begin{align*}
\|u_{s_n}-u\|_{W^{\eta,p}(\Omega)}
&\leq C \|u_{s_n}-u\|^{\frac{\eta}{s_n}}_{W^{s_n,p}(\Omega)} \|u_{s_n}-u\|^{1-\frac{\eta}{s_n}}_{L^{p}(\Omega)} \leq C \|u_{s_n}-u\|^{1-\frac{\eta}{s_n}}_{L^{p}(\Omega)} \xrightarrow{n\to\infty} 0.
\end{align*}
Note that $u_{s_n}, u \in W^{\tau,p}(\Omega)$ for $\frac{1}{p} < \tau \le s_n < 1$. By the continuity of the trace operator and the previous claims, we have
\begin{align*}
\|\gamma^{s_n}_0(u_{s_n}) - \gamma^1_0(u)\|_{L^p(\partial \Omega)}
&+ \|\gamma^{s_n}_0(u_{s_n}) - \gamma^1_0(u)\|_{W^{\tau-\frac{1}{p},p}(\partial \Omega)} \\
&\leq C_\tau \|u_{s_n} - u\|_{W^{\tau,p}(\Omega)} \xrightarrow{n \to \infty} 0,
\end{align*}
where $C_\tau = C(d,p,\Omega,\tau) > 0$ is a constant. This proves the third claim.
\end{proof}

\begin{theorem}[Robust regional Poincar\'e-Friedrichs]
\label{thm:robust-poinca-fried-regio}
Assume  $\Omega\subset \R^d$ is open  bounded set with Lipschitz boundary and $1< p <\infty$. Then there exist $s_0= s_0(d,p,\Omega)\in (\frac{1}{p},1 )$  and $B=B(d,p,\Omega)>1$ such that for all $u\in W^{s,p}_0(\Omega)$  and $s_0 \leq s <1$  we have
\begin{align*}
\int_\Omega |u(x)|^p\d x\leq Bs(1-s) \iint_{ \Omega \times \Omega}\frac{|u(x)-u(y)|^p}{|x-y|^{d+sp}}\d y \d x.
\end{align*}
In particular, letting  $s\to 1^-$ yields
\begin{align*}
\|u\|^p_{L^p(\Omega)}\leq C\|\nabla u\|^p_{L^p(\Omega)}.
\end{align*}
\end{theorem}
\begin{proof}
Assume there exist no such $s_0$ and $B$ exist. Taking $s_0 = 1-\frac{1}{2^n}>\frac{1}{p}$ with $n \geq 1$ large and $B = 2^n$ there exist $s_n \in (1-\frac{1}{2^n},1)$ and $u_{s_n} \in W^{s_n, p}_0(\Omega)$ such that
\begin{align*}
\|u_{s_n}\|_{L^p(\Omega)} = 1\quad \text{and}\quad s_n(1-s_n)\iint_{ \Omega \times \Omega}\frac{|u_{s_n}(x)-u_{s_n}(y)|^p}{|x-y|^{d+s_sp}}\d y \d x<\frac{1}{2^n}.
\end{align*}
In particular, $u_{s_n}\in W^{s_n, p}_0(\Omega)$, $\gamma^{s_n}_0 u_{s_n}=0$ and
\begin{align*}
\sup_{n\geq 1}\Big( \int_\Omega |u_{s_n}(x)|^p\d x+ s_n(1-s_n)\iint_{ \Omega \times \Omega}\frac{|u_{s_n}(x)-u_{s_n}(y)|^p}{|x-y|^{d+s_sp}}\d y \d x\Big)<2.
\end{align*}
By the asymptotic compactness Theorem \ref{thm:asymp-compact-frac} there is $u \in W^{1,p}(\Omega)$ and a subsequence still denoted $(u_{s_n})_n$  such that $\|u_{s_n}-u\|_{W^{\eta,p}(\Omega)}\to0$, $0\leq \eta<1$ as $n\to \infty$ and
\begin{align*}
K_{d,p} \frac{|\mathbb{S}^{d-1}|}{p}\|\nabla u\|_{L^p(\Omega)}^p \leq \liminf_{n\to\infty} s_n(1-s_n)\iint_{ \Omega \times \Omega}\frac{|u_{s_n}(x)-u_{s_n}(y)|^p}{|x-y|^{d+s_sp}}\d y \d x= 0.
\end{align*}
In particular $\|u\|_{L^p(\Omega)}=1$ and $\nabla u$=0 a.e on $\Omega$.  For fixed $\eta\in (\frac{1}{p}, 1)$ and  $N\geq 1$ sufficiently large such that $\eta<s_n<1$ for all $n\geq N$ we have $W^{s_n,p}(\Omega)\subset W^{\eta,p}(\Omega)$. This implies $\gamma^{s_n}_0 v= \gamma^\eta_0v $ for all $v\in W^{s_n,p}(\Omega)\cap C^1(\overline{\Omega})$. By density argument, it is not difficult to conclude that
$ \gamma^\eta_0\mid_{W^{s_n,p}(\Omega)}= \gamma^{s_n}_0$.
Since, $u_{s_n}\in  W^{s_n,p}_0(\Omega)$, it follows that
\begin{align*}
\gamma^\eta_0 (u_{s_n})= \gamma^{s_n}_0 (u_{s_n})=0\qquad\text{for all $n\geq N$}.
\end{align*}
Therefore, by definition of the trace norm we have
\begin{align*}
&\|\gamma^1_0(u)\|_{W^{\eta-\frac{1}{p},p}(\partial\Omega)}
=\|\gamma^{s_n}_0(u_{s_n}) -\gamma^1_0(u)\|_{W^{\eta-\frac{1}{p},p}(\partial\Omega)}\leq \|u_{s_n}-u\|_{W^{\eta,p}(\Omega)} \xrightarrow{s_n\to 1^-}0.
\end{align*}
It follows that  $\gamma^1_0 u=0$. Since  $\partial\Omega$ is Lipschitz, we deduce that  $u \in W^{1,p}_0(\Omega)$. In conclusion,  $u \in W^{1,p}_0(\Omega)$, $\nabla u = 0$ a.e. in $\Omega$ and $\|u\|_{L^p(\Omega)}=1$. Consequently,  the zero extension $\widetilde{u}$ of $u$ belongs to  $ W^{1,p}(\R^d)$ and satisfies $\nabla \widetilde{u}=  \widetilde{\nabla u}= 0$ a.e. in $\R^d$. Hence $\widetilde{u}$ is constant on $\R^d$ and we have $\widetilde{u}\equiv0$ on $\R^d$, since $\widetilde{u}= 0$ on $\R^d \setminus \Omega$. This clearly implies that  $u \equiv 0$ on $\Omega$. Thereby contradicting the fact that $\|u\|_{L^p(\Omega)}=1$  and thus our initial assumption was wrong.
\end{proof}

\subsection{Dirichlet  problem for regional fractional $p$-Laplacian}\label{sec:regional-frac-plaplace}
\begin{proof}[Proof of Theorem \ref{thm:converg-dirichlet-frac-regional}]
Throughout this proof, we assume $sp>1$, let $C>0$ denote a generic constant independent of $s$, while $s_* \in (\tfrac{1}{p}, 1)$ is sufficiently close to $1$.
Given that  $(f_s)_s$ in $L^{p'}(\Omega)$  converges  weakly as $s\to 1^-$ we may assume thanks to the uniform boundedness principle that $(\|f_s\|_{L^{p'}(\Omega)})_{s\in (s_*,1)}$ is bounded; namely, $\sup_{s_*<s<1} \| f_s\|_{L^{p'}(\Omega)} < \infty.$ Now let $\overline{g}_1, h_s \in W^{1,p}(\Omega)$ such that
$\gamma^1_0(\overline{g}_1)= g_1$, $\gamma^1_0 (h_s)= g_s-g_1$ and
\begin{align*}
\|h_s\|_{W^{1,p}(\Omega)}\leq 1-s+ \|g_s-g_1\|_{W^{1-\frac{1}{p},p}(\partial \Omega)}.
\end{align*}
The function $ \overline{g}_s= h_s +\overline{g}_1\in W^{1,p}(\Omega)$ satisfies  $\gamma^1_0(\overline{g}_s) = \gamma^1_0(h_s+\overline{g}_1) = g_s$  and
\begin{align}\label{eq:gs-frac-extens}
\|\overline{g}_s-\overline{g}\|_{W^{1,p}(\Omega)}\leq 1-s+ \|g_s-g\|_{W^{1-\frac{1}{p},p}(\partial \Omega)}\xrightarrow{s\to1^-}0.
\end{align}
Hence, replacing $g_s$ by $\overline{g}_s$ if necessary, we may assume without loss of generality that $g_s\in W^{1,p}(\Omega)$ such that $\|g_s-g_1 \|_{W^{1,p}(\Omega)}\xrightarrow{s\to1^-}0.$ In particular, taking into account Theorem \ref{thm:xBBM-frac}  we get
\begin{align}\label{eq:xgs-conv-uni}
 \|g_s-g_1 \|_{W^{s,p}(\Omega)}\leq C\|g_s-g_1 \|_{W^{1,p}(\Omega)}\xrightarrow{s\to1^-}0.
\end{align}
Thus, up to relabeling $s_*$, we can assume $\|g_s-g \|_{W^{s,p}(\Omega)}<1$ for $s\in (s_*, 1)$, so that taking into account Theorem \ref{thm:xBBM-frac}  we get
\begin{align*}
\sup_{s_*<s<1}\|g_s \|_{W^{s,p}(\Omega)}\leq 1+ \sup_{s_*<s<1}
\|g \|_{W^{s,p}(\Omega)}\leq 1+  C \|g\|^p_{W^{1,p}(\Omega)}<\infty.
\end{align*}
Accordingly, for $s_*$ close enough to $1$,  we obtain the uniform estimate
\begin{align}\label{eq:uniform-bdd-f-g-D-regio}
M:= 	\sup_{s_*<s<1}\big(\|f_s \|_{L^{p'}\Omega)}+\|g_s \|_{W^{s,p}(\Omega)}\big)<\infty.
\end{align}
By the robust Poincar{\'e}-Friedrichs inequality, Theorem \ref{thm:robust-poinca-fried-regio}, we can assume that $s_* \in (\frac{1}{p},1)$ and $C>0$ are such that, for all $s\in ( s_*,1)$ and $v \in W^{s,p}_0(\Omega)$ we have
\begin{align}\label{eq:uniform-coercivity-D-regio}
\|v\|^p_{W^{s,p}(\Omega)}\leq C\cE^{s,p}_\Omega(v,v).
\end{align}
Each $u_s\in W^{s,p}(\Omega)$, $s_*<s\leq 1$, uniquely  satisfies the minimization problem
\begin{align*}
\cJ^{s,p}_0(u_s) &= \min_{v-g_s\in W^{s,p}_0(\Omega)} \cJ^{s,p}_0(v),\\
\cJ^{s,p}_0(v) &= \tfrac{1}{p} \cE^{s,p}_\Omega(v,v)
-\int_\Omega f_s(v-g_s)\d x.
\end{align*}
In particular,  upon deriving the Euler Lagrange equation $u_s$ verifies
\begin{align*}
\cE^{s,p}_\Omega(u_s,u_s-g_s)
=  \int_\Omega f_s ( u_s-g_s)\d x \quad \text{and}\quad u_s-g_s\in W^{s,p}_0(\Omega).
\end{align*}
Since  $u_s-g_s\in W^{s,p}_0(\Omega),$ the estimate \eqref{eq:uniform-coercivity-D-regio} implies that
\begin{align*}
\begin{split}
\| u_s-g_s \|_{W^{s,p}(\Omega)}
\leq C\cE^{s,p}_\Omega(u_s,u_s)^{\frac{1}{p}}+ C\cE^{s,p}_\Omega(g_s,g_s)^{\frac{1}{p}}.
\end{split}
\end{align*}
Using this estimate  we also find that
\begin{align*}
\begin{split}
&\cE^{s,p}_\Omega(u_s,u_s)
=\int_\Omega f_s ( u_s-g_s)\d x +\cE^{s,p}_\Omega(u_s,g_s)\\
&\leq \| u_s-g_s \|_{W^{s,p}(\Omega)} \|f_s\|_{L^{p'}\Omega)}+\cE^{s,p}_\Omega(u_s,u_s)^{\frac{1}{p'}}\cE^{s,p}_\Omega(g_s,g_s)^{\frac{1}{p}}\\
&\leq  C \|f_s\|_{L^{p'}\Omega)}\big(
\cE^{s,p}_\Omega(u_s,u_s)^{\frac{1}{p}}+\cE^{s,p}_\Omega(g_s,g_s)^{\frac{1}{p}}\big)
+\cE^{s,p}_\Omega(u_s,u_s)^{\frac{1}{p'}}\cE^{s,p}_\Omega(g_s,g_s)^{\frac{1}{p}}.
\end{split}
\end{align*}
In view of the  Young inequality $|ab|\leq \frac{\delta^p |a|^p}{p}+ \frac{|b|^{p'}}{p'\delta^{p'}} $, $a,b\in \R$ and $\delta>0$,  we get

\begin{align*}
\cE^{s,p}_\Omega(u_s,u_s)^{\frac{1}{p'}}\cE^{s,p}_\Omega(g_s,g_s)^{\frac{1}{p}}
&\leq  \frac{\delta^{p}\cE^{s,p}_\Omega(u_s,u_s)}{p'}+ \frac{\cE^{s,p}_\Omega(g_s,g_s)}{p\delta^{\frac{p^2}{p'}}},\\
C\|f_s\|_{L^{p'}\Omega)} \cE^{s,p}_\Omega(u_s,u_s)^{\frac{1}{p}}
&\leq  \frac{\delta^p\cE^{s,p}_\Omega(u_s,u_s)}{p}+ \frac{C^{p'}\|f_s\|_{L^{p'}\Omega)} ^{p'}}{p'\delta^{p'}},\\
C\|f_s\|_{L^{p'}\Omega)}\cE^{s,p}_\Omega(g_s,g_s)^{\frac{1}{p}}
&\leq \frac{\delta^{p}\cE^{s,p}_\Omega(g_s,g_s)}{p} + \frac{C^{p'}\|f_s\|_{L^{p'}\Omega)} ^{p'}}{p'\delta^{p'}}.
\end{align*}
Inserting altogether in the previous estimate gives
\begin{align*}
\cE^{s,p}_\Omega(u_s,u_s)
&\leq 	 \delta^{p}\cE^{s,p}_\Omega(u_s,u_s) + ( \frac{1}{p\delta^{\frac{p^2}{p'}}}+ \frac{\delta^{p}}{p}) \cE^{s,p}_\Omega(g_s,g_s)
+ \frac{2C^{p'}}{p'\delta^{p'}} \|f_s\|_{L^{p'}\Omega)} ^{p'}.
\end{align*}
\noindent Taking in particular  $\delta^p=\frac12$ we arrive at
\begin{align*}
\cE^{s,p}_\Omega(u_s,u_s)&\leq C(\|f_s\|_{L^{p'}\Omega)} ^{p'} + \cE^{s,p}_\Omega(g_s,g_s)) .
\end{align*}
On the other hand, the coercivity estimate \eqref{eq:uniform-coercivity-D-regio} implies
\begin{align}\label{eq:xlp-bound-dirich}
\begin{split}
\|u_s\|_{L^p(\Omega)}
&\leq  \|g_s\|_{L^p(\Omega)}+ C\cE^{s,p}_\Omega(u_s-g_s,u_s-g_s)^{\frac{1}{p}}\\
&\leq C\|g_s\|_{W^{s,p}(\Omega) }+ C\cE^{s,p}_\Omega(u_s,u_s)^{\frac{1}{p}}.
\end{split}
\end{align}
By combining everything together and accounting \eqref{eq:uniform-bdd-f-g-D-regio}  leads to the uniform estimate
\begin{align}\label{eq:unifo-bdedness-D}
\|u_s\|_{W^{s,p}(\Omega)}
&\leq C\big(\|f_s \|^{\frac{1}{p-1}}_{L^{p'}\Omega)}+\|g_s \|_{W^{s,p}(\Omega)}\big)
\leq C \qquad s\in (s_*,1)
\end{align}
with  a generic  constant  $C>0$ independent of $s$. Accordingly, by the asymptotic compactness
Theorem \ref{thm:asymp-compact-frac}, there is $u\in W^{1,p}(\Omega)$ and subsequence $s_n\to 1^-$ such that
\begin{align*}
&\|u_{s_n} -u\|_{W^{\eta,p}(\Omega)}\xrightarrow{s_n\to 1^-}0
\qquad\text{for all $0\leq \eta<1$}, \\
&\|\gamma^{s_n}_0(u_{s_n}) -\gamma^1_0(u)\|_{W^{\tau-\frac{1}{p},p}(\partial\Omega)}\xrightarrow{s_n\to 1^-}0\qquad\text{for all $\frac1{p} <\tau<1$},\\
&\cE^{1,p}_\Omega(u,u) \leq \liminf_{n\to\infty}\cE^{s_n,p}_\Omega(u_{s_n}, u_{s_n}).
\end{align*}
Since $\gamma^{s_n}_0 (u_{s_n})= \gamma^{s_n}_0(g_{s_n})$,  $ \gamma^{\tau}_0|_{W^{s_n,p}(\Omega)}=\gamma^{s_n}_0$ and $\gamma^{s_n}_0|_{W^{1,p}(\Omega)}= \gamma^{1}_0$  it follows that
\begin{align*}
\|\gamma^1_0(u)-\gamma^1_0(g_1)\|_{L^{p}(\partial\Omega)}
&\leq \|\gamma^{s_n}_0(u_{s_n}) -\gamma^1_0(u)\|_{L^{p}(\partial\Omega)}
+ \|\gamma^{s_n}_0(g_{s_n}) -\gamma^1_0( g_1)\|_{L^{p}(\partial\Omega)}\\
&\leq C_\tau \|\gamma^{s_n}_0(u_{s_n}) -\gamma^1_0(u)\|_{W^{\tau-\frac{1}{p},p}(\partial\Omega)}
+ C\|g_{s_n}- g_1\|_{W^{1,p}(\Omega)}
\to0,
\end{align*}
as $s_n\to 1^-$. This implies $\gamma^1_0(u)= \gamma^1_0(g_1)$, i.e., $u-g_1\in W^{1,p}_0(\Omega)$.
 From the fact that  $u_{s_n}-g_{s_n}\to u-g_{1}$ in $L^p(\Omega)$ and $f_{s_n}\rightharpoonup f_1 $ we find that
\begin{align*}
\liminf_{n\to\infty}\cJ_0^{s_n,p}(u_{s_n})= \liminf_{n\to\infty}\big(\tfrac{1}{p} \cE_\Omega^{s_n,p}(u_{s_n},u_{s_n} )-\int_\Omega f_{s_n}(u_{s_n}-g_{s_n})\d x \big) \geq  \cJ^{1,p}_0(u).
\end{align*}
Now, we want show that $u=u_1$.  Let $v\in g_1+W^{1,p}_0(\Omega)\subset g_1+W^{s,p}_0(\Omega)$. Define $v_s=v-g_1+g_s$ so that $v_s\in g_s+ W^{s,p}_0(\Omega)$. Using $v_s-v=-g_1+g_s$, we find that
\begin{align*}
\big|\cE^{s,p}_\Omega(v_s,v_s)^{\frac{1}{p}}-\cE^{1,p}_\Omega(v,v)^{\frac{1}{p}}\big|
&\leq \cE^{s,p}_\Omega(v_s-v,v_s-v)^{\frac{1}{p}}+ \big|\cE^{s,p}_\Omega(v,v)^{\frac{1}{p}}-\cE^{1,p}_\Omega(v,v)^{\frac{1}{p}}\big|\\
&\leq  \|g_s-g\|_{W^{s,p}(\Omega)}+ \big|\cE^{s,p}_\Omega(v,v)^{\frac{1}{p}}-\cE^{1,p}_\Omega(v,v)^{\frac{1}{p}}\big|.
\end{align*}
Given that $\Omega$ is Lipschitz, we have  $ \cE^{s,p}_\Omega(v,v)\xrightarrow{s\to1^-}\cE^{1,p}_\Omega(v,v)$  and from \eqref{eq:xgs-conv-uni} we have $\|g_s-g_1\|_{W^{s,p}(\Omega)}\xrightarrow{s\to1^-}0$, it follows that
$\cE^{s,p}_\Omega(v_s,v_s)
\xrightarrow{s\to 1^-} \cE^{1,p}_\Omega(v,v).$
Noting that $v_s-g_s= v-g_1\in W^{1,p}_0(\Omega)$, this and  $f_{s}\rightharpoonup f_1 $ yields
\begin{align*}
\lim_{s\to1^-}\cJ^{s,p}_0(v_s)
&= \lim_{s\to1^-}\big(\tfrac{1}{p}  \cE^{s,p}_\Omega(v_s, v_s)-\int_\Omega f_{s} (u_{s}-g_{s})\d x \big)\\
&=\tfrac{1}{p}\cE^{1,p}(v,v) -\int_\Omega f_1(u-g_{1})\d x= \cJ^{1,p}_0(v).
\end{align*}
Since, each $u_{s_n}$ satisfies
so that  $\cJ_0^{s_n,p}(u_{s_n})\leq
\cJ_0^{s_n,p}(v_{s_n}) $, we deduce that
\begin{align*}
\cJ^{1,p}_0(u)\leq \liminf_{n\to\infty}\cJ_0^{s_n,p}(u_{s_n})\leq
\liminf_{n\to\infty}\cJ_0^{s_n,p}(v_{s_n})= \cJ^{1,p}_0(v).
\end{align*}
In conclusion, $\|u_{s_n}-u\|_{W^{\eta,p}(\Omega)}\to 0$ as $s_n\to1^-$ for all $0\leq \eta<1$ and
\begin{align*}
\cJ^{1,p}_0(u)&= \min_{v-g_1\in W^{1,p}_0(\Omega)} \cJ^{1,p}_0(v).
\end{align*}
In other words, $u= u_1\in W^{1,p}(\Omega)$ is the unique weak solution to the local  Dirichlet. The uniqueness of $u_1$ implies  that $\|u_{s}-u_1\|_{W^{\eta,p}(\Omega)}\to 0$ as $s\to1^-$ for all $0\leq \eta<1$.
\end{proof}

\appendix
\section{Robust extension}\label{sec:robust-exten}
In this section we prove that if the boundary $\partial \Omega$ is Lipschitz and compact, then $\Omega$ is a robust extension domain. More precisely, we show that the Sobolev extension operator for $E: W^{s,p}(\Omega)\to W^{s,p}(\R^d)$ is bounded, with an operator norm that is bounded by a constant  independent of $s$. For the convenience of the reader, we follow the approach presented in \cite{Hitchhiker}. For this section we merely consider the norm
\begin{align*}
\|u\|_{W^{s,p}(\Omega)}:=  \big(\|u\|^p_{L^p(\Omega)}+ [u]^p_{W^{s,p}(\Omega)} \big)^{1/p}\qquad s\in [0,1].
\end{align*}
In general, an arbitrary open set is not necessarily a $W^{s,p}$-extension domain.
The following result shows that the half–space $\R^d_+$ is a $W^{s,p}$-robust extension domain.
\begin{lemma}[Even reflection] \label{lem:even-reflec}
Assume that $\Omega \subset \mathbb{R}^d$ is an open set which is symmetric with respect to the hyperplane $\{x_d = 0\}$, i.e., for $(x', x_d)\in \R^{d-1}\times \R$  we have  $(x', x_d)\in \Omega$ if and only if $(x',-x_d)\in \Omega$.
In particular, $\Omega = \R^d$ is allowed. So that  $\Omega=\Omega_+\cup\Omega_-$ with  $\Omega_+=\Omega \cap \R^d_+$ and $\Omega_-= \Omega \cap (\R^d\setminus \R^d_+)$ where $\R^d_+=\{x\in \R^d\,:\, x_d>0\}$  is the upper half space.   Define refection of $u:\Omega_+\to\R$ by
\begin{align*}
\overline{u}(x)=u(x',|x_d|)=
\begin{cases}
u(x',x_d)&x_d\geq 0,\\
u(x',-x_d) &x_d<0.
\end{cases}
\end{align*}
Then the linear map  $\overline{\cdot}:W^{s,p}(\Omega_+)\to W^{s,p}(\Omega)$, $u\mapsto \overline{u}$  is continuous and satisfies
\begin{align*}
[\overline{u}]_{W^{s,p}(\Omega)} \leq
4^{1/p}[u]_{W^{s,p}(\Omega_+)}\qquad \text{for all $s\in [0,1]$,\,\,\, $1\leq p\leq \infty$}.
\end{align*}
In particular, $\R^d_+$ is a $W^{s,p}$-robust extension domain.
\end{lemma}

\begin{proof}
The case $s=1$ albeit standard, is somewhat technical and can be adapted for instance from \cite[Lemm 9.2]{Bre10} so  that $\partial_{x_i} \overline{u} = \overline{\partial_{x_i} u}$ for $i = 1,\dots,d-1$ and $\partial_{x_d} \overline{u} = -\overline{\partial_{x_d} u}$. For $1\leq p\leq \infty$, the  change of variables $z=(z',z_d)= (x',-x_d)$ yields
\begin{align*}
\|\overline{u}\|_{L^p(\Omega)}
= 2^{1/p} \|u\|_{L^p(\Omega_+)}\quad \text{and}\quad
\|\nabla \overline{u}\|_{W^{1,p}(\Omega)}
= 2^{1/p}\|\nabla u\|_{W^{1,p}(\Omega_+)}.
\end{align*}
Now we turn our attention to the case $0\leq s<1$. For $x\in \R^d_+$ and $y\in\R^d_-$ we have  $(x_d-y_d)^2\geq (x_d+y_d)^2 $ and therefore
\begin{align*}
\int_{\Omega}\int_{\Omega} \frac{|\overline{u}(x)-\overline{u}(y)|^p}{|x-y|^{d+sp}} &\d y\d x
=  \int_{ \Omega_-}\int_{ \Omega_-} \frac{|u(x',-x_d)-u(y',-y_d)|^p}{|x-y|^{d+sp}} \d y\d x\\
&+\int_{\Omega_+}\int_{\Omega_+} \frac{|u(x)-u(y)|^p}{|x-y|^{d+sp}} \d y \d x+ 2\int_{\Omega_+}\int_{ \Omega_-} \frac{|u(x)-u(y',-y_d)|^p}{|x-y|^{d+sp}} \d y\d x \\
& \leq  4\int_{\Omega_+}\int_{\Omega_+} \frac{|u(x)-u(y)|^p}{|x-y|^{d+sp}} \d y \d x.
\end{align*}
By the same token,  it readily follows that
\begin{align*}
[\overline{u}] _{W^{s,\infty}(\Omega)} \leq  [u]_{W^{s,\infty}(\Omega_+)}\qquad s\in (0,1).
\end{align*}
 Therefore, $\R^d_+$ is a $W^{s,p}$-extension domain since  if $\Omega_+=\R^d_+$ then $\Omega=\R^d$.
\end{proof}

To construct an extension of a function $u$ defined on $\Omega$ to the whole space $\R^d$, we first establish the following preliminary lemmas.
\begin{lemma}[Cutoff function]\label{lem:cutoff}
Let $\Omega\subset \R^d$ be open, $s\in[0,1]$ and
$1\leq p\leq \infty$.  Let $u\in L^p(\Omega)$ and $\varphi \in L^\infty(\Omega)$.  The following assertions are true.
\begin{itemize}
\item If $u, \varphi \in W^{s,p}(\Omega) \cap L^\infty(\Omega)$ then we have  $u\varphi \in W^{s,p}(\Omega)$ and
\begin{align*}
[u\varphi ]_{W^{s,p}(\Omega)}
\leq
2^{1-1/p}\big(\|\varphi\|_{L^\infty(\Omega)} [u]_{W^{s,p}(\Omega)}
+ \|u\|_{L^\infty(\Omega)}  [u]_{W^{s,p}(\Omega)}\big).
\end{align*}
\item If $u\in W^{s,p}(\Omega)$ and $\varphi\in C^{0,1}(\Omega)$, then we have  $u\varphi \in W^{s,p}(\Omega)$ and
\begin{align*}
\|u\varphi \|_{W^{s,p}(\Omega)}\leq  2^{1-1/p}(1+|\mathbb{S}^{d-1}|)^{1/p} \|\varphi\|_{W^{1,\infty}(\Omega)}  \|u\|_{W^{s,p}(\Omega)}.
\end{align*}
 \item If $u\in W^{s,p}(\Omega)$ and $\varphi\in C^{0,1}_c(\Omega)$, then we have  $u\varphi \in W^{s,p}_0(\Omega)$.
\end{itemize}
\end{lemma}

\begin{proof}
We only prove for $0\leq s <1$  as the case $s=1$ is easily yielded by the same token. Clearly   for $1\leq p< \infty$  we  have
\begin{align*}
|\varphi(x)u(x)-\varphi(y)u(y)|^p
&\leq 2^{p-1}\Big( \|\varphi\|^p_{L^\infty(\Omega)}|u(x)-u(y)|^p+ |u(x)|^p|\varphi(x)-\varphi(y)|^p \Big)
\end{align*}
for all $x,y\in \Omega$. This immediately implies
\begin{align*}
[u\varphi ]^p_{W^{s,p}(\Omega)}
\leq
2^{p-1}\big(\|\varphi\|^p_{L^\infty(\Omega)} [u]^p_{W^{s,p}(\Omega)}
+ \|u\|^p_{L^\infty(\Omega)}  [u]^p_{W^{s,p}(\Omega)}\big).
\end{align*}
Now if  $\varphi\in C^{0,1}(\Omega)$ then  $ |\varphi(x)-\varphi(y)|\leq \|\varphi\|_{W^{1,\infty}(\Omega)}(1\land |x-y|^p)$, so that
\begin{align*}
|\varphi(x)u(x)-\varphi(y)u(y)|^p
&\leq 2^{p-1} \|\varphi\|^p_{W^{1,\infty}(\Omega)}
\Big( |u(x)-u(y)|^p+ |u(x)|^p(1\land |x-y|^p)\Big).
\end{align*}
Moreover, we note that
\begin{align*}
 \int_{\Omega}\int_{\Omega} \frac{|u(x)|^p(1\land |x-y|^p)}{|x-y|^{d+sp}} \d y \d x
&\leq  \|u\|^p_{L^p(\Omega)} \int_{\R^d} \frac{1\land |h|^p}{|h|^{d+sp}}\d h =  \frac{|\mathbb{S}^{d-1}|}{ps(1-s)} \|u\|^p_{L^p(\Omega)}.
\end{align*}
Therefore, integrating both sides  yields
\begin{align*}
 s(1-s)&\int_{\Omega}\int_{\Omega} \frac{|\varphi(x)u(x)-\varphi(y)u(y)|^p}{|x-y|^{d+sp}} \d y \d x
\\
&\leq  2^{p-1} \|\varphi\|^p_{W^{1,\infty}(\Omega)}   s(1-s)\int_{\Omega}\int_{\Omega} \frac{|u(x)|^p(1\land |x-y|^p)}{|x-y|^{d+sp}} \d x\d y\\
& \qquad\qquad\qquad + 2^{p-1} \|\varphi\|^p_{W^{1,\infty}(\Omega)}  s(1-s)\int_{\Omega}\int_{\Omega} \frac{|u(x)-u(y)|^p}{|x-y|^{d+sp}} \d y \d x\\
& \leq  2^{p-1} \|\varphi\|^p_{W^{1,\infty}(\Omega)}
\bigg(  |\mathbb{S}^{d-1}|\|u\|^p_{L^p(\Omega)}+  s(1-s) \int_{\Omega}\int_{\Omega} \frac{|u(x)-u(y)|^p}{|x-y|^{d+sp}} \d y \d x\bigg).
\end{align*}
Moreover, we have  $\|u\varphi\|_{L^p(\Omega)}
\leq \|\varphi \|_{L^\infty(\Omega)}  \|u\|_{L^p(\Omega)}$. Finally, while the case $p=\infty$ is obvious, we deduce that
\begin{align*}
\|u\varphi \|_{W^{s,p}(\Omega)}
&\leq  2^{1-1/p}(1+|\mathbb{S}^{d-1}|)^{1/p} \|\varphi\|_{W^{1,\infty}(\Omega)}  \|u\|_{W^{s,p}(\Omega)}.
\end{align*}
 Now assume $\varphi\in C^{0,1}_c(\Omega)$ with $\supp \varphi\subset K$. Let  $\zeta_\eps(x)=\eps^{-d} \zeta(\eps^{-1}x)$, $\eps>0$ with $\zeta$ satisfying $\zeta\in C^\infty(B_1(0))$, $\int\zeta =1$ and $0\leq \zeta\leq 1$. There is  bounded set $\Omega'\subset \Omega$  such that $\delta:= \dist(\Omega', \partial \Omega)>0$ and for  $\eps>0$ sufficiently small we have $\supp (u\varphi)*\eta_\eps \subset \supp\varphi+ B_\eps(0)\subset \overline{\Omega'}\subset \Omega$ and $\Omega'+ B_\eps(0)\subset\Omega$. It follows in particular that $[u\varphi]*\zeta_\eps\in C_c^\infty(\Omega)$ and $\| [u\varphi]*\zeta_\eps-[u\varphi]\|_{ L^p(\Omega)}\xrightarrow{\eps\to 0}0$.
 Moreover, the function
$(x,y)\mapsto U_\varphi (x, y)=([u\varphi](x)-[u\varphi](y))|x-y|^{-\frac{d	}{p}-s}$ satisfies, since $\Omega'+ B_\eps(0)\subset \Omega$, $\|U_\varphi (\cdot-\eps z, \cdot-\eps z)-U_\varphi \|_{L^p(\Omega'\times\Omega')}\leq 2^p \|U_\varphi \|_{L^p(\Omega\times\Omega)}=2^p |u|^p_{W^{s,p}(\Omega)}$ and, by continuity of the shift in $L^p(\Omega'\times \Omega')$, $\|U_\varphi (\cdot-\eps z, \cdot-\eps z)-U_\varphi \|_{L^p(\Omega'\times\Omega')}\xrightarrow{\eps\to1^-}0$, for each $z\in B_1(0)$. This in conjunction with dominated convergence theorem yield
\begin{align*}
\big|[u\varphi]*\zeta_\eps- &[u\varphi]\big|^p_{W^{s,p}(\Omega)}= \int_{\Omega'}\int_{\Omega'}\frac{|([u\varphi]*\zeta_\eps- [u\varphi])(x)- ([u\varphi]*\zeta_\eps- [u\varphi])(y)|^p}{|x-y|^{d+sp}}\d y\d x\\
&+2 \int_{\Omega'}|[u\varphi]*\zeta_\eps(x)- [u\varphi](x)|^p\int_{\Omega\setminus
\Omega'}\frac{\d y}{|x-y|^{d+sp}}\d x\\
&\leq \int_{B_1(0)}\hspace{-2ex}\zeta(z)\int_{\Omega'}\int_{\Omega'}|U_\varphi (x-\eps z, y-\eps z)-U_\varphi (x, y)|^p\d y\d x\d z \\
&
+2 \|[u\varphi]*\eta_\eps -[u\varphi]\|^p_{L^p(\Omega')}\int_{|h|>\delta}\frac{\d h}{|h|^{d+sp}}\,\xrightarrow{\eps\to1^-}0.
\end{align*}
Thus we obtain  $\| [u\varphi]*\zeta_\eps-[u\varphi]\|_{ W^{s,p}(\Omega)} \to 0$ as $\eps \to 0$, and hence $u\varphi \in W^{s,p}_0(\Omega)$.
\end{proof}

\begin{lemma}[Zero-extension]\label{lem:zero-exten}
Let $\Omega\subset \R^d$ be an open set, and $u\in W^{s,p}(\Omega)$ with $s\in[0,1]$ and
$1\leq p\leq \infty$. Assume there exists a closed subset $K\subset\Omega$ such that $u\equiv 0$ in $\Omega\setminus K$, then the zero-extension function $\widetilde{u}$ defined as
\begin{align*}
\widetilde{u}(x)=
\begin{cases}
  u(x) \quad x\in \Omega \,,\\
\,0 \qquad x\in \R^d \setminus \Omega\,
\end{cases}
\end{align*}
belongs to $W^{s,p}(\R^d)$ and we have
\begin{align*}
\|\widetilde{u}\|^p_{W^{s,p}(\R^d)} \leq 2(1+ |\mathbb{S}^{d-1}|) \max(1,\delta^{-p}) \|u\|^p_{W^{s,p}(\Omega)}  .
\end{align*}
\end{lemma}

\begin{proof}
Let $\delta=\dist(K,\partial\Omega)>0$, so that $|x-y|\geq \delta>0$ for $x\in K$ and $y\in\R^d\setminus \Omega$. In other words $ \R^d\setminus \Omega\subset \R^d\setminus B_\delta(x)$ for every $x\in K$. Hence  for $x\in K$ we find that
\begin{align*}
\int_{\R^d\setminus \Omega} \hspace{-0.4ex}\frac{\d y }{|x-y|^{d+sp}}
\leq \int_{B^c_\delta (0)}\hspace{-0.4ex}\frac{\d h }{|h|^{d+sp}}
&\leq \max(1,\delta^{-p}) \int_{\R^d}\frac{1\land |h|^p }{|h|^{d+sp}} \d h=  \max(1,\delta^{-p}) \frac{|\mathbb{S}^{d-1}|}{s(1-s)p}.
\end{align*}
Since $\supp u\subset K\subset\Omega$  we can split the seminorm as follows
\begin{align*}
s(1-s)\int_{\R^d}\int_{\R^d} &\frac{|\widetilde{u}(x)-\widetilde{u}(y)|^p}{|x-y|^{d+sp}} \d y\d x =  s(1-s) \int_{\Omega}\int_{\Omega} \frac{|u(x)-u(y)|^p}{|x-y|^{d+sp}} \d y \d x  \\
&+ 2 s(1-s) \int_{K}  |u(x)|^p \left( \int_{\R^d \setminus \Omega} \frac{\d y}{|x-y|^{d+sp}}\right)\d x\\
&\leq  s(1-s) \int_{\Omega}\int_{\Omega} \frac{|u(x)-u(y)|^p}{|x-y|^{d+sp}} \d y \d x + 2|\mathbb{S}^{d-1}| \max(1,\delta^{-p})\|u\|^p_{L^p(\Omega )}.
\end{align*}
Beside this, it is readily seen that  $\|\widetilde{u}\|_{L^p(\R^d)} =  \|u\|_{L^p(\Omega)}$, while  the case $s=1$  is standard and $\|\nabla \widetilde{u}\|_{L^p(\R^d)} =  \|\nabla u\|_{L^p(\Omega)}$ if $u\in W^{1,p}(\Omega)$.   Therefore, we find that
\begin{align*}
\|\widetilde{u}\|^p_{W^{s,p}(\R^d)} \leq 2(1+ |\mathbb{S}^{d-1}|) \max(1,\delta^{-p}) \|u\|^p_{W^{s,p}(\Omega)},
\end{align*}
for $s\in [0,1]$ and $1\leq p< \infty$. The case $p=\infty$ and $s\in (0,1)$ is straightforward.
\end{proof}

\begin{lemma}[Change of variables]\label{lem:change-variables}
Let $T : U'\to U$ be a bi-Lipschitz mapping  with  $U\subset \R^d$ and $U' \subset \R^d$ be open sets. Given $u:U\to \R$ and $v:U'\to \R$ such that $v= u\circ  T$, i.e., $u= v\circ T^{-1}$ then $u \in W^{s,p}(U)$, if and only if $u\circ T\in W^{s,p}(U')$.  Moreover there exists  $C = C(d,p,T)> 0$, such that
\begin{align*}
\begin{split}
[u\circ T]_{W^{s,p}(U')}&\leq C[u]_{W^{s,p}(U)}\\
[v\circ T^{-1}]_{W^{s,p}(U)}&\leq C[u]_{W^{s,p}(U')}
\end{split}\quad\quad\text{for $s\in [0,1]$ \,\, $1\leq p\leq \infty$}.
\end{align*}
\end{lemma}

\begin{proof}
We only prove the first estimate as the second one is similar. Since $T : U' \to U$ and $T^{-1} : U \to U'$ are Lipschitz with constants $\mathrm{Lip}(T)$ and $\mathrm{Lip}(T^{-1})$, respectively, it follows from Rademacher's theorem (see \cite{Hei05}), that both $T$ and $T^{-1}$ are differentiable almost everywhere.Moreover , the Jacobian $DT$ and $DT^{-1}$ are bounded, with $\|DT\|_{L^\infty(U')} \leq \mathrm{Lip}(T)$ and $
\|DT^{-1}\|_{L^\infty(U)} \leq \mathrm{Lip}(T^{-1})
$. Let $M>1$ such that
\begin{align*}
\|\det D T\|_{L^\infty(U)}+ \|\det D T^{-1}\|_{L^\infty(U')}+ \mathrm{Lip}(T)+\mathrm{Lip}(T^{-1})\leq M.
\end{align*}
While the case $p=\infty$ is obvious, the case $s=1$ is also standard since (by mollifying $T$) one obtains the chain rule $\nabla (u\circ T)= (\nabla u)\circ T\cdot D T$; see  \cite[Theorem III.2.13.]{BF13}. Now we consider the case $1\leq p<\infty$ and $s\in [0,1)$. It is easy to find that
$\|u\circ T\|^p_{L^p(U')}
\leq M\|u\|^p_{L^p(U)}$.  Next,  we have $|x-y|^{d+sp}\leq M^{d+p}|T^{-1}(x)-T^{-1} (y)|^{d+sp}$ as since $T$ is Lipschitz.
Using the change of variables $x=T(\hat{x})$  and $y=T(\hat{y})$ implies
\begin{align*}
\int_{U'}\int_{U'}& \frac{|u(T(\hat{x}))-u(T(\hat{y}))|^p}{|\hat{x}-\hat{y}|^{d+sp}} \d\hat{y} \d\hat{x}  \\
 & = \int_{U}\int_{U} \frac{|u(x)-u(y)|^p}{|T^{-1}(x)-T^{-1}(y)|^{d+sp}}|\det DT^{-1}(x)| \d y |\det DT^{-1}(y)| \d x \\
 &  \leq M^{d+p+2} \int_{U}\int_{U} \frac{|u(x)-u(y)|^p}{|x-y|^{d+sp}}\d y \d x.
\end{align*}
That is $[u\circ T]_{W^{s,p}(U') } \leq C [u]_{W^{s,p}(U)} $ with $C=M^{\frac{d+2}{p}+1}$.
\end{proof}

\begin{theorem}[$W^{s,p}$-robust extension]\label{thm:robust-extension}
Assume $\Omega \subset \R^d$ is an open set whose boundary $\partial\Omega$
 is compact and Lipschitz. Then $\Omega$ is a robust  $W^{s,p}$-\emph{extension domain}, with $1\leq p \leq \infty$, in the following sense:

\smallskip
 \noindent For any open set $\Omega'\subset \R^d$  such that $\overline{\Omega}\subset \Omega'$, there exist a constant $C=C(d,p,\Omega, \Omega')>0$ and a linear operator $E: L^p(\Omega)\to L^p(\R^d) $
for which  the following properties hold.
\begin{itemize}
\item For every $u\in L^p(\Omega)$ we have, $Eu|_\Omega=u$  on $\Omega$ and $\supp Eu\subset \Omega'$.
\item  For  every $u\in W^{s,p}(\Omega)$ with $s\in [0,1]$,  we have $Eu\in W^{s,p}_0 (\Omega')$.
\item
For  each $s\in [0,1]$, the operator $E:  W^{s,p}(\Omega)\to  W^{s,p}(\R^d)$ is  bounded  and
\begin{align*}
\|Eu \|_{W^{s,p}(\R^d)}
\leq C\|u\|_{W^{s,p}(\Omega)},
\qquad\text{ $u\in W^{s,p}(\Omega)$}.
\end{align*}
\end{itemize}
\end{theorem}
\smallskip

\begin{proof}
\textbf{Step 1: Local flattening.} Since $\partial\Omega$ is compact Lipschitz and $\dist(\Omega, \partial \Omega')>0$, there exists a finite number of open sets $U_1,U_2,\cdots, U_k$ and a finite number of elements $x_1, x_2, \cdots , x_k\in \partial\Omega$ where $\overline{U}_j\subset \Omega' $ and  $x_j\in \partial\Omega\cap U_j$, such that $\partial{\Omega} \subset \bigcup_{j=1}^k U_j$.  Moreover, for each $j\in\{1,\cdots,k\}$ there is a bi-Lipschitz isomorphism $T_j:Q\to U_j$ with inverse $T^{-1}_j: U_j\to Q$
such that
\begin{align*}
T_j(Q_+)=U_j\cap\Omega, \qquad
T_j(Q_0)=U_j\cap\partial\Omega
\end{align*}
where $Q_0:=\{ x\in Q :
 x_d=0\}$ and $Q_+=Q\cap \R^d_+$ with
\begin{align*}
Q&:=\{ x=(x',x_d)\in\R^d : |x'|<1\;\mbox{and}\;|x_d|<1\}.
\end{align*}
\textbf{Step 2: Partition of unity.} Clearly, we can write $ \Omega'= \bigcup_{j=1}^k U_j \cup \big(\Omega' \setminus \partial{\Omega}  \big)$. Accordingly, by the partition of unity,  there exist $k+1$ smooth functions $\psi_0$, $\psi_1$,..., $\psi_k\in C_c^\infty(\R^d)$ such that for each $j\in \{0,... ,k\}$ we have $\supp \psi_0\subset \R^d\setminus \partial{\Omega}$, $\supp\psi_j\subset U_j$, $0\leq \psi_j \leq 1$ and $\sum_{j=0}^k \psi_j=1$. We obviously have $
u=\sum_{j=0}^k \psi_j u.$
\medskip

\textbf{Step 3: Pull-back.} Recall that $T_j: Q\to U_j$ is the isomorphism of class $C^{0,1}$ with $T_j(Q_+)=U_j\cap \Omega$, $j\in\{1,...,k\}$. Let us consider the pull-back of $u|_{U_j\cap\Omega }$ defined by
\begin{align*}
v_j(y):=u( T_j(y) ) \qquad \text{ for any } y\in Q_+.
\end{align*}
 According to Lemma \ref{lem:change-variables}, $v_j \in W^{s,p}\left( Q_+ \right)$ and we have
\begin{align}\label{eq:Tj-push-forward}
\begin{split}
[v_j]^p_{W^{s,p}(Q_+) }
&\leq C [u]^p_{W^{s,p}(U_j\cap\Omega)}
\qquad \text{ and }\qquad
\|v_j\|^p_{L^p(Q_+) }
\leq C \|u\|^p_{L^p(U_j\cap\Omega)}.
\end{split}
\end{align}
Here and in the sequel, it will be clear from the context that $C = C(d,p,\Omega)>0$  denotes a generic constant independent of $s$ and possibly different step by step.

\smallskip
\textbf{Step 4: Even reflection.} According to  Lemma~\ref{lem:even-reflec}, the even reflection $\overline{v}_j:Q\to \R$ with  $\overline{v}_j(x)=v_j(x',|x_d|)$  belongs to $W^{s,p}(Q)$ and satisfies
\begin{align*}
[\overline{v}_j]^p_{W^{s,p}(Q)}
&\leq 4 [v_j]^p_{W^{s,p}(Q_+)}\qquad \text{ and }\qquad
\|\overline{v}_j\|^p_{L^p(Q) }
= 2 \|v_j\|^p_{L^p(Q_+)}.
\end{align*}
\textbf{Step 5: Push-forward.} Next, after extending $v_j$ to $Q$ as $\overline{v}_j$ we  have to push the extension back to $\Omega$ by defining
\begin{align*}
w_j(x):=\overline{v}_j\big(T_j^{-1}(x)\big)
\quad \text{for all } x\in U_j .
\end{align*}
It is worth noting that  $w_j\equiv u$. Since $T_j$ is bi-Lipschitz, by arguing  as above using  Lemma \ref{lem:change-variables},  it follows that $w_j\in W^{s,p}(U_j)$ and
\begin{align} \label{eq:Tj-pull-back}
\begin{split}
[w_j]^p_{W^{s,p}(U_j)}
&\leq C [\overline{v}_j]^p_{W^{s,p}(Q)}\leq C [v_j]^p_{W^{s,p}(Q_+)}\leq C [u]^p_{W^{s,p}(U_j\cap \Omega)}\\
\|w_j\|^p_{L^p(U_j) }
&\leq C \|\overline{v}_j\|^p_{L^p(Q)} \leq C\|v_j\|^p_{L^p(Q_+) }
\leq C \|u\|^p_{L^p(U_j\cap\Omega)}.
\end{split}
\end{align}
\textbf{Step 6: Zero-extension.} By definition $\psi_0u$ and $\psi_jw_j$,  $j\in\{1,\cdots,k\}$ , have compact support in $U_j\subset \Omega'$ so that Lemma \ref{lem:cutoff} implies $\psi_0u,\psi_jw_j\in W^{s,p}_0(U_j)$. By Lemma~\ref{lem:zero-exten} their corresponding zero the extensions $\widetilde{\psi_0 u}$ and $\widetilde{\psi_jw_j}$ to all $\R^d$ belong to $W^{s,p}(\R^d)$. Furthermore, since $\psi_0u,\psi_jw_j\in W^{s,p}_0(U_j)$ and  $\overline{U}_j\subset \Omega'$ it is not difficult to verify that $\widetilde{\psi_0 u}, \widetilde{\psi_jw_j}\in W^{s,p}_0(\Omega')$.
In sum,  $\widetilde{\psi_0 u}, \widetilde{\psi_jw_j}\in W^{s,p}_0(\Omega')\cap W^{s,p}(\R^d)$ and using Lemma~\ref{lem:zero-exten}, Lemma~\ref{lem:cutoff} and the estimate \eqref{eq:Tj-pull-back} we get
\begin{align}\label{eq:jth-cutoff}
\|\widetilde{\psi_j\,w_j}\|_{W^{s,p}(\R^d)}
\leq  C \|\psi_j\,w_j\|_{W^{s,p}(U_j)}  &\leq C \|w_j\|_{W^{s,p}(U_j)}\leq C \|u\|_{W^{s,p}(\Omega\cap U_j)},
\\
\label{eq:first-cutoff}
\|\widetilde{\psi_0\,u}\|_{W^{s,p}(\R^d)}&\leq \,C\,\| \psi_0\,u\|_{W^{s,p}(\Omega)}\,\leq\, C\, \|u\|_{W^{s,p}(\Omega)}.
\end{align}
\textbf{Step 7: Extension.} Finally, we define the desired linear extension operator by
\begin{align*}
Eu=\widetilde{\psi_0u}
+ \sum_{j=1}^k \widetilde{\psi_j w_j}.
\end{align*}
It readily it follows that $\supp Eu\subset \Omega'$, since $\supp \psi_j\subset U_j\subset \Omega'$. Moreover, we have $Eu \in W^{s,p}_0(\Omega')\cap W^{s,p}(\R^d)$  since $\widetilde{\psi_0 u}, \widetilde{\psi_jw_j}\in W^{s,p}_0(\Omega')\cap W^{s,p}(\R^d)$.  Next, since  $w_j\equiv u$ (and hence $\psi_jw_j\equiv \psi_ju$) on $U_j\cap \Omega$,  we have
\begin{align*}
 Eu(x)=\widetilde{\psi_0u}(x)+ \sum_{j=1}^k \widetilde{\psi_j w_j}(x)= \psi_0u(x)+ \sum_{j=1}^k \psi_j w_j(x)= u(x)\quad\text{for $x\in \Omega$}.
\end{align*}
That is $Eu|_{\Omega}= u$. Moreover, combining \eqref{eq:jth-cutoff} with \eqref{eq:first-cutoff}, we get
\begin{align*}
\|Eu\|_{W^{s,p}(\R^d)}  \leq  C \|u\|_{W^{s,p}(\Omega)}
\end{align*}
with  $C=C(d,p,\Omega)$ independent of $s\in [0,1]$.
\end{proof}


\vspace{1mm}
\noindent \textbf{Data Availability Statement (DAS)}: Data sharing not applicable, no datasets were generated or analyzed during the current study.

\vspace{1mm}
\noindent {\small \textbf{Conflict of Interest:}
{The author declares that there is no conflict of interest regarding the publication of this paper.}


\end{document}